\documentclass{amsart}
\usepackage{amsmath,amssymb,amsthm,bm,bbm}
\usepackage{graphicx,enumerate}
\usepackage{layout}
\usepackage{longtable}
\usepackage[all,cmtip]{xy}

\theoremstyle{plain}
\newtheorem{theorem}{Theorem}[section]
\newtheorem*{theorem-nn}{Theorem}
\newtheorem{lemma}[theorem]{Lemma}

\newtheorem*{proposition-nn}{Proposition}
\newtheorem{corollary}[theorem]{Corollary}

\theoremstyle{definition}
\newtheorem{definition}[theorem]{Definition}
\newtheorem{defn}[theorem]{Definition}

\newtheorem{remark}[theorem]{Remark}

\newtheorem*{acknowledgment}{Acknowledgment}

\theoremstyle{remark}

\newcommand{\ta}{\tau}

\newcommand{\bZ}{\mathbbm{Z}}\newcommand{\bQ}{\mathbbm{Q}}
\newcommand{\bC}{\mathbbm{C}}\newcommand{\bG}{\mathbbm{G}}
\newcommand{\bF}{\mathbbm{F}}

\newcounter{sub}
{\begin{list}{(\arabic{sub})}{\usecounter{sub}%
\setlength{\leftmargin}{2em}}}{\end{list}}

\def\fn#1{\operatorname{#1}} 
\def\bm#1{\mathbbm{#1}}

\title{Multiplicative Invariant Fields of Dimension $\le 6$}

\author[A. Hoshi]{Akinari Hoshi}
\address{Department of Mathematics, Niigata University, Niigata 950-2181,
Japan}
\email{hoshi@math.sc.niigata-u.ac.jp}

\author[M. Kang]{Ming-chang Kang}
\address{Department of Mathematics, National Taiwan University, Taipei, Taiwan}
\email{kang@math.ntu.edu.tw}

\author[A. Yamasaki]{Aiichi Yamasaki}
\address{Department of Mathematics, Kyoto University, Kyoto 606-8502, Japan}
\email{aiichi.yamasaki@gmail.com}

\thanks{
This work was partially supported by JSPS KAKENHI Grant Numbers
24540019, 25400027, 16K05059.
Parts of the work were finished when the
first-named author and the third-named author were visiting the
National Center for Theoretic Sciences (Taipei), whose
support is gratefully acknowledged.}


\begin{document}
\begin{abstract}
The finite subgroups of $GL_4(\bm{Z})$ are classified up to conjugation in \cite{BBNWZ}; in particular, there exist $710$ non-conjugate finite groups in $GL_4(\bm{Z})$. Each finite group $G$ of $GL_4(\bm{Z})$ acts naturally on $\bm{Z}^{\oplus 4}$; thus we get a faithful $G$-lattice $M$ with ${\rm rank}_\bm{Z} M=4$. In this way, there are exactly $710$ such lattices. Given a $G$-lattice $M$ with ${\rm rank}_\bm{Z} M=4$,
the group $G$ acts on the rational function field $\bm{C}(M):=\bm{C}(x_1,x_2,x_3,x_4)$
by multiplicative actions, i.e. purely monomial automorphisms over $\bm{C}$. We are concerned with the rationality problem of the fixed field $\bm{C}(M)^G$. A tool of our investigation is the unramified Brauer group of the field $\bm{C}(M)^G$ over $\bm{C}$. It is known that, if the unramified Brauer group, denoted by  ${\rm Br}_u(\bm{C}(M)^G)$, is non-trivial, then
the fixed field $\bm{C}(M)^G$ is not rational (= purely transcendental) over $\bm{C}$. A formula of the unramified Brauer group ${\rm Br}_u(\bm{C}(M)^G)$ for the multiplicative invariant field was found by Saltman in 1990. However, to calculate ${\rm Br}_u(\bm{C}(M)^G)$ for a specific multiplicatively invariant field requires additional efforts, even when the lattice $M$ is of rank equal to $4$. There is a direct decomposition ${\rm Br}_u(\bm{C}(M)^G)= B_0(G) \oplus H^2_u(G,M)$ where $H^2_u(G,M)$ is some subgroup of $H^2(G,M)$. The first summand $B_0(G)$, which is related to the faithful linear representations of $G$, has been investigated by many authors. But the second summand $H^2_u(G,M)$ doesn't receive much attention except when the rank is $\le 3$. Theorem 1. Among the $710$ finite groups $G$, let $M$ be the associated faithful
$G$-lattice with ${\rm rank}_\bm{Z} M=4$,
there exist precisely $5$ lattices $M$ with ${\rm Br}_u(\bm{C}(M)^G)\neq 0$.
In these situations, $B_0(G)=0$ and thus 
${\rm Br}_u(\bm{C}(M)^G)\subset H^2(G,M)$. 
The $5$ groups are isomorphic to 
$D_4$, $Q_8$, $QD_8$, $SL_2(\bF_3)$, $GL_2(\bF_3)$ whose 
{\rm GAP IDs} 
are {\rm (4,12,4,12), (4,32,1,2), (4,32,3,2), (4,33,3,1), (4,33,6,1)} 
respectively 
in {\rm \cite{BBNWZ}} and in {\rm \cite{GAP}}.
Theorem 2. There exist $6079$ (resp. $85308$) finite subgroups $G$ 
in $GL_5(\bm{Z})$ (resp. $GL_6(\bm{Z})$). 
Let $M$ be the lattice with rank $5$ (resp. $6$) associated to each group $G$. 
Among these lattices precisely $46$ (resp. $1073$) 
of them satisfy the condition 
${\rm Br}_u(\bm{C}(M)^G)\neq 0$. The {\rm GAP IDs} (actually the 
{\rm CARAT IDs}) of the corresponding groups $G$ may be determined explicitly. 
Motivated by these results, we construct $G$-lattices $M$ of rank 
$2n+2$, $4n$, $p(p-1)$ ($n$ is any positive integer and 
$p$ is any odd prime number) satisfying that $B_0(G)=0$ 
and $H^2_u(G,M)\neq 0$; and therefore $\bm{C}(M)^G$ are not rational 
over $\bm{C}$. 
For these $G$-lattices $M$, we prove that the flabby class 
$[M]^{fl}$ of $M$ is not invertible. 
We also construct an example of $(C_2)^3$-lattice (resp. $A_6$-lattice) 
$M$ of rank $7$ (resp. $9$) with ${\rm Br}_u(\bm{C}(M)^G)\neq 0$. 
As a consequence, we give a counter-example to Noether's problem 
for $N\rtimes A_6$ over $\bC$ where $N$ is some abelian group. 
\end{abstract}

\maketitle

\tableofcontents

%
\section{Introduction}\label{seInt}

Let $k$ be a field, $G$ be a finite group and
$\rho : G\rightarrow GL(V)$ be a faithful representation of $G$ where $V$ is a finite-dimensional vector space over $k$.
Then $G$ acts on the rational function field $k(V)$. 
Noether's problem asks whether the fixed field $k(V)^G$ is rational 
(= purely transcendental) over $k$ \cite[Sa2]{Sw}.

In order to solve the rationality problem of $k(V)^G$, it is natural and almost inevitable that we reduce the problem to that of the multiplicative invariant field $k(M)^G$ defined in Definition \ref{d1.2}; an illustration of reducing Noether's problem to the multiplicative invariant field can be found, for example, in \cite{CHKK}. When $M$ is a $G$-lattice with ${\rm rank}_\bm{Z} M=n$, the multiplicative invariant field $k(M)^G$ is nothing but $k(x_1,\ldots, x_n)^G$, the fixed field of the rational function field $k(x_1,\ldots, x_n)$ on which $G$ acts by multiplicative actions. The purpose of this article is to study the multiplicative invariant fields $k(M)^G$ where $M$ is a $G$-lattice with ${\rm rank}_\bm{Z} M=n \le 6$. In Theorem \ref{thHaj87} and Theorem \ref{thHKHR} we will review the known results for the rationality problem of the multiplicative invariant fields up to dimension $3$; see Theorem \ref{th116} for some partial result of dimension $4$.

\begin{defn}\label{d1.1}
Let $G$ be a finite group and $\bm{Z}[G]$ be the group ring.
A finitely generated $\bm{Z}[G]$-module $M$ is called a {\it $G$-lattice} if,
as an abelian group, $M$ is a free abelian group of finite rank.
We will write ${\rm rank}_\bm{Z} M$ for the rank of $M$ as a free
abelian group.
A $G$-lattice $M$ is called {\it faithful} if, for any
$\sigma\in G\setminus\{1\}$, $\sigma\cdot x\neq x$ for some $x\in M$.

Suppose that $G$ is any finite group and $\Phi : G\rightarrow GL_n(\bm{Z})$ is a group homomorphism, i.e. an integral representation of $G$. Then the group $\Phi(G)$ acts naturally on the free abelian group $M:=\bm{Z}^{\oplus n}$; thus $M$ becomes a $\bm{Z}[G]$-module. We call $M$ the $G$-lattice associated to $\Phi$ (or $\Phi(G)$). Conversely, if $M$ is a $G$-lattice with ${\rm rank}_\bm{Z} M=n$, write $M=\oplus_{1\leq i\leq n}\bm{Z}\cdot x_i$.
Then there is a group homomorphism
$\Phi : G\rightarrow GL_n(\bm{Z})$ defined as follows: If
$\sigma\cdot x_i=\sum_{1\leq j\leq n}a_{ij}\,x_j$
where $\sigma\in G$ and $a_{ij}\in\bm{Z}$,
define $\Phi(\sigma)=(a_{ij})_{1\leq i,j\leq n}\in GL_n(\bm{Z})$.

When the group homomorphism $\Phi : G\rightarrow GL_n(\bm{Z})$ is injective, the corresponding $G$-lattice is a faithful $G$-lattice. For examples, any finite subgroup $G$ of $GL_n(\bm{Z})$ gives rise to
a faithful $G$-lattice of rank $n$.

The list of all the finite subgroups of $GL_n(\bm{Z})$ (with $n \le 4$), up to conjugation, can be found in the book \cite{BBNWZ} and in GAP. As to the situations of $GL_n(\bm{Z})$ (with $n \ge 5$), Plesken etc. found the lists of all the finite subgroups of $GL_n(\bm{Z})$ (with $n =5$ and $6$); see \cite{PS} and the references therein. These lists may be found in the GAP package
CARAT \cite{CARAT} and also in \cite[Chapter 3]{HY}.

Here is a list of the total number of lattices,
up to isomorphism, of a given rank:\\
\begin{center}
\begin{tabular}{c|r}
rank & number of the lattices\\\hline
$1$ & $2$\\
$2$ & $13$\\
$3$ & $73$\\
$4$ & $710$\\
$5$ & $6079$\\
$6$ & $85308$
\end{tabular}
~\\~\\
\end{center}
\end{defn}

\begin{defn}\label{d1.2}
Let $M$ be a $G$-lattice of rank $n$ and write
$M=\oplus_{1\leq i\leq n}\bm{Z}\cdot x_i$.
For any field $k$, define $k(M)=k(x_1,x_2,\ldots,x_n)$
the rational function field of $n$ variables over $k$.
Define a {\it multiplicative action} of $G$ on $k(M)$: For any $\sigma\in G$,
if $\sigma\cdot x_i=\sum_{1\leq j\leq n}a_{ij}\,x_j$ in the $G$-lattice $M$,
then we define $\sigma\cdot x_i=\prod_{1\leq j\leq n}x_j^{a_{ij}}$
in the field $k(M)$.
Note that $G$ acts trivially on $k$.
The above multiplicative action is called a {\it purely monomial action} of $G$
on $k(M)$ in \cite{HK1}; elements in the fixed field
$k(M)^G=\{u\in k(M) : \sigma\cdot u=u\ {\rm for\ any}\ \sigma\in G\}$
is called a {\it multiplicative field invariant} in \cite{Sa4}.

In case $M$ is the $G$-lattice $\bm{Z}[G]$ where $M=\oplus_{g \in G} \bm{Z} \cdot x_g$ and $h \cdot x_g=x_{hg}$ for $h, g \in G$, then $k(M)=k(x_g : g \in G)$. In this situation we simply write the fixed field $k(M)^G$ as $k(G)$. Note that $k(G)= k(V_{\rm reg})^G$ where $G \rightarrow GL(V_{\rm reg})$ is the regular representation of $G$ over $k$.

By the no-name lemma, it is known that $k(G)$ is stably rational over $k$ if and only if so is $k(V)^G$ where $\rho : G\rightarrow GL(V)$ is any faithful representation of $G$ over $k$ (see the proof of \cite[Proposition 2.2]{CHK}). 
Thus the rationality problem of $k(G)$ over $k$ is also called 
{\it Noether's problem} \cite{Sw}.
\end{defn}

\begin{defn}\label{d1.3}
Let $k$ be a field and $\mu$ be a multiplicative subgroup of $k \setminus \{0 \}$ containing all the
roots of unity in $k$.
If $M$ is a $G$-lattice, a {\it $\mu$-extension} is an exact sequence of
$\bm{Z}[G]$-modules given by
$(\alpha) : 1\rightarrow\mu\rightarrow M_\alpha\rightarrow M\rightarrow 0$
where $G$ acts trivially on $\mu$. Be aware that $M_{\alpha}=\mu \oplus M$ as abelian groups, but not as $\bm{Z}[G]$-modules except when the extension $(\alpha)$ splits.

As in Definition \ref{d1.2}, if $M=\oplus_{1\leq i\leq n}\bm{Z}\cdot x_i$ and $M_{\alpha}$ is a $\mu$-extension,
we define the field $k_\alpha(M)=k(x_1,\ldots,x_n)$ the rational function
field of $n$ variables over $k$; the action of $G$ on $k_\alpha(M)$ will be described in the next paragraph. Note that $M_{\alpha}$ is embedded into the
multiplicative group $k_\alpha(M) \setminus \{ 0 \}$ by sending $(\epsilon, \sum_{1 \le i \le n}b_ix_i)\in \mu \oplus M$ to the element $\epsilon \prod_{1 \le i \le n}x_i^{b_i}$ in the field $k_\alpha(M) =k(x_1,\ldots,x_n)$.

The group $G$ acts on $k_\alpha(M)$ by a twisted multiplicative action:
Suppose that, in $M$ we have $\sigma\cdot x_i=\sum_{1\leq j\leq n}a_{ij}\,x_j$, and
in $M_\alpha$ we have $\sigma\cdot x_i=\varepsilon_i(\sigma)+\sum_{1\leq j\leq n}a_{ij}\,x_j$ where $\varepsilon_i(\sigma)\in\mu$.
Then we define
$\sigma\cdot x_i=\varepsilon_i(\sigma)\prod_{1\leq j\leq n}x_j^{a_{ij}}$
in $k_\alpha(M)$. Again $G$ acts trivially on the coefficient field $k$.
The above group action is called a {\it monomial group action} in \cite{HK1};
the elements of $k_\alpha(M)^G$ are called {\it twisted multiplicative
field invariants} in \cite{Sa5}.

Note that, if the extension $(\alpha) : 1\rightarrow \mu\rightarrow M_\alpha
\rightarrow M\rightarrow 0$ is a split extension, then $k_\alpha(M)=k(M)$
and the twisted multiplicative action is reduced to the multiplicative
action in Definition \ref{d1.2}.
\end{defn}

Now return to the rationality problem.

First of all, recall some terminology.

Let $k$ be a field, and $L$ be a finitely generated field
extension of $k$. $L$ is called {\it $k$-rational} (or {\it rational over
$k$}) if $L$ is purely transcendental over $k$, i.e.\ $L$ is
isomorphic to some rational function field over $k$.
$L$ is called {\it stably $k$-rational} if $L(y_1,\ldots,y_m)$ is $k$-rational
for some $y_1,\ldots,y_m$ which are algebraically independent over $L$.
$L$ is called {\it $k$-unirational} if $L$ is $k$-isomorphic to a
subfield of some $k$-rational extension of $k$.
We recall another notion, retract rationality, due to Saltman \cite{Sa3}.

\begin{defn}\label{d1.9}
Let $k$ be an infinite field and $L$ be a field containing $k$.
$L$ is called {\it retract $k$-rational}, if $L$ is the quotient field of some affine
domain $A$ over $k$ and there exist $k$-algebra morphisms $\varphi: A\to
k[X_1,\ldots,X_m][1/f]$, $\psi:k[X_1,\ldots, X_m][1/f] \to A$
satisfying that $\psi\circ \varphi=1_A$, the identity map on $A$, where 
$k[X_1,\ldots,X_m]$ is a polynomial ring of $m$ variables over $k$, $f\in
k[X_1,\ldots,X_m]\backslash \{0\}$.
\end{defn}

It is not difficult to see that ``$k$-rational" $\Rightarrow$ ``stably $k$-rational" $\Rightarrow$ ``retract $k$-rational" $\Rightarrow$ ``$k$-unirational".

As mentioned before, in solving Noether's problem, no matter
what in the affirmative direction or in the negative direction,
it is crucial to consider the rationality problems of
$k(M)^G$ or $k_\alpha(M)^G$ where $M$ (resp. $M_\alpha$) is a $G$-lattice
(resp. a $\mu$-extension). In the following we list some previously known results along this line.

\begin{theorem}[{Hajja \cite{Ha}}]\label{thHaj87}
Let $k$ be a field and $G$ be a finite group acting on $k(x_1,x_2)$
by monomial $k$-automorphisms.
Then $k(x_1,x_2)^{G}$ is $k$-rational.
\end{theorem}

\begin{theorem}[{Hajja, Kang \cite{HK1,HK2}, Hoshi, Rikuna \cite{HR}}]
\label{thHKHR}
Let $k$ be a field and $G$ be a finite group acting on
$k(x_1,x_2,x_3)$ by purely monomial $k$-automorphisms.
Then $k(x_1,x_2,x_3)^G$ is $k$-rational.
\end{theorem}

\begin{theorem}[Hoshi, Kang, Kitayama {\cite[Theorem 1.16]{HKK}}]\label{th116}
Let $k$ be a field, $G$ be a finite group and $M$ be a $G$-lattice
with ${\rm rank}_{\bZ} M=4$ such that $G$ acts on $k(M)$ by
purely monomial $k$-automorphisms.
If $M$ is decomposable, i.e. $M=M_1\oplus M_2$ as $\bZ[G]$-modules where
$1\le {\rm rank}_{\bZ} M_1 \le 3$, then $k(M)^G$ is $k$-rational.
\end{theorem}

\begin{theorem}[Hoshi, Kang, Kitayama {\cite[Theorem 6.2]{HKK}}]\label{th62}
Let $k$ be a field, $G$ be a finite group and $M$ be a $G$-lattice
such that $G$ acts on $k(M)$ by purely monomial $k$-automorphisms.
Assume that
{\rm (i)} $M=M_1\oplus M_2$ as $\bZ[G]$-modules where
${\rm rank}_{\bZ}M_1=3$ and ${\rm rank}_{\bZ} M_2=2$,
{\rm (ii)} either $M_1$ or $M_2$ is a faithful $G$-lattice.
Then $k(M)^G$ is $k$-rational except the following situation:
{\rm char} $k\ne 2$, $G=\langle\sigma,\tau\rangle \simeq D_4$
and $M_1=\bigoplus_{1\le i\le 3}
\bZ x_i$, $M_2=\bigoplus_{1\le j\le 2} \bZ y_j$ such that
$\sigma:x_1\leftrightarrow x_2$, $x_3\mapsto -x_1-x_2-x_3$,
$y_1\mapsto y_2\mapsto -y_1$, $\tau: x_1\leftrightarrow x_3$,
$x_2\mapsto -x_1-x_2-x_3$, $y_1\leftrightarrow y_2$ where the
$\bZ[G]$-module structure of $M$ is written additively.
For the exceptional case, $k(M)^G$ is not retract $k$-rational.
\end{theorem}

The proofs of Theorem \ref{thHaj87} and Theorem \ref{thHKHR} are achieved by constructing the transcendence bases in a case by case fashion (there are no more than $13 + 73$ cases!). When we move to the rank $4$ lattices, we have $710$ lattices. This explains partially the reason why we consider only the decomposable lattices in Theorem
\ref{th116}. However, Theorem \ref{th62} tells us some pathological case may arise even for the decomposable lattices of rank $5$. How to understand the rationality problem of $k(M)^G$ for the situation ${\rm rank}_\bm{Z} M \ge 4$?

A useful obstruction to the rationality problem is the unramified Brauer group of an algebraic variety $V$ defined over $\bC$, denoted by $\fn{Br}_u(V)$. If a unirational smooth projective variety $V$ is rational, then $\fn{Br}_u(V)=0$. $\fn{Br}_u(V)$ is a birational invariant; it is the torsion subgroup of $H^3(V, \bZ)$ in the counter-example constructed by M. Artin and Mumford \cite{AM}. The unramified Brauer group was recast by Saltman in an algebraic form \cite{Sa2}; thus we may define $\fn{Br}_u(K)$ for any field extension $K$ over $\bC$.  Saltman and Bogomolov proposed very power methods to computing $\fn{Br}_u(K)$ \cite[Bo]{Sa5}. In this article, we focus on the computing of $\fn{Br}_u(\bm{C}(M)^G)$ for any multiplicative invariant field $\bm{C}(M)^G$ where $M$ is a lattice of rank $\le 6$. If $\bm{C}(M)^G$ is retract rational, it is necessary that $\fn{Br}_u(\bm{C}(M)^G)=0$. Equivalently, the non-vanishing of $\fn{Br}_u(\bm{C}(M)^G)$ implies that the field $\bm{C}(M)^G$ is not retract rational; in particular, it is not rational over $\bC$. We will emphasize that the vanishing of $\fn{Br}_u(\bm{C}(M)^G)$ is a necessary condition, but not a sufficient condition, of retract rationality; it may happen that $\fn{Br}_u(\bm{C}(M)^G)=0$ while $\bm{C}(M)^G$ is not retract rational (see \cite[HKY]{Pe}).

In case ${\rm rank}_\bm{Z} M \le 3$, $\fn{Br}_u(\bm{C}(M)^G)=0$ for all lattices $M$ because $\bm{C}(M)^G$ are always rational (see Theorem \ref{thHaj87} and Theorem \ref{thHKHR}). We will classify all the lattices $M$ with $\fn{Br}_u(\bm{C}(M)^G)\neq 0$ when ${\rm rank}_\bm{Z} M =4, 5$ and $6$. Thus $\bm{C}(M)^G$ are not retract rational for these lattices (and thus are not rational). In this way we obtain a plenty of ``counter-examples" to the rationality problem for lattices of rank $\ge 4$.

On the other hand, Theorem \ref{t7.5} provides a sufficient condition for  $\bm{C}(M)^G$ to be retract $\bm{C}$-rational. In this way, among lattices $M$ with rank $= 4, 5$ and $6$, we find many lattices $M$ (but not all of them) such that $\bm{C}(M)^G$ are retract $\bm{C}$-rational.

\medskip
\begin{defn}\label{d1.12}
Let $k\subset K$ be a finitely generated extension of fields.
The notion of the {\it unramified Brauer group} of $K$ over $k$ was introduced by Saltman \cite{Sa2}; we denote it by $\fn{Br}_{u,k}(K)$, or by $\fn{Br}_u(K)$ if the base field $k$ is understood from the context.

By definition, $\fn{Br}_{u,k}(K)=\bigcap_R \fn{Image} \{ \fn{Br}(R)\to
\fn{Br}(K)\}$ where $\fn{Br}(R)\to \fn{Br}(K)$ is the natural morphism of
Brauer groups and $R$ runs over all the discrete valuation rings $R$
such that $k\subset R\subset K$ and $K$ is the quotient field of
$R$.

If $k$ is an infinite field and $K$ is retract $k$-rational, it can be shown that the natural map $\fn{Br}(k)\to \fn{Br}_{u,k} (K)$ is an isomorphism \cite[Proposition 2.2]{Sa4}.
In particular, if $k$ is an algebraically closed field (for examples, $k= \bm{C}$) and $K$ is
retract $k$-rational, then $\fn{Br}_{u,k}(K)=0$. Consequently, if $k$ is algebraically closed and ${\rm Br}_{u,k}(K)$ is non-trivial, then $K$ is not retract $k$-rational (and thus it is not stably $k$-rational, in particular).

In Definition \ref{d1.4}, we will see that ${\rm Br}_u(\bm{C}(M)^G) \simeq B_0(G) \oplus H^2_u(G,M)$ where $B_0(G)$ is the so-called Bogomolov multiplier and $H^2_u(G,M)$ is some subgroup of the group cohomology $H^2(G,M)$. We remark that $B_0(G)$ is related to the rationality of $\bC(V)^G$ where $G \to GL(V)$ is any faithful linear representation of $G$ over $\bC$; on the other hand, $H^2_u(G,M)$ arises from the multiplicative nature of the field $\bm{C}(M)^G$.
\end{defn}

\bigskip
Here is one of the main results of this article, which tells exactly the lattices $M$ with rank $\le 6$ for which the unramified Brauer group of $\bC(M)^G$ is non-trivial. For those $M$ with trivial ${\rm Br}_u(\bm{C}(M)^G)$ we do not know whether $\bC(M)^G$ are $\bC$-rational or not.

\begin{theorem}\label{t1.5}
Let $G$ be a finite group and $M$ be a faithful $G$-lattice.
\\
{\rm (1)} If ${\rm rank}_\bm{Z} M \le 3$, then ${\rm Br}_u(\bm{C}(M)^G)=0$.\\
{\rm (2)} If ${\rm rank}_\bm{Z} M =4$, then ${\rm Br}_u(\bm{C}(M)^G)\neq 0$
if and only if $M$ is one of the $5$ cases in {\rm Table} $1$. 
Moreover, if $M$ is one of the $5$ $G$-lattices with 
${\rm Br}_u(\bm{C}(M)^G)\neq 0$, then $B_0(G)=0$ and ${\rm Br}_u(\bm{C}(M)^G)=H_u^2(G,M)$.\\
{\rm (3)} If ${\rm rank}_\bm{Z} M =5$, then ${\rm Br}_u(\bm{C}(M)^G)\neq 0$ if and only if
$M$ is one of the $46$ cases in {\rm Table} $2$ of {\rm Section} $\ref{seTables}$.
Moreover, if $M$ is one of the $46$ $G$-lattices with 
${\rm Br}_u(\bm{C}(M)^G)\neq 0$, then $B_0(G)=0$ and ${\rm Br}_u(\bm{C}(M)^G)=H_u^2(G,M)$. \\
{\rm (4)} If ${\rm rank}_\bm{Z} M =6$, then
${\rm Br}_u(\bm{C}(M)^G)\neq 0$ if and only if
$M$ is one of the $1073$ cases as in {\rm Table} $3$ of {\rm Section} $\ref{seTables}$.
Moreover, if $M$ is one of the 
$1073$ $G$-lattices with 
${\rm Br}_u(\bm{C}(M)^G)\neq 0$, then $B_0(G)=0$ and ${\rm Br}_u(\bm{C}(M)^G)=H_u^2(G,M)$,
except for $24$ cases
with $B_0(G)=\bm{Z}/2\bm{Z}$ where the {\rm CARAT ID} of $G$ are
$(6,6458,i)$, $(6,6459,i)$, $(6,6464,i)$ $(1\leq i\leq 8)$.
Note that $22$ cases out of the exceptional $24$ cases satisfy $H_u^2(G,M)=0$.
$($See {\rm  Definition} $\ref{d1.4}$ for the definitions of $B_0(G)=H_u^2(G,\bm{Q}/\bm{Z})$ and $H_u^2(G,M)$.$)$
\end{theorem}~\vspace*{-4mm}\\

Table $1$: $M$ is indecomposable of rank $4$ ($5$ cases with $H_u^2(G,\bm{Q}/\bm{Z})=0$)\\

\begin{tabular}{lllll}
\hline
$G(n,i)$ & $G$ & GAP ID & $H_u^2(G,M)$\\\hline
$(8,3)$ & $D_4$ & $(4,12,4,12)$ & $\bZ/2\bZ$\\
$(8,4)$ & $Q_8$ & $(4,32,1,2)$ & $(\bZ/2\bZ)^{\oplus 2}$\\
$(16,8)$ & $QD_8$ & $(4,32,3,2)$ & $\bZ/2\bZ$\\
$(24,3)$ & $SL_2(\bm{F}_3)$ & $(4,33,3,1)$ & $(\bZ/2\bZ)^{\oplus 2}$\\
$(48,29)$ & $GL_2(\bm{F}_3)$ & $(4,33,6,1)$ & $\bZ/2\bZ$\\
\end{tabular}\\

\begin{remark}
(1) 
Theorem \ref{t1.5} remains valid if we replace the coefficient field $\bm{C}$ by any algebraically closed field $k$ with {\rm char} $k=0$.\\
(2) If $M$ is of rank $\leq 6$ and ${\rm Br}_u(\bm{C}(M^G))\neq 0$,
then $G$ is solvable and non-abelian, and
${\rm Br}_u(\bm{C}(M)^G)\simeq \bZ/2\bZ$, $\bZ/3\bZ$ or
$\bZ/2\bZ\oplus\bZ/2\bZ$.
Namely, if $G$ is abelian or non-solvable with
${\rm Br}_u(\bC(M)^G)\neq 0$, then $n\geq 7$
(for explicit examples, see {\rm Section} $\ref{seAbelian}$ and {\rm Section} $\ref{seA6}$).
The case where ${\rm Br}_u(\bm{C}(M)^G)\simeq\bZ/3\bZ$ occurs only for
$4$ groups $G$ of order $27$, $27$, $54$, $54$
with the {\rm CARAT ID}
$(6,2865,1)$, $(6,2865,3)$, $(6,2899,3)$, $(6,2899,5)$
which are isomorphic to $C_9\rtimes C_3$, $C_9\rtimes C_3$,
$(C_9\rtimes C_3)\rtimes C_2$, $(C_9\rtimes C_3)\rtimes C_2$ respectively.
See {\rm Section} $\ref{seTables}$ for details.
For {\rm CARAT ID}, see {\rm Section} $\ref{seCarat}$. 
\\
(3)
The group $G \, (\simeq D_4)$ which appears as the exceptional case in Theorem \ref{th62} (i.e. \cite[Theorem 6.2]{HKK}) satisfies the property that ${\rm Br}_u(\bm{C}(M)^G)= H^2 _u (G,M)\neq 0$ where $M$ is the associated lattice. It follows that $\bC(M)^G$ is not retract $\bm{C}$-rational.
In fact, this group is the group with {\rm CARAT ID} $(5,100,11)$ which is the unique group in Table $2$-$3$.

In Theorem \ref{th62}, note that both $\bC(M_1)^G$ and $\bC(M_2)^G$ are $\bm{C}$-rational by Theorem \ref{thHKHR} and Theorem \ref{thHaj87}. 
Thus
${\rm Br}_u(\bm{C}(M_2)^G)=0$ and 
$H^2 _u (G,M_2)=0$. But $M_1$ is not a faithful $G$-lattice and we cannot apply Theorem \ref{thSa4} to $\bC(M_1)^G$. Hence it is possible that $H^2 _u (G,M_1)$ is non-trivial. Because $H^2 _u (G,M) \simeq H^2 _u (G,M_1)\oplus H^2 _u (G,M_2)$, this allows for the possibility that $H^2 _u (G,M)$ is non-trivial. Indeed, by the same method as in {\rm Section} $\ref{seAbelian}$ and {\rm Section} $\ref{seA6}$, it can be shown that $H^2 _u (G,M_1) \simeq \bZ/2 \bZ$ and therefore ${\rm Br}_u(\bm{C}(M)^G)= H^2 _u (G,M_1)\simeq \bZ/2 \bZ$.

We also note that the fixed field $\bm{C}(M)^G$ is stably isomorphic to $\bm{C}(M')^{D_4}$ for some indecomposable $D_4$-lattice $M'$ in Theorem \ref{t1.5} (2) (see Theorem \ref{equiv} for details). \\
(4) In the decomposable cases $M=M_1\oplus M_2$, the lattices in Table $2$-$2$ (resp. $3$-$2$, $3$-$3$-$2$, $3$-$4$, $3$-$5$, $3$-$6$) of {\rm Section} $\ref{seTables}$ satisfy the property that $H^2 _u(G, M_1)\neq 0$, but $M_1$ is not a faithful $G$-lattice in some situations; thus the non-triviality of ${\rm Br}_u(\bC(M)^G)=H^2 _u(G, M)$ is automatic. However, for the lattices in Table $2$-$3$ (resp. $3$-$3$-$1$), we have ${\rm Br}_u(\bm{C}(M_i)^G)=0$ for $(i=1,2)$. An explanation of this peculiar phenomenon can be found in the remark (3) above. \\
{\rm (5)} Here is a summary of Theorem \ref{t1.5}:\\
\begin{center}
\begin{tabular}{c|r|r}
rank & number of the lattices & number of lattices with non-trivial\\
& & unramified Brauer groups\\\hline
$1$ & $2$ & $0$\\
$2$ & $13$ & $0$\\
$3$ & $73$ & $0$\\
$4$ & $710$ & $5$\\
$5$ & $6079$ & $46$\\
$6$ & $85308$ & $1073$
\end{tabular}
~\\~\\
\end{center}


\end{remark}

The idea of the proof of Theorem \ref{t1.5} is to write an algorithm {\tt H2nrM(g)} for computing $H_u^2(G,M)$ (see Definition \ref{d1.4} and Section \ref{seGAP} for the algorithm). In devising this algorithm, some known GAP functions do not fit our purpose. Thus we are forced to modify these functions. For examples, we write new algorithms for computing the $2$-cocycles and the restriction maps of cohomology groups. With this algorithm in arms, GAP computation will tell us the unramified Brauer groups. All of the computation of this paper may be carried out in a personal computer. As it is common in most questions involving computations, we should employ some theoretic arguments to minimize the computing time and the memory of the computer; thus the traditional methods are inevitable (see the proof of Theorem \ref{thm3}). We remark that the main problem of Theorem \ref{t1.5} is to determine $H_u^2(G,M)$, because there have been algorithms for the computation of $B_0(G)$ (see Theorem \ref{t2.12} and Definition \ref{d1.4} for $H_u^2(G,M)$ and $B_0(G) \simeq H_u^2(G,\bQ/\bZ)$).

 Motivated by the $G$-lattices in Theorem \ref{t1.5}, in Section \ref{seHigher} we will construct $G$-lattices $M$ of rank $2n+2$, $4n$ and $p(p-1)$ ($n$ is any positive integer and $p$ is any odd prime number) such that the unramified Brauer groups of $\bC(M)^G$ are non-trivial. The proof of this result is the traditional theoretic method, because $n$ and $p$ are general integers that the computer computing cannot be of help. The key tool of the proof is the 7-term exact sequence associated to the Hochschild-Serre spectral sequence.

\bigskip
We organize this paper as follows. We recall some preliminaries in Section \ref{sePre}. In particular, Saltman's formula of the unramified Brauer group of $\bC(M)^G$ in terms of certain subgroups of the cohomology groups $H^2(G,M)$ and $H^2(G,\bm{Q}/\bm{Z})$ is given. The proof of Theorem \ref{t1.5} is given in Section \ref{sePT}, while the important algorithm {\tt H2nrM(g)} for computing $H_u^2(G,M)$ is contained in Section \ref{seGAP}.

\bigskip
For computing $H_u^2(G,M)$, we rely on
Saltman's formula (see Theorem \ref{thSa4}):
$H_u^2(G,M)\simeq \bigcap_{A} {\rm Ker}({\rm res} : H^2(G,M)
\rightarrow H^2(A,M)$ where $A$ runs over bicyclic subgroups of $G$.
Hence the main part of the algorithm {\tt H2nrM(g)}
is to computing the restriction map
${\rm res} : H^2(G,M)\rightarrow H^2(A,M)$.
The function {\tt ResH2(g,h)} serves to compute the restriction map
${\rm res} : H^2(G,M)\rightarrow H^2(H,M)$ where $H$ is any subgroup of $G$. This function is available from\\
{\tt https://www.math.kyoto-u.ac.jp/\~{}yamasaki/Algorithm/MultInvField/res.gap}

\bigskip

We remark that, in Section \ref{sePT}, the proofs of Theorems \ref{thm1}, \ref{thm2}, \ref{thm3} are preceded by some theoretic explanation which serves to justify the computer algorithm following the discussion; thus the complete proof consists of the discussion and the algorithm (together with the computer computation). The final results are contained in the tables of Section \ref{seTables}.

The lattices of general ranks modeled on some lattices in Theorem \ref{t1.5} are constructed in Section \ref{seHigher}. Section \ref{seC}, Section \ref{seAbelian} and Section \ref{seA6} arise from a second thought of Theorem \ref{th62}. In fact, Barge's Theorem (Theorem \ref{thB1}) guarantees that there is a $G$-lattice $M$ with ${\rm Br}_u(\bm{C}(M)^G) \neq 0$ where $G \simeq (C_2)^3$, although the rank of this ``bad" lattice is not specified. However, from the tables in Section \ref{seTables}, we cannot find $G$-lattices with non-trivial unramified Brauer groups with $G \simeq (C_2)^3$. This prompts us to conclude that these ``bad" lattices are of rank $\ge 7$. But no database for $G$-lattices of rank $7$ was available when we wrote this article. Thus we construct these lattices ourselves in Section \ref{seC} and compute their unramified Brauer groups in Theorem \ref{t5.4}. So far all the ``bad" $G$-lattices are related to solvable groups $G$. In Section \ref{seA6}, we construct $G$-lattices $M$ with ${\rm Br}_u(\bm{C}(M)^G) \neq 0$ and $G \simeq A_6$, the alternating group; such lattices are of rank $9$ (see Theorem \ref{t7.1}). Remember Kunyavskii's Theorem that $B_0(G)=0$ for any non-abelian simple group $G$ \cite{Ku}. Thus the multiplicative nature of $\bm{C}(M)^G$ will add a new component to its unramified Brauer group, a phenomenon which we have encountered in Theorem \ref{t1.5}, Theorem \ref{t5.4} and Theorem \ref{t7.1}. As a consequence, we also 
give a counter-example to Noether's problem 
for $N\rtimes A_6$ over $\bC$ where $N$ is some abelian group as in Section \ref{seA6}. 

\medskip
Notation and terminology: Recall the definition of $k(G)$ (and thus $\bC(G)$ in Definition \ref{d1.2}). The groups $C_n$ and $D_n$ refer to the cyclic group of order $n$ and the dihedral group of order $2n$ respectively. If $n$ is any positive integer, the quasi-dihedral group of order $16n$ is denoted by $QD_{8n}$. If $n$ is any positive integer, the generalized quaternion group of order $8n$ is denoted by $Q_{8n}$. For a linear map $\sigma$, we will write its matrix with respect to a basis $x_i, \ldots, x_n$ in the row form, i.e. if $\sigma \cdot x_i = \sum_{1 \le j \le n} a_{ij} x_j$, then the matrix of $\sigma$ is $(a_{ij})_{1 \le i, j \le n}$. If $M$ is a faithful $G$-lattice and $M$ decomposes into $M_1 \oplus M_2$, it may happen that $M_i$ is not a faithful $G$-lattice. We will denote by $G_i =G |_{M_i}$ the image of the group homomorphism $G \to {\rm Aut}(M_i)$; this notation is used in the tables of Section \ref{seTables}. The reader can find the definitions of the $\bm{Z}$-class and the $\bm{Q}$-class of $GL_n(\bm{Z})$ at the third paragraph of Section \ref{seC}. 

\begin{acknowledgment}
The authors would like to thank the referee who 
gave them useful comments and suggestions. 
\end{acknowledgment}

\section{Preliminaries and the unramified Brauer groups}\label{sePre}

In this section we review some preliminaries which will be used in the sequel, in particular, the powerful theorems of Bogomolov and Saltman for computing the unramified Brauer group.

Recall that a group $A$ is called bicyclic if $A$ is either a cyclic group or a direct
product of two cyclic groups.

\begin{theorem}[{Bogomolov \cite{Bo}, Saltman \cite[Theorem 12]{Sa5}}]
\label{t2.12}
Let $G$ be a finite group and $k$ be an algebraically closed field with
$\fn{char}k=0$. Let $\bm{Q}/\bm{Z}$ denote the $\bm{Z}[G]$-module with the trivial $G$-action $($i.e. $\bm{Q}/\bm{Z}$ is isomorphic to the multiplicative
subgroup of all roots of unity in $k$$)$. Then $\fn{Br}_{u,k}(k(G))$ is
isomorphic to the group $B_0(G)$ defined by
\[
B_0(G)=\bigcap_A \fn{Ker}(\fn{res}: H^2(G,\bm{Q}/\bm{Z})\to H^2(A,\bm{Q}/\bm{Z}))
\]
where $A$ runs over all the bicyclic subgroups of $G$.
\end{theorem}

If $G \to GL(V)$ is any faithful linear representation of $G$ over $\bC$, then $\fn{Br}_{u,\bC}(\bC(V)^G) \simeq B_0(G)$ by the No-Name Lemma \cite{Sa5}. The group $B_0(G)$ is a subgroup of $H^2(G,\bm{Q}/\bm{Z})$ (the Schur multiplier); thus it is called {\it the Bogomolov multiplier} of $G$ in \cite{Ku}. Note that, the formula in \cite[Theorem 12]{Sa5} can be used to compute not
only $\fn{Br}_{u,\bm{C}}(\bm{C}(V)^G)$, but also $\fn{Br}_{u,\bm{C}}(\bm{C}_\alpha(M)^G)$ where $\bm{C}_\alpha(M)$ is the rational function field associated to the $\mu$-extension $M_{\alpha}$ (see Definition \ref{d1.3}). We record it as follows.

\begin{theorem}[Saltman {\cite[Theorem 12]{Sa5}}]\label{thSa4}
Let $k$ be an algebraically closed field with {\rm char} $k=0$, and $G$ be a finite group. If $M$ is a $G$-lattice and $(\alpha): 1 \to \mu \to M_{\alpha} \to M \to 0$ is a $\mu$-extension such that {\rm (i)} $M$ is a faithful $G$-lattice, and {\rm (ii)} $H^2(G,\mu) \to H^2(G,M_\alpha)$ is injective, then 
\[
{\rm Br}_{u,k}(k_\alpha(M)^G)=\bigcap_{A} {\rm Ker}(
{\rm res} : H^2(G,M_\alpha)\rightarrow H^2(A,M_\alpha))
\]
where $A$ runs over all the bicyclic subgroups of $G$.

In particular, if the $\mu$-extension $(\alpha) : 1\rightarrow\mu\rightarrow M_\alpha\rightarrow M\rightarrow 0$ splits, then ${\rm Br}_{u,k}(k(M)^G)\simeq B_0(G) \oplus \bigcap_{A} {\rm Ker}(
{\rm res} : H^2(G,M)\rightarrow H^2(A,M))$ where $A$ runs over bicyclic subgroups of $G$.
\end{theorem}

\begin{defn}\label{d1.4}
{}From Definition \ref{d1.12}, $\fn{Br}_{u,k}(K)$ is a subgroup of the Brauer group $\fn{Br}(K)$. On the other hand, from the field extension $k_\alpha(M)^G \subset k_\alpha(M)$, the map of the Brauer groups ${\rm Br}(k_\alpha(M)^G) \to {\rm Br}(k_\alpha(M))$ sends ${\rm Br}_{u,k}(k_\alpha(M)^G)$ to ${\rm Br}_{u,k}(k_\alpha(M))$ \cite[Theorem 2.1]{Sa4}. Since ${\rm Br}_{u,k}(k_\alpha(M))=0$ by \cite[Proposition, 2.2]{Sa4}, it follows that the unramified Brauer group ${\rm Br}_{u,k}(k_\alpha(M)^G)$ is a subgroup of the relative Brauer group ${\rm Br}(k_\alpha(M)/k_\alpha(M)^G)$. As ${\rm Br}(k_\alpha(M)/k_\alpha(M)^G)$ is isomorphic to the cohomology group $H^2(G,k_\alpha(M)^{\times})$, we may regard ${\rm Br}_{u,k}(k_\alpha(M)^G)$ as a subgroup of $H^2(G,k_\alpha(M)^{\times})$.

Through the embedding of $M_\alpha$ into $k_\alpha(M)^\times$,
there is a canonical injection of $H^2(G,M_\alpha)$ into ${\rm Br}(k_\alpha(M)^G)$ \cite[page 536]{Sa5}. Identifying both ${\rm Br}_{u,k}(k_\alpha(M)^G)$ and $H^2(G,M_\alpha)$ as subgroups of $H^2(G,k_\alpha(M)^{\times})$, it can be shown that ${\rm Br}_{u,k}(k_\alpha(M)^G)$ is a subgroup of $H^2(G,M_\alpha)$ \cite[page 536]{Sa5}. Thus we write $H^2_u(G,M_\alpha)$ for ${\rm Br}_{u,k}(k_\alpha(M)^G)$ (see \cite{Sa5})  .

Note that there is a natural map $H^2(G,\bm{Q}/\bm{Z}) \to H^2(G,M_\alpha)$.
Clearly this map is injective if the $\mu$-extension $(\alpha) : 1\rightarrow\mu\rightarrow M_\alpha\rightarrow M\rightarrow 0$ splits. In this case, regarding $H^2(G,\bm{Q}/\bm{Z})$ and $H^2(G,M)$ as subgroups of $H^2(G,M_\alpha)$, we define $H^2_u(G,\bm{Q}/\bm{Z})=H^2(G,\bm{Q}/\bm{Z}) \cap
{\rm Br}_u(k_\alpha(M)^G)$ and $H^2_u(G,M)=H^2(G,M) \cap
{\rm Br}_u(k_\alpha(M)^G)$.
It follows that
${\rm Br}_u(k_\alpha(M)^G)=H_u^2(G,\bQ/\bZ)\oplus H_u^2(G,M)$.
By Theorems \ref{t2.12} and \ref{thSa4},
we have $H^2_u(G,\bQ/\bZ)\simeq B_0(G)$ and
$H^2_u(G,M)\simeq \bigcap_{A} {\rm Ker}(
{\rm res} : H^2(G,M)\rightarrow H^2(A,M))$ where $A$ runs over
bicyclic subgroups of $G$.
\end{defn}

\begin{theorem}[Huebschmann \cite{Hu}, Dekimpe, Hartl, Wauters \cite{DHW}]\label{thDHW}
Let $G$ be a finite group and $N$ be a normal subgroup of $G$.
Then the Hochschild--Serre spectral sequence gives rise to the following
$7$-term exact sequence
\begin{align*}
0 &\to H^1(G/N,M^N) \to H^1(G,M)\to H^1(N,M)^{G/N} \to H^2(G/N,M^N) \\
&\to H^2(G,M)_1
\to H^1(G/N,H^1(N,M))
\to H^3(G/N,M^N)
\end{align*}
where $H^2(G,M)_1={\rm Ker}({\rm res}:H^2(G,M)\to H^2(N,M))$.
In particular, if $M^N=0$, then
we have two isomorphisms $H^1(G,M)\simeq H^1(N,M)^{G/N}$ and
$H^2(G,M)_1\simeq H^1(G/N,H^1(N,M))$.
\end{theorem}

\begin{theorem}[{Fischer \cite{Fi}, see also Swan \cite[Theorem 6.1]{Sw}}]\label{thFi}
Let $G$ be a finite abelian group with exponent $e$.
Assume that {\rm (i)} either {\rm char} $k=0$ or {\rm char} $k>0$ with
{\rm char} $k$ ${\not |}$ $e$, and
{\rm (ii)} $k$ contains a primitive $e$-th root of unity.
Then $k(G)$ is $k$-rational.
In particular, $\bm{C}(G)$ is $\bm{C}$-rational.
\end{theorem}

\begin{theorem}[{Chu and Kang \cite[Theorem 1.6]{CK}}] \label{thCK}
Let $G$ be a $p$-group of order $\leq p^4$.
If $k$ is a field satisfying {\rm (i)} $\fn{char}k=p>0$, or
{\rm (ii)} $\fn{char}k\ne p$ with $\zeta_e\in k$ where
$e$ is the exponent of the group $G$ and $\zeta_e$ is a primitive $e$-th root of unity, then $k(G)$ is $k$-rational.
\end{theorem}

\begin{theorem}[{Chu, Hu, Kang and Prokhorov \cite[Theorem 1.5]{CHKP}}]
\label{thCHKP}
Let $G$ be a group of order $32$ with exponent $e$.
If $k$ is a field satisfying {\rm (i)} $\fn{char} k=2$,
or {\rm (ii)} $\fn{char}k\ne 2$ with $\zeta_e \in k$ $($where $\zeta_e$ is a primitive $e$-th root of unity$)$,
then $k(G)$ is $k$-rational.
\end{theorem}

\begin{theorem}[{Barge \cite[Theorem II.7]{Ba1}}]\label{thB1}
Let $G$ be a finite group.
The following two statements are equivalent:\\
{\rm (i)} all the Sylow subgroups of $G$ are bicyclic;\\
{\rm (ii)} ${\rm Br}_u(\bm{C}(M)^G)=0$ for all $G$-lattices $M$.
\end{theorem}

\begin{theorem}[Barge {\cite[Theorem IV-1]{Ba2}}] \label{thB2}
Let $G$ be a finite group.
The following two statements are equivalent:\\
{\rm (i)} all the Sylow subgroups of $G$ are cyclic;\\
{\rm (ii)} ${\rm Br}_u(\bm{C}_{\alpha}(M)^G)=0$ for all $G$-lattices $M$,
for all short exact sequences of
$\bm{Z}[G]$-modules $\alpha : 0 \rightarrow \bm{C}^{\times}
\rightarrow M_{\alpha} \rightarrow M \rightarrow 0$.
\end{theorem}




\section{CARAT ID of the $\bZ$-classes in dimensions $5$ and $6$}\label{seCarat}
We explain how to access the GAP ID 
and the CARAT ID of a finite subgroup $G$ of $GL_n(\bZ)$ $(n\leq 6)$. 
We need the GAP (\cite{GAP}) packages CrystCat and CARAT to do 
the computations below.

The CrystCat package of GAP provides a catalog of $\bQ$-classes 
and $\bZ$-classes (conjugacy classes) of finite subgroups $G$ 
of $GL_n(\bQ)$ and $GL_n(\bZ)$ $(2\leq n\leq 4)$. 
For $2\leq n\leq 4$, the GAP ID $(n,i,j,k)$ of a finite subgroup 
$G$ of $GL_n(\bZ)$ means that $G$ belongs to 
the $k$-th $\bZ$-class of the $j$-th $\bQ$-class of 
the $i$-th crystal system of dimension $n$ in GAP 
(see also \cite[Table 1]{BBNWZ}).

The CARAT\footnote{CARAT works on Linux or macOS, but not on Windows.} 
(\cite{CARAT}) package of GAP provides all conjugacy classes of finite 
subgroups of $GL_n(\bQ)$ ($n \leq 6$) (see \cite{PS}). 
There exist exactly $2$ (reps. $13$, $73$, $710$, $6079$, $85308$)
$\bZ$-classes forming $2$ (resp. $10$, $32$, $227$, $955$, $7103$) 
$\bQ$-classes in dimension $n=1$ (resp. $2$, $3$, $4$, $5$, $6$)
\footnote{In the old version of CARAT, 
the number of $\bQ$-classes of $GL_6(\bQ)$, 7104, 
was not correct because of the overlapping two same $\bQ$-classes 
(cf. \cite{PS}). 
The third-named author detected it and reported the correct number, 7103, 
to the CARAT group (see also \cite{CARAT}). 
It also turned out that 
the correct number of $\bZ$-classes of $GL_6(\bQ)$ is 85308 (in the old version, 
85311 was wrong). 
This has been fixed in version 2.1b1 of CARAT.}.

After unpacking the CARAT, we get the $\bQ$-catalog file
{\tt carat-2.1b1/tables/qcatalog.tar.gz}.
Unpacking this file,
we get lists of $\bQ$-classes of $GL_n(\bQ)$ 
$(n=1,\ldots,6)$ in ${\tt qcatalog/data1}, \ldots, {\tt qcatalog/data6}$.
Generators of each group are in individual files under the folders
${\tt qcatalog/dim1/}, \ldots, {\tt qcatalog/dim6/}$. 

The third-named author wrote the perl script {\tt crystlst.pl} to 
collect these generators into a single file. 
Files ${\tt cryst1.gap}, \ldots, {\tt cryst6.gap}$ are
lists of representatives of $\bQ$-classes
of $GL_1(\bQ), \ldots, GL_6(\bQ)$ respectively.
These files are available from  
{\tt https://www.math.kyoto-u.ac.jp/\~{}yamasaki/Algorithm/RatProbAlgTori/} as {\tt GLnQ.zip}.

Let $G$ be a finite subgroup of $GL_n(\bZ)$. 
CARAT has a command {\tt ZClassRepsQClass(G)}
to compute the complete $\bZ$-class representatives of the 
$\bQ$-class of $G$. 
We split the $\bQ$-class of $G$ into $\bZ$-classes by the 
command {\tt ZClassRepsQClass(G)}. 
For the $l$-th group $\widetilde{G}$ in the list of $\bZ$-classes 
obtained by {\tt ZClassRepsQClass(}$G${\tt )} 
where $G$ is the $m$-th group ($\bQ$-class) in {\tt qcatalog/data$n$}, 
we say that the CARAT ID of $\widetilde{G}$ is $(n,m,l)$.

The third-named author wrote a GAP program to determine
the $\bQ$-class and the $\bZ$-class of a group $G$.
The files {\tt crystcat.gap} and {\tt caratnumber.gap}
contain programs related to the GAP ID and the CARAT ID respectively.
The file {\tt caratnumber.gap} uses other files which are packed in the 
{\tt crystdat.zip}.
This zip file should be packed at the current directory. 
All the files above are available from\\ 
{\tt https://www.math.kyoto-u.ac.jp/\~{}yamasaki/Algorithm/RatProbAlgTori/}.


%

\bigskip

\section{Proof of Theorem {\ref{t1.5}}}\label{sePT}



Using GAP \cite{GAP}, we can find all the $710$ lattices $M$ of rank $4$.
The readers may find all these finite groups and their associated lattices by using the related functions which are available from\\
{\tt https://www.math.kyoto-u.ac.jp/\~{}yamasaki/Algorithm/MultInvField/}.
\footnote{
We do not need 
CARAT package \cite{CARAT} of GAP when $G\leq {\rm GL}(n,\bZ)$ 
($n\leq 4$) although it is needed when $n=5, 6$.}

\bigskip
Now we  start to prove Theorem \ref{t1.5}.

The key idea is to device a computer algorithm to compute $H_u^2(G,M)$. We call this algorithm the function {\tt H2nrM(g)} of GAP \cite{GAP}. For any $G$-lattice $M$ where $G$ is a finite subgroup of $GL_n(\bm{Z})$, this function returns $H_u^2(G,M)$ (see Section \ref{seGAP}; Section \ref{seTables} is the outcome of the computation).
The function {\tt H2nrM(g)} and associated functions are available from
{\tt https://www.math.kyoto-u.ac.jp/\~{}yamasaki/Algorithm/MultInvField/}.\\

The proof of Theorem \ref{t1.5} proceeds as follows.

By Theorem \ref{thHKHR}, if $M$ is a faithful $G$-lattice of rank $\leq 3$,
then $\bm{C}(M)^G$ is $\bm{C}$-rational and hence ${\rm Br}_u(\bm{C}(M)^G)=0$. For the remaining cases, we will prove the following three theorems. 
Note that there exist $710$ (resp. $6079$, $85308$) faithful 
$G$-lattices $M$ of rank $4$ (resp. $5$, $6$).

\begin{theorem}\label{thm1}
%
Let $G$ be a finite group and $M$ be a faithful $G$-lattice of rank $4$.
Then
${\rm Br}_u(\bm{C}(M)^G)\neq 0$ if and only if
$M$ is one of the $5$ cases as in {\rm Table} $1$ of 
{\rm Section} $\ref{seTables}$ $($or {\rm Section} $\ref{seInt}$$)$.
In particular, such $G$-lattices $M$ with
${\rm Br}_u(\bm{C}(M)^G)\neq 0$ are indecomposable.
Moreover, if $M$ is one of the five $G$-lattices with 
${\rm Br}_u(\bm{C}(M)^G)\neq 0$, then
${\rm Br}_u(\bm{C}(M)^G)=H_u^2(G,M)$; in other words, $H_u^2(G,\bm{Q}/\bm{Z})=0$.
\end{theorem}


\begin{theorem}\label{thm2}
Let $G$ be a finite group and $M$ be a faithful $G$-lattice of rank $5$.
Then
${\rm Br}_u(\bm{C}(M)^G)\neq 0$ if and only if
$M$ is one of the $46$ cases as in {\rm Table} $2$ of {\rm Section} $\ref{seTables}$.
In particular, there exist $27$ $($resp. $18$, $1$$)$ $G$-lattices
with ${\rm Br}_u(\bm{C}(M)^G)\neq 0$
which are indecomposable $($resp. decomposable into indecomposable component
with rank $4+1$, $3+2$$)$.
Moreover, if $M$ is one of the $46$ $G$-lattices with 
${\rm Br}_u(\bm{C}(M)^G)\neq 0$, then ${\rm Br}_u(\bm{C}(M)^G)=H_u^2(G,M)$; in other words,  $H_u^2(G,\bm{Q}/\bm{Z})=0$.
\end{theorem}


\begin{theorem}\label{thm3}
Let $G$ be a finite group and $M$ be a faithful $G$-lattice of rank $6$.
Then
${\rm Br}_u(\bm{C}(M)^G)\neq 0$ if and only if
$M$ is one of the 
$1073$ cases as in {\rm Table} $3$ of {\rm Section} $\ref{seTables}$.
In particular, there exist 
$613$
$($resp. $161$,
$230$,
$59$, $5$, $5$$)$
$G$-lattices with ${\rm Br}_u(\bm{C}(M)^G)\neq 0$
which are indecomposable $($resp. decomposable into indecomposable component
with rank $5+1$, $4+2$, $4+1+1$, $3+3$, $3+2+1$$)$.
Moreover, if $M$ is one of the 
$1073$ $G$-lattices with 
${\rm Br}_u(\bm{C}(M)^G)\neq 0$, then ${\rm Br}_u(\bm{C}(M)^G)=H_u^2(G,M)$,
i.e. $H_u^2(G,\bm{Q}/\bm{Z})=0$, except for $24$ cases
with $H_u^2(G,\bm{Q}/\bm{Z})=\bm{Z}/2\bm{Z}$ of the {\rm CARAT ID}
$(6,6458,i)$, $(6,6459,i)$, $(6,6464,i)$ $(1\leq i\leq 8)$.
The $22$ cases $($resp. $2$ cases$)$
of the exceptional $24$ cases satisfy $H_u^2(G,M)=0$
$($resp. $H_u^2(G,M)\simeq\bm{Z}/2\bm{Z})$.
In particular, only $2$ cases satisfy $H_u^2(G,\bm{Q}/\bm{Z})\neq 0$
and $H_u^2(G,M)\neq 0$.
\end{theorem}

Note that new phenomena arise when the rank increases. For examples, in Theorem \ref{thm2}, there are decomposable lattices $M$ with ${\rm Br}_u(\bm{C}(M)^G)\neq 0$ (compare with Theorem \ref{th116} and Theorem \ref{thm1}). Also, in Theorem \ref{thm3}, there are lattices $M$ with ${\rm Br}_u(\bm{C}(M)^G)\neq 0$ and $H_u^2(G,\bm{Q}/\bm{Z}) \neq 0$ while $H_u^2(G,M)$ may be trivial or not (compare with Theorem \ref{thm2}). In Theorem \ref{equiv}, we will find an example that multiplicative invariant fields arising from different lattices become stably isomorphic.

\bigskip
We explain how to apply the function {\tt H2nrM(g)} of GAP \cite{GAP} to prove Theorem \ref{t1.5}.

Let $G$ be a finite subgroup of $GL_n(\bm{Z})$. For each prime number $p \mid |G|$, choose a $p$-Sylow subgroup $G_p$ of $G$. Since the restriction map $H^i (G, M_{\alpha}) \to \oplus_{p\mid |G|}H^i (G_p, M_{\alpha})$ is injective for $i \ge 1$ (see \cite[page 130, Proposition 4]{Se}), it follows that $H^2_u (G, M_{\alpha}) \to \oplus_{p\mid |G|}H^2_u (G_p, M_{\alpha})$ is also injective. On the other hand, if the $\mu$-extension $(\alpha) : 1\rightarrow \mu\rightarrow M_{\alpha}\rightarrow M\rightarrow 0$ is a split extension, then $H_u^2(G, M_{\alpha}) \simeq H_u^2(G,\bm{Q}/\bm{Z})\oplus H_u^2(G,M)$.

In other words, in order to show that $H^2_u (G, M_{\alpha})=0$ in Theorems \ref{thm1}, \ref{thm2} and \ref{thm3}, it suffices to show that $B_0(G)=0$ and $H^2_u (G_p, M)=0$ for $p\mid |G|$.

To show that $H^2_u (G_p, M)=0$, we may apply Theorem \ref{thB1} to narrow down  the class of lattices and then apply the function {\tt H2nrM(g)}. In the following proofs of Theorems \ref{thm1}, Theorem \ref{thm2} and Theorem \ref{thm3}, the Steps $1$-$4$ are meant to explain the ideas of the algorithms, then the computer program follows (with the corresponding Steps $1$-$4$).

Now we go ahead to the proofs of these theorems.

\vspace*{2mm}
{\it Proof of Theorem \ref{thm1}.}\vspace*{2mm}

We will compute $H_u^2(G,M)$ first. 
The computation of $B_0(G)$ is given in Step $4$.

\medskip
Step 1. First we study the structure of maximal finite $p$-subgroups of $GL_4(\bm{Z})$.
It is easy to see that, if $G$ is a finite $p$-group in $GL_4(\bm{Z})$, then $p=2, 3,$ or $5$ 
(see the 4th command of Step 1 on page \pageref{4th}).

The maximal $3$-subgroup, denoted by $G_3$ (resp. $5$-subgroup, denoted by $G_5$)
is of order $9$ (resp. $5$). Hence
${\rm Br}_u(\bm{C}(M)^{G_p})=0$ by Theorem \ref{thB1}. It remains to
check $2$-groups $G$ with $|G|\geq 8$ which are not bicyclic.

It can be shown a 
$2$-group in $GL_4(\bm{Z})$ is of order $\le 128$. 
Moreover, 
the total numbers of non-bicyclic subgroups $G$ in $GL_4(\bm{Z})$ of order
$8$, $16$, $32$, $64$, $128$ are: $90$, $81$, $33$, $14$, $2$ respectively 
(see the 6th command of Step 1 on page \pageref{6th}).

Apply the function {\tt H2nrM(G)} to the lattices associated to these groups. We find that only three groups $G$ have non-trivial $H^2_u (G, M)$. The GAP ID of these three groups are :
$(4,12,4,12)$, $(4,32,1,2)$, $(4,32,3,2)$. More precisely, if the GAP ID of $G$ is $(4,12,4,12)$, then
$G \simeq D_4$ and $H_u^2(G,M) \simeq \bm{Z}/2\bm{Z}$; if the GAP ID of $G$ is $(4,32,1,2)$, then
$G \simeq Q_8$ and $H_u^2(G,M) \simeq \bm{Z}/2\bm{Z}\oplus\bm{Z}/2\bm{Z}$; if the GAP ID of $G$ is $(4,32,3,2)$, then
$G \simeq QD_8$ and $H_u^2(G,M) \simeq \bm{Z}/2\bm{Z}$ 
(see the 7th to the 10th commands of Step 1 on page \pageref{7th}).

Conclusion: Let $G$ be a finite subgroup of $GL_4(\bm{Z})$ (no matter what $G$ is a $2$-group or not). If the $2$-Sylow subgroup of $G$ is not isomorphic to groups with GAP ID $(4,12,4,12)$, $(4,32,1,2)$, $(4,32,3,2)$, then  ${\rm Br}_u(\bm{C}(M)^{G})=0$.

\medskip
Step 2.
Let $G$ be a finite subgroup of $GL_4(\bm{Z})$ which is not a $2$-group. Because of Theorem \ref{thB1}, it suffices to consider the case when $G$ is not bicyclic.
By the conclusion of Step 1, we may restrict only to those groups of order $8m$ or $16m$ where $m \ge 3$ is an odd integer (if the group order is $8$ or $16$, it becomes a $2$-group which has been treated in Step 1). Inspecting the list of finite subgroups in $GL_4(\bm{Z})$, we find that there are no groups of order $8 \cdot 7$, $8 \cdot 11$, $8 \cdot 13, \ldots, 16 \cdot 5, 16 \cdot 7, \ldots$. We are reduced to groups of order $24$, $40$, $72$, $120$, $48$, $144$, $240$.

The total numbers of non-bicyclic subgroups $G$ in $GL_4(\bm{Z})$ of order
$24$, $40$, $72$, $120$ are: $95$, $2$, $27$, $6$ respectively 
(see Reference A of Step 2 on page \pageref{refA}); 
and the total numbers of non-bicyclic subgroups $G$ in $GL_4(\bm{Z})$ 
of order $48$, $144$, $240$ are: $58$, $9$, $2$ respectively 
(see Reference B of Step 2 on page \pageref{refB}).

Not all of the above groups are eligible, because only those groups 
whose $2$-Sylow subgroups are isomorphic to  one of the three groups 
we find in Step 1 are the ``bad" groups. 
The numbers of these ``bad" groups are $4$ (for order $24$) 
and $1$ (for order $72$) 
(see Reference C of Step 2 on page \pageref{refC}); 
similarly, the numbers of these ``bad" groups are $4$ 
(for order $48$) and $1$ (for order $144$) 
(see Reference D of Step 2 on page \pageref{refD}).

Apply the function {\tt H2nrM(G)} to the above ten groups $G$,
we find that $H_u^2(G,M)=0$
except for two groups. These two groups are the groups whose GAP ID are
$(4,33,3,1)$ and $(4,33,6,1)$.
The group $G$ of the GAP ID $(4,33,3,1)$ (resp. $(4,33,6,1)$)
is isomorphic to $SL_2(\bm{F}_3)$
(resp. $GL_2(\bm{F}_3)$)
with $H_u^2(G,M)=\bm{Z}/2\bm{Z}\oplus\bm{Z}/2\bm{Z}$
(resp. $\bm{Z}/2\bm{Z}$).

\medskip
Step 3. By Theorem \ref{th116},
all of the $5=3+2$ $G$-lattices $M$ with $H_u^2(G,M)\neq 0$ 
in Steps $1$ and $2$ 
are indecomposable because $\bm{C}(M)^G$ is $\bm{C}$-rational 
if $M$ is decomposable 
(see Step 3 on page \pageref{Step3}).

\medskip
Step 4.
We will show that $H_u^2(G,\bm{Q}/\bm{Z})=0$ when $G$ are the ``bad" groups found above.

By Theorem \ref{thCK} and Theorem \ref{thCHKP},
if $G$ is a $p$-group of order $\leq p^4$ or $32$,
then $\bm{C}(G)$ is $\bm{C}$-rational and hence $H_u^2(G,\bm{Q}/\bm{Z})=0$.
It remains to check whether $H_u^2(G,\bm{Q}/\bm{Z})=0$ for $|G|=64$ or $128$.
It follows from \cite{CHKK} and \cite{Mo} that $H_u^2(G,\bm{Q}/\bm{Z})=0$
when $G$ is a subgroup of $GL_4(\bm{Z})$ of order $64$ or $128$.
We may also check this using the function {\tt B0G(G)}
in the program b0g.g by Moravec \cite{Mo} or the function
{\tt BogomolovMultiplier(G)} in HAP \cite{HAP} 
(see Step 4 on page \pageref{Step4}). \\

Before proceeding to the computations of GAP \cite{GAP} below,
we should read the related data from\\
{\tt https://www.math.kyoto-u.ac.jp/\~{}yamasaki/Algorithm/MultInvField/}.
\footnote{
We need 
CARAT package \cite{CARAT} of GAP for  
$G\leq {\rm GL}(n,\bZ)$ $(n=5, 6)$ and it 
works on Linux or macOS, but not on Windows.}
\\


\begin{verbatim}
gap> Read("MultInvField.gap");

# Step 1
gap> GL4Z:=Flat(List([1..33],x->List([1..NrQClassesCrystalSystem(4,x)],
> y->List([1..NrZClassesQClass(4,x,y)],z->MatGroupZClass(4,x,y,z)))));; 
# all the finite subgroups of GL4(Z) up to Z-conjugate
gap> Length(GL4Z);
710
gap> o4:=List(GL4Z,Order);;
\end{verbatim}\vspace*{-1.6mm}
\label{4th}
\begin{verbatim}
gap> Collected(Factors(Lcm(o4)));
[ [ 2, 7 ], [ 3, 2 ], [ 5, 1 ] ]
gap> nbc4_2:=Filtered(GL4Z,x->2^7 mod Order(x)=0
> and not IsBicyclic(x));; # |G|=2^n case
\end{verbatim}\vspace*{-1.6mm}
\label{6th}
\begin{verbatim}
gap> Collected(List(nbc4_2,Order)); # #=90+81+33+14+2
[ [ 8, 90 ], [ 16, 81 ], [ 32, 33 ], [ 64, 14 ], [ 128, 2 ] ]
\end{verbatim}\vspace*{-1.6mm}
\label{7th}
\begin{verbatim}
gap> ns4_2:=Filtered(nbc4_2,x->H2nrM(x).H2nrM<>[]); 
# choosing only the cases with non-trivial H2nr(G,M)
[ MatGroupZClass( 4, 12, 4, 12 ), MatGroupZClass( 4, 32, 1, 2 ),
  MatGroupZClass( 4, 32, 3, 2 ) ]
gap> ns4_2i:=List(ns4_2,CrystCatZClassNumber); # #=3 with non-trivial H2nr(G,M)
[ [ 4, 12, 4, 12 ], [ 4, 32, 1, 2 ], [ 4, 32, 3, 2 ] ]
gap> List(ns4_2,Order); # order of G
[ 8, 8, 16 ]
gap> List(ns4_2,StructureDescription); # structure of G
[ "D8", "Q8", "QD16" ]

# Step 2
gap> nbc4_8m:=Filtered(GL4Z,x->Order(x) mod 16=8 and Order(x)>8
> and not IsBicyclic(x));; # |G|=8m case
\end{verbatim}\vspace*{-1.6mm}
\label{refA}
\begin{verbatim}
gap> Collected(List(nbc4_8m,Order)); # #=95+2+27+6 (Reference A)
[ [ 24, 95 ], [ 40, 2 ], [ 72, 27 ], [ 120, 6 ] ]
gap> nbc4_8m_2:=Filtered(nbc4_8m,
> x->CrystCatZClassNumber(SylowSubgroup(x,2)) in ns4_2i); 
# only the case where 2-Sylow subgroup is in Step 1
[ MatGroupZClass( 4, 32, 5, 2 ), MatGroupZClass( 4, 32, 5, 3 ),
  MatGroupZClass( 4, 33, 1, 1 ), MatGroupZClass( 4, 33, 3, 1 ),
  MatGroupZClass( 4, 33, 7, 1 ) ]
gap> Collected(List(nbc4_8m_2,x->CrystCatZClass(SylowSubgroup(x,2)))); 
[ [ [ 4, 32, 1, 2 ], 5 ] ]
\end{verbatim}\vspace*{-1.6mm}
\label{refC}
\begin{verbatim}
gap> Collected(List(nbc4_8m_2,Order)); # #=5=4+1 (Reference C)
[ [ 24, 4 ], [ 72, 1 ] ]
gap> ns4_8m:=Filtered(nbc4_8m_2,x->H2nrM(x).H2nrM<>[]); 
# choosing only the cases with non-trivial H2nr(G,M)
[ MatGroupZClass( 4, 33, 3, 1 ) ]
gap> ns4_8mi:=List(ns4_8m,CrystCatZClassNumber); # #=1 with non-trivial H2nr(G,M)
[ [ 4, 33, 3, 1 ] ]
gap> List(ns4_8m,Order); # order of G
[ 24 ]
gap> List(ns4_8m,StructureDescription); # structure of G
[ "SL(2,3)" ]

gap> nbc4_16m:=Filtered(GL4Z,x->Order(x) mod 32=16 and Order(x)>16
> and not IsBicyclic(x));; # |G|=16m case
\end{verbatim}\vspace*{-1.6mm}
\label{refB}
\begin{verbatim}
gap> Collected(List(nbc4_16m,Order)); # #=58+9+2 (Reference B)
[ [ 48, 58 ], [ 144, 9 ], [ 240, 2 ] ]
gap> nbc4_16m_2:=Filtered(nbc4_16m,
> x->CrystCatZClassNumber(SylowSubgroup(x,2)) in ns4_2i);
# only the case where 2-Sylow subgroup is in Step 1
[ MatGroupZClass( 4, 32, 11, 2 ), MatGroupZClass( 4, 32, 11, 3 ),
  MatGroupZClass( 4, 33, 4, 1 ), MatGroupZClass( 4, 33, 6, 1 ),
  MatGroupZClass( 4, 33, 11, 1 ) ]
gap> Collected(List(nbc4_16m_2,x->CrystCatZClass(SylowSubgroup(x,2))));
[ [ [ 4, 32, 3, 2 ], 5 ] ]
\end{verbatim}\vspace*{-1.6mm}
\label{refD}
\begin{verbatim}
gap> Collected(List(nbc4_16m_2,Order)); # #=5+4+1 (Reference D)
[ [ 48, 4 ], [ 144, 1 ] ]
gap> ns4_16m:=Filtered(nbc4_16m_2,x->H2nrM(x).H2nrM<>[]); 
# choosing only the cases with non-trivial H2nr(G,M)
[ MatGroupZClass( 4, 33, 6, 1 ) ]
gap> ns4_16mi:=List(ns4_16m,CrystCatZClassNumber); # #=1 with non-trivial H2nr(G,M)
[ [ 4, 33, 6, 1 ] ]
gap> List(ns4_16m,Order); # order of G
[ 48 ]
gap> List(ns4_16m,StructureDescription); # structure of G
[ "GL(2,3)" ]

gap> ns4:=Union(ns4_2i,ns4_8mi,ns4_16mi); # all cases with non-trivial H2nr(G,M)
[ [ 4, 12, 4, 12 ], [ 4, 32, 1, 2 ], [ 4, 32, 3, 2 ], [ 4, 33, 3, 1 ], [ 4, 33, 6, 1 ] ]
gap> Length(ns4); # total #=5=3+1+1 cases with non-trivial H2nr(G,M)
5

\end{verbatim}\vspace*{-2mm}
\label{Step3}
\begin{verbatim}
# Step 3
gap> NS4:=Intersection(ns4,d4); 
# d4 is the list of GAP IDs of indecomposable lattices of rank 4
[ [ 4, 12, 4, 12 ], [ 4, 32, 1, 2 ], [ 4, 32, 3, 2 ], [ 4, 33, 3, 1 ], [ 4, 33, 6, 1 ] ]
gap> Length(NS4); # all the 5 cases are indecomposable
5
gap> List(NS4,x->H2nrM(MatGroupZClass(x[1],x[2],x[3],x[4])).H2nrM); 
# structure of H2nr(G,M)
[ [ 2 ], [ 2, 2 ], [ 2 ], [ 2, 2 ], [ 2 ] ]

\end{verbatim}\vspace*{-2mm}
\label{Step4}
\begin{verbatim}
# Step 4
gap> LoadPackage("HAP");
true
gap> GL4Q:=Flat(List([1..33],x->List([1..NrQClassesCrystalSystem(4,x)],
> y->MatGroupZClass(4,x,y,1))));;
# all the finite subgroups of GL4(Z) up to Q-conjugate
gap> nb4:=Filtered(GL4Q,x->BogomolovMultiplier(x)<>[]); # checking B0(G)=0
[  ]
\end{verbatim}

\vspace*{2mm}
{\it Proof of Theorem \ref{thm2}.}\vspace*{2mm}

We compute $H^2_u (G,M)$ in Steps $1$-$3$ as before.

\medskip
Step 1.
It is not difficult to show that, if $G$ is a finite $p$-group in $GL_5(\bm{Z})$, then $p=2, 3,$ or $5$ 
(see the 4th command of Step 1 on page \pageref{4th-2}).

The maximal $3$-subgroup, denoted by $G_3$ (resp. $5$-subgroup, denoted by $G_5$)
is of order $9$ (resp. $5$). Hence
${\rm Br}_u(\bm{C}(M)^{G_p})=0$ by Theorem \ref{thB1}. It remains to
check $2$-groups $G$ with $|G|\geq 8$ which are not bicyclic.

It can be shown a non-bicyclic $2$-group in $GL_5(\bm{Z})$ is of order $\le 256$. Moreover,
the total numbers of non-bicyclic subgroups $G$ in $GL_5(\bm{Z})$ of order
$8$, $16$, $32$, $64$, $128$, $256$ are: $507$, $1030$, $700$, $247$, $61$, $4$ respectively 
(see the 6th command of Step 1 on page \pageref{6th-2}).

Apply the function {\tt H2nrM(G)} to the lattices associated to these groups. We find that exactly 41 groups $G$ have non-trivial $H^2_u (G, M)$. The number of these groups are: $15$ (groups of order $8$), $20$ (groups of order $16$), and $6$ (groups of order $32$) 
(see the 7th to the 11th commands of Step 1 on page \pageref{7th-2}); 
we omit the GAP ID of these groups (the reader can find them in the lists of Section \ref{seTables}).

Conclusion: Let $G$ be a finite subgroup of $GL_5(\bm{Z})$ (no matter what $G$ is a $2$-group or not). If the $2$-Sylow subgroup of $G$ is not isomorphic to one of the $41=15+20+6$ groups found before, then  ${\rm Br}_u(\bm{C}(M)^{G})=0$.

\medskip
Step 2.
Let $G$ be a finite subgroup of $GL_5(\bm{Z})$ which is not a $2$-group. Because of Theorem \ref{thB1}, we consider the non-bicyclic groups.
By the conclusion of Step 1, we may restrict only to those groups of order $8m$, $16m$ or $32m$ where $m \ge 3$ is an odd integer $\ge 3$ (as before, we will not consider groups of order $8$, $16$, or $32$).

As in the case of Step 2 of the proof of Theorem \ref{thm1}, apply {\tt H2nrM(G)} again for these non-bicyclic groups. We find that
$H_u^2(G,M)=0$  except for the groups $G$ whose CARAT ID are :
$(5,691,1)$, $(5,730,1)$, $(5,733,1)$, $(5,734,1)$, $(5,776,1)$. The group orders of these $5$ exceptional groups are: $24$, $48$, $48$, $48$, $96$ respectively, and they are isomorphic to $SL_2(\bm{F}_3)$, $C_2\times SL_2(\bm{F}_3)$, $GL_2(\bm{F}_3)$, $GL_2(\bm{F}_3)$,
$C_2\times GL_2(\bm{F}_3)$ respectively. The corresponding unramified cohomology $H_u^2(G,M)$ of them are $\bm{Z}/2\bm{Z}\oplus\bm{Z}/2\bm{Z}$, $\bm{Z}/2\bm{Z}\oplus\bm{Z}/2\bm{Z}$, $\bm{Z}/2\bm{Z}$,
$\bm{Z}/2\bm{Z}$, $\bm{Z}/2\bm{Z}$ respectively.

In conclusion, we find all the lattices $M$ associated to the $46$ groups ($41$ groups obtained in Step 1 and $5=1+3+1$ additional groups in this step) with $H_u^2(G,M) \neq 0$.

\medskip
Step 3.
Using the function {\tt LatticeDecompositions(n)} which is provided by [HY, Section 4.1] we find that, among the $46$ lattices in Step 2, the numbers of lattices which are indecomposable, decomposable into indecomposable component with rank $4+1$, decomposable into indecomposable component with rank $3+2$ are: $27$, $18$, $1$ respectively. The reader may consult also [HY, Example 4.9], where a method to classify the decomposable lattices of type $4+1$ and $3+2$ is provided.

\medskip
Step 4.
We verify that $H_u^2(G,\bm{Q}/\bm{Z})=0$ by using the function {\tt B0G(G)}
in the program b0g.g by Moravec \cite{Mo} or the function
{\tt BogomolovMultiplier(G)} in HAP \cite{HAP} (see also Step 4 in the proof of Theorem \ref{thm1}).
 \\

Before proceeding to the computations of GAP \cite{GAP} below,
we should read the related data from\\
{\tt https://www.math.kyoto-u.ac.jp/\~{}yamasaki/Algorithm/MultInvField/}.
\footnote{
We need 
CARAT package \cite{CARAT} of GAP for  
$G\leq {\rm GL}(n,\bZ)$ $(n=5, 6)$ and it 
works on Linux or macOS, but not on Windows.}
\\

\begin{verbatim}
gap> Read("MultInvField.gap");

# Step 1
gap> GL5Z:=Flat(List([1..955],x->List([1..CaratNrZClasses(5,x)],
> y->CaratMatGroupZClass(5,x,y))));;
# all the finite subgroups of GL5(Z) up to Z-conjugate
gap> Length(GL5Z);
6079
gap> o5:=List(GL5Z,Order);;
\end{verbatim}\vspace*{-1.6mm}
\label{4th-2}
\begin{verbatim}
gap> Collected(Factors(Lcm(o5)));
[ [ 2, 8 ], [ 3, 2 ], [ 5, 1 ] ]
gap> nbc5_2:=Filtered(GL5Z,x->2^8 mod Order(x)=0
> and not IsBicyclic(x));; # |G|=2^n case
\end{verbatim}\vspace*{-1.6mm}
\label{6th-2}
\begin{verbatim}
gap> Collected(List(nbc5_2,Order)); # #=507+1030+700+247+61+4
[ [ 8, 507 ], [ 16, 1030 ], [ 32, 700 ], [ 64, 247 ], [ 128, 61 ], [ 256, 4 ] ]
\end{verbatim}\vspace*{-1.6mm}
\label{7th-2}
\begin{verbatim}
gap> ns5_2:=Filtered(nbc5_2,x->H2nrM(x).H2nrM<>[]);; 
# choosing only the cases with non-trivial H2nr(G,M)
gap> ns5_2i:=List(ns5_2,CaratZClassNumber);
[ [ 5, 39, 5 ], [ 5, 63, 12 ], [ 5, 65, 12 ], [ 5, 71, 19 ], [ 5, 73, 37 ],
  [ 5, 76, 31 ], [ 5, 76, 49 ], [ 5, 76, 50 ], [ 5, 76, 51 ], [ 5, 79, 18 ],
  [ 5, 99, 5 ], [ 5, 99, 23 ], [ 5, 99, 24 ], [ 5, 99, 25 ], [ 5, 100, 5 ],
  [ 5, 100, 11 ], [ 5, 100, 12 ], [ 5, 100, 23 ], [ 5, 105, 5 ],
  [ 5, 109, 5 ], [ 5, 116, 20 ], [ 5, 118, 18 ], [ 5, 119, 4 ],
  [ 5, 140, 23 ], [ 5, 142, 14 ], [ 5, 664, 1 ], [ 5, 664, 2 ],
  [ 5, 672, 1 ], [ 5, 672, 2 ], [ 5, 673, 1 ], [ 5, 673, 2 ], [ 5, 674, 1 ],
  [ 5, 675, 1 ], [ 5, 704, 3 ], [ 5, 706, 8 ], [ 5, 721, 1 ], [ 5, 721, 2 ],
  [ 5, 773, 3 ], [ 5, 773, 4 ], [ 5, 774, 3 ], [ 5, 774, 4 ] ]
gap> Collected(List(ns5_2,Order)); # #=41=15+20+6
[ [ 8, 15 ], [ 16, 20 ], [ 32, 6 ] ]
gap> List(ns5_2,StructureDescription); # structure of G
[ "D8", "D8", "D8", "C2 x D8", "C2 x D8", "C2 x D8", "C2 x D8", "C2 x D8",
  "C2 x D8", "C2 x D8", "D8", "D8", "D8", "D8", "D8", "D8", "D8", "D8",
  "C4 : C4", "(C4 x C2) : C2", "(C4 x C2) : C2", "(C4 x C2) : C2", "C2 x D8",
  "(C4 x C2 x C2) : C2", "(C2 x C2 x C2 x C2) : C2", "C2 x Q8", "C2 x Q8",
  "QD16", "QD16", "QD16", "QD16", "QD16", "QD16", "((C4 x C2) : C2) : C2",
  "((C4 x C2) : C2) : C2", "C2 x QD16", "C2 x QD16", "Q8", "Q8", "Q8", "Q8" ]
gap> List(ns5_2,IdSmallGroup);
[ [ 8, 3 ], [ 8, 3 ], [ 8, 3 ], [ 16, 11 ], [ 16, 11 ], [ 16, 11 ],
  [ 16, 11 ], [ 16, 11 ], [ 16, 11 ], [ 16, 11 ], [ 8, 3 ], [ 8, 3 ],
  [ 8, 3 ], [ 8, 3 ], [ 8, 3 ], [ 8, 3 ], [ 8, 3 ], [ 8, 3 ], [ 16, 4 ],
  [ 16, 3 ], [ 16, 3 ], [ 16, 3 ], [ 16, 11 ], [ 32, 28 ], [ 32, 27 ],
  [ 16, 12 ], [ 16, 12 ], [ 16, 8 ], [ 16, 8 ], [ 16, 8 ], [ 16, 8 ],
  [ 16, 8 ], [ 16, 8 ], [ 32, 6 ], [ 32, 6 ], [ 32, 40 ], [ 32, 40 ],
  [ 8, 4 ], [ 8, 4 ], [ 8, 4 ], [ 8, 4 ] ]

# Step 2
gap> nbc5_8m:=Filtered(GL5Z,x->Order(x) mod 16=8 and Order(x)>8
> and not IsBicyclic(x)); # |G|=8m
gap> Collected(List(nbc5_8m,Order));
[ [ 24, 663 ], [ 40, 15 ], [ 72, 249 ], [ 120, 32 ], [ 360, 4 ] ]
gap> nbc5_8m_2:=Filtered(nbc5_8m,
> x->CaratZClassNumber(SylowSubgroup(x,2)) in ns5_2i);;
gap> Collected(List(nbc5_8m_2,x->CaratZClass(SylowSubgroup(x,2))));
[ [ [ 5, 773, 3 ], 5 ], [ [ 5, 773, 4 ], 1 ], [ [ 5, 774, 3 ], 1 ] ]
gap> Collected(List(nbc5_8m_2,Order));
[ [ 24, 6 ], [ 72, 1 ] ]
gap> ns5_8m:=Filtered(nbc5_8m_2,x->H2nrM(x).H2nrM<>[]); 
# choosing only the cases with non-trivial H2nr(G,M)
[ <matrix group of size 24 with 2 generators> ]
gap> ns5_8mi:=List(ns5_8m,CaratZClassNumber); # #=1 with non-trivial H2nr(G,M)
[ [ 5, 691, 1 ] ]
gap> List(ns5_8m,Order); # order of G
[ 24 ]
gap> List(ns5_8m,StructureDescription); # structure of G
[ "SL(2,3)" ]
gap> List(ns5_8m,x->CaratZClass(SylowSubgroup(x,2)));
[ [ 5, 773, 3 ] ]

gap> nbc5_16m:=Filtered(GL5Z,x->Order(x) mod 32=16 and Order(x)>16
> and not IsBicyclic(x)); # |G|=16m
gap> Collected(List(nbc5_16m,Order));
[ [ 48, 714 ], [ 80, 5 ], [ 144, 185 ], [ 240, 20 ], [ 720, 12 ] ]
gap> nbc5_16m_2:=Filtered(nbc5_16m,
> x->CaratZClassNumber(SylowSubgroup(x,2)) in ns5_2i);;
gap> Collected(List(nbc5_16m_2,x->CaratZClass(SylowSubgroup(x,2))));
[ [ [ 5, 664, 1 ], 5 ], [ [ 5, 664, 2 ], 1 ], [ [ 5, 672, 1 ], 5 ],
  [ [ 5, 672, 2 ], 1 ], [ [ 5, 673, 1 ], 5 ], [ [ 5, 673, 2 ], 1 ],
  [ [ 5, 674, 1 ], 1 ], [ [ 5, 675, 1 ], 1 ] ]
gap> Collected(List(nbc5_16m_2,Order)); 
[ [ 48, 17 ], [ 144, 3 ] ]
gap> ns5_16m:=Filtered(nbc5_16m_2,x->H2nrM(x).H2nrM<>[]); 
# choosing only the cases with non-trivial H2nr(G,M)
[ <matrix group of size 48 with 2 generators>,
  <matrix group of size 48 with 2 generators>,
  <matrix group of size 48 with 2 generators> ]
gap> ns5_16mi:=List(ns5_16m,CaratZClassNumber); # #=3 with non-trivial H2nr(G,M)
[ [ 5, 730, 1 ], [ 5, 733, 1 ], [ 5, 734, 1 ] ]
gap> List(ns5_16m,Order);
[ 48, 48, 48 ]
gap> List(ns5_16m,StructureDescription); # structure of G
[ "C2 x SL(2,3)", "GL(2,3)", "GL(2,3)" ]
gap> List(ns5_16m,x->CaratZClass(SylowSubgroup(x,2)));
[ [ 5, 664, 1 ], [ 5, 673, 1 ], [ 5, 672, 1 ] ]

gap> nbc5_32m:=Filtered(GL5Z,x->Order(x) mod 64=32 and Order(x)>32
> and not IsBicyclic(x)); # |G|=32m
gap> Collected(List(nbc5_32m,Order));
[ [ 96, 336 ], [ 160, 9 ], [ 288, 63 ], [ 480, 2 ], [ 1440, 4 ] ]
gap> nbc5_32m_2:=Filtered(nbc5_32m,
> x->CaratZClassNumber(SylowSubgroup(x,2)) in ns5_2i);;
gap> Collected(List(nbc5_32m_2,x->CaratZClass(SylowSubgroup(x,2))));
[ [ [ 5, 721, 1 ], 5 ], [ [ 5, 721, 2 ], 1 ] ]
gap> Collected(List(nbc5_32m_2,Order));
[ [ 96, 5 ], [ 288, 1 ] ]
gap> ns5_32m:=Filtered(nbc5_32m_2,x->H2nrM(x).H2nrM<>[]);
# choosing only the cases with non-trivial H2nr(G,M)
[ <matrix group of size 96 with 2 generators> ]
gap> ns5_32mi:=List(ns5_32m,CaratZClassNumber); # #=1 with non-trivial H2nr(G,M)
[ [ 5, 776, 1 ] ]
gap> List(ns5_32m,Order);
[ 96 ]
gap> List(ns5_32m,StructureDescription); # structure of G
[ "C2 x GL(2,3)" ]
gap> List(ns5_32m,x->CaratZClass(SylowSubgroup(x,2)));
[ [ 5, 721, 1 ] ]

gap> ns5:=Union(ns5_2i,ns5_8mi,ns5_16mi,ns5_32mi); 
# all cases with non-trivial H2nr(G,M)
[ [ 5, 39, 5 ], [ 5, 63, 12 ], [ 5, 65, 12 ], [ 5, 71, 19 ], [ 5, 73, 37 ],
  [ 5, 76, 31 ], [ 5, 76, 49 ], [ 5, 76, 50 ], [ 5, 76, 51 ], [ 5, 79, 18 ],
  [ 5, 99, 5 ], [ 5, 99, 23 ], [ 5, 99, 24 ], [ 5, 99, 25 ], [ 5, 100, 5 ],
  [ 5, 100, 11 ], [ 5, 100, 12 ], [ 5, 100, 23 ], [ 5, 105, 5 ],
  [ 5, 109, 5 ], [ 5, 116, 20 ], [ 5, 118, 18 ], [ 5, 119, 4 ],
  [ 5, 140, 23 ], [ 5, 142, 14 ], [ 5, 664, 1 ], [ 5, 664, 2 ],
  [ 5, 672, 1 ], [ 5, 672, 2 ], [ 5, 673, 1 ], [ 5, 673, 2 ], [ 5, 674, 1 ],
  [ 5, 675, 1 ], [ 5, 691, 1 ], [ 5, 704, 3 ], [ 5, 706, 8 ], [ 5, 721, 1 ],
  [ 5, 721, 2 ], [ 5, 730, 1 ], [ 5, 733, 1 ], [ 5, 734, 1 ], [ 5, 773, 3 ],
  [ 5, 773, 4 ], [ 5, 774, 3 ], [ 5, 774, 4 ], [ 5, 776, 1 ] ]
gap> Length(ns5); # total #=46=41+1+3+1 cases with non-trivial H2nr(G,M)
46

# Step 3
gap> NS5:=Intersection(ns5,e5); 
# e5 is the list of CARAT IDs of indecomposable lattices of rank 5
[ [ 5, 39, 5 ], [ 5, 71, 19 ], [ 5, 73, 37 ], [ 5, 76, 49 ], [ 5, 76, 50 ],
  [ 5, 76, 51 ], [ 5, 79, 18 ], [ 5, 99, 23 ], [ 5, 99, 24 ], [ 5, 99, 25 ],
  [ 5, 100, 12 ], [ 5, 100, 23 ], [ 5, 105, 5 ], [ 5, 109, 5 ],
  [ 5, 116, 20 ], [ 5, 118, 18 ], [ 5, 119, 4 ], [ 5, 140, 23 ],
  [ 5, 142, 14 ], [ 5, 664, 2 ], [ 5, 672, 2 ], [ 5, 673, 2 ], [ 5, 704, 3 ],
  [ 5, 706, 8 ], [ 5, 721, 2 ], [ 5, 773, 4 ], [ 5, 774, 4 ] ]
gap> Length(NS5); # 27 cases are indecomposable
27
gap> List(NS5,x->H2nrM(CaratMatGroupZClass(x[1],x[2],x[3])).H2nrM);
# structure of H2nr(G,M)
[ [ 2 ], [ 2 ], [ 2 ], [ 2 ], [ 2 ], [ 2 ], [ 2 ], [ 2 ], [ 2 ], [ 2 ],
  [ 2 ], [ 2 ], [ 2 ], [ 2 ], [ 2 ], [ 2 ], [ 2 ], [ 2 ], [ 2 ], [ 2, 2 ],
  [ 2 ], [ 2 ], [ 2 ], [ 2 ], [ 2 ], [ 2, 2 ], [ 2 ] ]

gap> NS41:=Intersection(ns5,e41);
# e41 is the list of CARAT IDs of decomposable lattices with rank 4+1
[ [ 5, 63, 12 ], [ 5, 65, 12 ], [ 5, 76, 31 ], [ 5, 99, 5 ], [ 5, 100, 5 ],
  [ 5, 664, 1 ], [ 5, 672, 1 ], [ 5, 673, 1 ], [ 5, 674, 1 ], [ 5, 675, 1 ],
  [ 5, 691, 1 ], [ 5, 721, 1 ], [ 5, 730, 1 ], [ 5, 733, 1 ], [ 5, 734, 1 ],
  [ 5, 773, 3 ], [ 5, 774, 3 ], [ 5, 776, 1 ] ]
gap> Length(NS41); # 18 cases are of rank 4+1 
18
gap> List(NS41,x->H2nrM(CaratMatGroupZClass(x[1],x[2],x[3])).H2nrM);
# structure of H2nr(G,M)
[ [ 2 ], [ 2 ], [ 2 ], [ 2 ], [ 2 ], [ 2, 2 ], [ 2 ], [ 2 ], [ 2 ], [ 2 ],
  [ 2, 2 ], [ 2 ], [ 2, 2 ], [ 2 ], [ 2 ], [ 2, 2 ], [ 2, 2 ], [ 2 ] ]

gap> NS32:=Intersection(ns5,e32);
# e32 is the list of CARAT IDs of decomposable lattices with rank 3+2
[ [ 5, 100, 11 ] ]
gap> Length(NS32); # 1 case is of rank 3+2 
1
gap> List(NS32,x->H2nrM(CaratMatGroupZClass(x[1],x[2],x[3])).H2nrM);
# structure of H2nr(G,M)
[ [ 2 ] ]

gap> Union(NS5,NS41,NS32)=ns5;
true
gap> List([NS5,NS41,NS32],Length); 
# total #=46=27+18+1 cases with non-trivial H2nr(G,M)
[ 27, 18, 1 ]

gap> CaratZClassNumber(DirectProductMatrixGroup(
> [MatGroupZClass(4,12,4,12),Group([[[1]]])]));
[ 5, 99, 5 ]
gap> CaratZClassNumber(DirectProductMatrixGroup(
> [MatGroupZClass(4,12,4,12),Group([[[-1]]])]));
[ 5, 76, 31 ]
gap> M1:=DirectProductMatrixGroup([MatGroupZClass(4,12,4,12),Group([[[-1]]])]);
<matrix group with 3 generators>
gap> Filtered(List(ConjugacyClassesSubgroups(M1),Representative),
> x->Order(PartialMatrixGroup(x,[1..4]))=8);
[ <matrix group of size 8 with 3 generators>,
  <matrix group of size 8 with 3 generators>,
  <matrix group of size 8 with 3 generators>,
  <matrix group of size 8 with 3 generators>,
  <matrix group of size 16 with 4 generators> ]
gap> Set(last,CaratZClassNumber);
[ [ 5, 63, 12 ], [ 5, 65, 12 ], [ 5, 76, 31 ], [ 5, 99, 5 ], [ 5, 100, 5 ] ]
gap> CaratZClassNumber(DirectProductMatrixGroup(
> [MatGroupZClass(4,32,1,2),Group([[[1]]])]));
[ 5, 773, 3 ]
gap> CaratZClassNumber(DirectProductMatrixGroup(
> [MatGroupZClass(4,32,1,2),Group([[[-1]]])]));
[ 5, 664, 1 ]
gap> M2:=DirectProductMatrixGroup([MatGroupZClass(4,32,1,2),Group([[[-1]]])]);
<matrix group with 3 generators>
gap> Filtered(List(ConjugacyClassesSubgroups(M2),Representative),
> x->Order(PartialMatrixGroup(x,[1..4]))=8);
[ <matrix group of size 8 with 3 generators>,
  <matrix group of size 8 with 3 generators>,
  <matrix group of size 8 with 3 generators>,
  <matrix group of size 8 with 3 generators>,
  <matrix group of size 16 with 4 generators> ]
gap> Set(last,CaratZClassNumber);
[ [ 5, 664, 1 ], [ 5, 773, 3 ], [ 5, 774, 3 ] ]
gap> CaratZClassNumber(DirectProductMatrixGroup(
> [MatGroupZClass(4,32,3,2),Group([[[1]]])]));
[ 5, 672, 1 ]
gap> CaratZClassNumber(DirectProductMatrixGroup(
> [MatGroupZClass(4,32,3,2),Group([[[-1]]])]));
[ 5, 721, 1 ]
gap> M3:=DirectProductMatrixGroup([MatGroupZClass(4,32,3,2),Group([[[-1]]])]);
<matrix group with 3 generators>
gap> Filtered(List(ConjugacyClassesSubgroups(M3),Representative),
> x->Order(PartialMatrixGroup(x,[1..4]))=16);
[ <matrix group of size 16 with 4 generators>,
  <matrix group of size 16 with 4 generators>,
  <matrix group of size 16 with 4 generators>,
  <matrix group of size 16 with 4 generators>,
  <matrix group of size 32 with 5 generators> ]
gap> Set(last,CaratZClassNumber);
[ [ 5, 672, 1 ], [ 5, 673, 1 ], [ 5, 674, 1 ], [ 5, 675, 1 ], [ 5, 721, 1 ] ]
gap> CaratZClassNumber(DirectProductMatrixGroup(
> [MatGroupZClass(4,33,3,1),Group([[[1]]])]));
[ 5, 691, 1 ]
gap> CaratZClassNumber(DirectProductMatrixGroup(
> [MatGroupZClass(4,33,3,1),Group([[[-1]]])]));
[ 5, 730, 1 ]
gap> M4:=DirectProductMatrixGroup([MatGroupZClass(4,33,3,1),Group([[[-1]]])]);
<matrix group with 4 generators>
gap> Filtered(List(ConjugacyClassesSubgroups(M4),Representative),
> x->Order(PartialMatrixGroup(x,[1..4]))=24);
[ <matrix group of size 24 with 4 generators>,
  <matrix group of size 48 with 5 generators> ]
gap> Set(last,CaratZClassNumber);
[ [ 5, 691, 1 ], [ 5, 730, 1 ] ]
gap> CaratZClassNumber(DirectProductMatrixGroup(
> [MatGroupZClass(4,33,6,1),Group([[[1]]])]));
[ 5, 734, 1 ]
gap> CaratZClassNumber(DirectProductMatrixGroup(
> [MatGroupZClass(4,33,6,1),Group([[[-1]]])]));
[ 5, 776, 1 ]
gap> M5:=DirectProductMatrixGroup([MatGroupZClass(4,33,6,1),Group([[[-1]]])]);
<matrix group with 5 generators>
gap> Filtered(List(ConjugacyClassesSubgroups(M5),Representative),
> x->Order(PartialMatrixGroup(x,[1..4]))=48);
[ <matrix group of size 48 with 5 generators>,
  <matrix group of size 48 with 5 generators>,
  <matrix group of size 96 with 6 generators> ]
gap> Set(last,CaratZClassNumber);
[ [ 5, 733, 1 ], [ 5, 734, 1 ], [ 5, 776, 1 ] ]
gap> G:=CaratMatGroupZClass(5,100,11);
<matrix group with 2 generators>
gap> Display(G.1);
[ [   0,   1,   0,   0,   0 ],
  [   1,   0,   0,   0,   0 ],
  [   0,   0,   0,   0,   1 ],
  [   0,   0,   0,  -1,   0 ],
  [   0,   0,   1,   0,   0 ] ]
gap> Display(G.2);
[ [   1,   0,   0,   0,   0 ],
  [   0,  -1,   0,   0,   0 ],
  [   0,   0,   0,   1,  -1 ],
  [   0,   0,   0,  -1,   0 ],
  [   0,   0,  -1,  -1,   0 ] ]
gap> G1:=PartialMatrixGroup(G,[1,2]);
Group([ [ [ 0, 1 ], [ 1, 0 ] ], [ [ 1, 0 ], [ 0, -1 ] ] ])
gap> G2:=PartialMatrixGroup(G,[3,4,5]);
Group([ [ [ 0, 0, 1 ], [ 0, -1, 0 ], [ 1, 0, 0 ] ],
  [ [ 0, 1, -1 ], [ 0, -1, 0 ], [ -1, -1, 0 ] ] ])
gap> CrystCatZClassNumber(G1);
[ 2, 3, 2, 1 ]
gap> CrystCatZClassNumber(G2);
[ 3, 3, 1, 3 ]
gap> StructureDescription(G1);
"D8"
gap> StructureDescription(G2);
"C2 x C2"
gap> [3,3,1,3] in N3;
true

# Step 4
gap> LoadPackage("HAP");
true
gap> GL5Q:=List([1..955],x->CaratMatGroupZClass(5,x,1));;
# all the finite subgroups of GL5(Z) up to Q-conjugate
gap> nb5:=Filtered(GL5Q,x->BogomolovMultiplier(x)<>[]); # checking B0(G)=0
[  ]
\end{verbatim}

\vspace*{2mm}
{\it Proof of Theorem \ref{thm3}.}\vspace*{2mm}

\medskip
Step 1.
It is not difficult to show that, if $G$ is a finite $p$-group in $GL_6(\bm{Z})$, then $p=2, 3, 5$ or $7$.

Because the maximal $5$-subgroup $G_5$ (resp. $7$-subgroup $G_7$)
of $GL_6(\bm{Z})$ is of order $5$ (resp. $7$) and hence
${\rm Br}_u(\bm{C}(M)^{G_p})=0$.

We should check $2$-groups $G$ with $|G|\geq 8$
and $3$-groups $G$ with $|G|\geq 27$ which are not bicyclic.

Note that the maximal order of a $2$-group in $GL_6(\bm{Z})$ is $2^{10}$, 
and the maximal order of a $3$-group in $GL_6(\bm{Z})$ is $81$. 

The total numbers of non-bicyclic subgroups $G$ in $GL_6(\bm{Z})$ of order
$2^i$ where $3 \le i \le 10$ are: $2708$, $11198$, $14261$, $8290$, $3008$, $868$, $128$, $4$ respectively (in the order $i=3$ to $i=10$). The number of non-bicyclic subgroups $G$ in $GL_6(\bm{Z})$ of order
$27$ is $10$, while the number for those of order $81$ is $3$.

Applying the function {\tt H2nrM(G)} to the above groups, we find that exactly $2$ (resp. $0$) groups $G$ of order $27$ (resp. $81$) 
satisfy $H_u^2(G,M)\neq 0$. 
Hence the total number of non-bicyclic $3$-groups $G$ in $GL_6(\bZ)$ with 
$H_u^2(G,M)\neq 0$ is $2$. 

Similarly, we find that exactly $895$ groups which are $2$-groups of order $\le 256$ and satisfying that $H_u^2(G,M)\neq 0$; more precisely, the numbers of such groups of order $2^i$ where $3 \le i \le 8$ are $96$, $346$, $347$, $95$, $11$, $0$ respectively. We may also apply the function {\tt H2nrM(G)} to groups of order $2^9= 512$ or $2^{10}= 1024$, but it takes very long computation time and requires huge memory sources for a personal computer. Thus we take a detour for these groups.

In the following we will show that there is no non-bicyclic groups of order $512$ or $1024$ satisfying that $H_u^2(G,M)\neq 0$.

\medskip
We consider the case $G$ is of order $2^9=512$ first. There are $128$ such groups in total.

Case 1.1 It is possible that we can choose a subgroup $N$ of $G$ satisfying the conditions: (i) $[G:N]=2$ and $M^N=0$, (ii) the cohomology group $H^1(N,M)$ is an elementary abelian $2$-group (thus it may be regarded as a finite-dimensional vector space over the finite field $\bm{F}_2$ with $G/N \simeq C_2$ acting on $H^1(N,M)$), and (iii) the Jordan normal form of the generator of $G/N$ acting on $H^1(N,M)$ consists of a direct sum of several copies of the block
$(\begin{smallmatrix}1&1\\0&1\end{smallmatrix})$. We remark that there are $60$ groups with the above properties.

By Theorem \ref{thDHW} we have
the isomorphisms $H^1(G,M)\simeq H^1(N,M)^{G/N}$ and
$H^2(G,M)_1=\fn{Ker}(H^2(G,M)\xrightarrow{\rm res}H^2(N,M))\simeq H^1(G/N,H^1(N,M))$. Since the Jordan normal form of the generator of $G/N$ acting on $H^1(N,M)$ is of the required form, it follows that dim$_{\bm{F}_2} H^1(N,M)= 2 \cdot$dim$_{\bm{F}_2} H^1(G,M)$ (remember that $H^1(G,M)\simeq H^1(N,M)^{G/N}$).

By a straight-forward computation, we find $H^1(G/N,H^1(N,M))=0$. From $H^2(G,M)_1 \simeq H^1(G/N,H^1(N,M))$, we find that $H^2(G,M)_1=0$, i.e. the restriction map $H^2(G, M) \to H^2(N, M)$ is injective. It follows that the restriction map $H_u^2(G,M) \to H_u^2(N, M)$ is also injective.

On the other hand, if $N$ is of order $2^8$, we have shown that $H_u^2(N, M)= 0$. Thus $H_u^2(G,M)=0$.

Case 1.2. $G$ doesn't satisfy the conditions in Case 1.1. There are $68$ such groups.

We apply the function {\tt H2nrM(G)} to these groups. Fortunately the GAP computation is feasible for these groups. We find that $H_u^2(G,M)=0$.

\medskip
Now consider groups $G$ with order $2^{10}=1024$. There are only 4 such groups.

We search the index 2 subgroup $N$ of $G$ as in Case 1.1. Two groups $G$ possess the required properties. We proceed as in Case 1.1 and finish the proof. For the remaining two groups $G$, just apply the function {\tt H2nrM(G)} to them, because the GAP computation is feasible.

\medskip
In summary, the total number of non-bicyclic $2$-groups (resp. $3$-groups) in $GL_6(\bm{Z})$ satisfying that $H_u^2(G,M)\neq 0$ is $895$ (resp. $2$).

\medskip
Step 2. Let $G$ be a finite group in $GL_6(\bm{Z})$ which is not a $p$-group. As before, we should check whether the $2$-Sylow subgroup $G_2$ (resp. $3$-Sylow subgroup $G_3$) of $G$
is one of the $895$ (resp. $2$) groups obtained in Step 1.

Using {\tt H2nrM(G)} for groups $G$ of order $27m'$ ($m' = 2^d$ is some integer $d$),
we get $H_u^2(G,M)=0$ except for $2$ groups
$G$ of the CARAT ID $(6,2899,3)$, $(6,2899,5)$ ; these two groups are isomorphic to $(C_9\rtimes C_3)\rtimes C_2$ as abstract groups.

\medskip
Next applying {\tt H2nrM(G)} to groups $G$
of order $8m$ (resp. $16m$, $128m$), where $m$ is an odd integer $\ge 3$, we have
$H_u^2(G,M)=0$ except for $35$ (resp. $58$, $1$) groups $G$. Note that there is no group of order $256m$ where $m \ge 3$ is an odd integer.

For the case where $|G|=32m$, $64m$, we need
additional argument treated below (as in Step 1 due to the computational feasibility again).

We treat the situation that $|G|=32m$ ($m \ge 3$ is an odd integer) first.

There are $10216$ such groups (i.e. $10216$ $\bZ$-classes) and $9$ possibilities for the group orders of $G$. Check $2$-Sylow subgroups of these groups and use the results in Step 1. We find that, if $G$ is a ``possibly bad" group, then the order of $G$ should be $32 \cdot 3$, $32 \cdot 9$, $32 \cdot 27$, $32 \cdot 81$. Although groups of order $32 \cdot 27$ have been treated above, we will check it once again.

Case 2.1. $27 \mid |G|$, i.e. $|G|=32 \cdot 27$ or $32 \cdot 81$.

By similar technique as in Case 1.1, We can find a subgroup $N$ of $G$ with index $[G:N]=2$
which satisfies (i) $M^N=0$ and (ii) $H^1(N,M)=0$ or $\bZ/3\bZ$.
Applying Theorem \ref{thDHW},
we have
$H^2(G,M)_1={\rm Ker}(H^2(G,M)\xrightarrow{\rm res}H^2(N,M))
\simeq H^1(G/N,H^1(N,M))$.
However the latter group is zero because $\gcd(|G/N|,|H^1(N,M)|)=1$.
This implies $H_u^2(G,M)=0$ (see also Step 1).

Case 2.2. $27$ ${\not|}$ $|G|$, i.e. $|G|=32 \cdot 3$ or $32 \cdot 9$.

We simply apply {\tt H2nrM(G)} and obtain the result $H_u^2(G,M)$. Then there are precisely $46$ groups with non-trivial $H_u^2(G,M)$: the number of groups of order $32 \cdot 3$ is $39$, and the number of groups of order $32 \cdot 9$ is $7$.

\medskip

Now for the case $|G|=64m$ ($m \ge 3$ is an odd integer). The proof is almost the same as the situation $|G|=32m$.

In short, if $|G|=64m$ ($m \ge 3$ is an odd integer), then there are $5107$ such groups (i.e. $5107$ $\bZ$-classes). Check $2$-Sylow subgroups of these groups and use the results in Step 1. We find that, if $G$ is a ``possibly bad" group, then the order of $G$ should be $64 \cdot 3$, $64 \cdot 9$ or $64 \cdot 27$. Groups of order $64 \cdot 27$ have been treated before and we know that $H_u^2(G,M)=0$. For groups of order $64 \cdot 3$ or $64 \cdot 9$, apply {\tt H2nrM(G)}. Then there are precisely $12$ groups with non-trivial $H_u^2(G,M)$: the number of groups of order $64 \cdot 3$ is $11$, and the number of groups of order $64 \cdot 9$ is $1$.

\medskip
In summary, the total number of groups $G$ with $H_u^2(G,M)\neq 0$
and are of order $24$, $48$, $54$, $72$, $96$, $144$, $192$, $288$, $384$, $576$ is: $34$, $51$, $2$, $1$, $39$, $7$, $11$, $7$, $1$, $1$ respectively. In other words, we find exactly $154$ non-$p$-groups $G$ satisfying that  $H_u^2(G,M)\neq 0$.
We find that there are, in total, $1051=895+2+154$ subgroups $G$ in $GL_6(\bm{Z})$ with $H_u^2(G,M)\neq 0$.

\medskip
Step 3.
Using the function {\tt LatticeDecompositions(n)} provided by \cite[Section 4.1]{HY} (also see Step 2 in the proof of Theorem \ref{thm2}), we find that, among the $1051$ lattices in Step 2, the numbers of lattices are $603$ (indecomposable lattices), $161$ (decomposable into $5+1$), $218$ (decomposable into $4+2$), $59$ (decomposable into $4+1+1$), $5$ (decomposable into $3+3$), $5$ (decomposable into $3+2+1$).

\medskip
Step 4.
We can check whether $H_u^2(G,\bm{Q}/\bm{Z})=0$
using the function {\tt B0G(G)}
in the program b0g.g by Moravec \cite{Mo} or the function
{\tt BogomolovMultiplier(G)} in HAP \cite{HAP}
(see also Step 4 in the proof of Theorem \ref{thm1}).
We see that $H_u^2(G,\bm{Q}/\bm{Z})\neq 0$ if and only if
CARAT ID of $G$ is one of the $24$
$(6,6458,i)$, $(6,6459,i)$, $(6,6464,i)$ $(1\leq i\leq 8)$.
The corresponding $G$-lattice $M$ is indecomposable
(resp. decomposable into indecomposable component with rank $4+2$)
for $i=2,4,6,8$ (resp. $i=1,3,5,7$).
{}From Steps $1$-$3$, we find that
these $24$ $G$-lattices $M$ satisfy $H_u^2(G,M)=0$
except for the two groups whose CARAT ID are $(4,6459,8)$, $(4,6464,6)$ (with $H_u^2(G,M)=\bm{Z}/2\bm{Z}$ for both groups).

We conclude that there exist $1073=1051+24-2$
subgroups $G$ of $GL_6(\bm{Z})$ with ${\rm Br}_u(\bm{C}(M)^G)\neq 0$ where the $1051$ groups are those obtained in Step 2, the $24$ groups are obtained in this step while $2$ groups correspond to the groups with CARAT ID $(4,6459,8)$ and $(4,6464,6)$ respectively.\\

Before proceeding to the computations of GAP \cite{GAP} below,
we should read the related data from\\
{\tt https://www.math.kyoto-u.ac.jp/\~{}yamasaki/Algorithm/MultInvField/}.
\footnote{
We need 
CARAT package \cite{CARAT} of GAP for  
$G\leq {\rm GL}(n,\bZ)$ $(n=5, 6)$ and it 
works on Linux or macOS, but not on Windows.}
\\



\bigskip
As mentioned before, multiplicative invariant fields arising from distinct lattices may be stably isomorphic. 
Note that, in Theorem \ref{equiv}, the actions of (1), (2), (3) 
are multiplicative (purely monomial) as in Definition \ref{d1.2} but 
the action of (4) is not multiplicative but twisted multiplicative (monomial) 
as in Definition \ref{d1.3} and the action of (5) is neither multiplicative 
nor twisted multiplicative. 

\begin{theorem}\label{equiv}
The following fields $K$ are stably isomorphic each other:\\
{\rm (1)} $\bm{C}(G)$ where $G$ is a group of order $64$ which belongs to the $16$th isoclinism class $\Phi_{16}$ $($see the $9$ groups defined in Lemma 5.5 of \cite{CHKK}$)$;\\
{\rm (2)} $\bm{C}(x_1,x_2,x_3,x_4)^{D_4}$
where $D_4=\langle\sigma,\tau\rangle$ acts on
$\bm{C}(x_1,x_2,x_3,x_4)$ by 
\begin{align*}
&\sigma: x_1\mapsto x_2x_3, x_2\mapsto x_1x_3, x_3\mapsto x_4, x_4\mapsto \frac{1}{x_3},\\
&\tau: x_1\mapsto \frac{1}{x_2}, x_2\mapsto\frac{1}{x_1},
x_3\mapsto\frac{1}{x_4}, x_4\mapsto\frac{1}{x_3}.
\end{align*}
{\rm (3)} $\bm{C}(y_1,y_2,y_3,y_4,y_5)^{D_4}$ where $D_4=\langle\sigma,\tau\rangle$ acts on $\bm{C}(y_1,y_2,y_3,y_4,y_5)$ by
\begin{align*}
&\sigma: y_1\mapsto y_2, y_2\mapsto y_1, y_3\mapsto \frac{1}{y_1y_2y_3},
y_4\mapsto y_5, y_5\mapsto \frac{1}{y_4},\\
&\tau: y_1\mapsto y_3, y_2\mapsto \frac{1}{y_1y_2y_3}, y_3\mapsto y_1,
y_4\mapsto y_5, y_5\mapsto y_4.
\end{align*}
{\rm (4)} $\bm{C}(z_1,z_2,z_3,z_4)^{C_2\times C_2}$ where
$C_2\times C_2=\langle\sigma,\tau\rangle$ acts on
$\bm{C}(z_1,z_2,z_3,z_4)$ by
\begin{align*}
&\sigma: z_1\mapsto z_2, z_2\mapsto z_1, z_3\mapsto \frac{1}{z_1z_2z_3},
z_4\mapsto \frac{-1}{z_4},\\
&\tau: z_1\mapsto z_3, z_2\mapsto \frac{1}{z_1z_2z_3}, z_3\mapsto z_1,
z_4\mapsto -z_4.
\end{align*}
{\rm (5)} $\bm{C}(w_1,w_2,w_3,w_4)^{C_2}$ where
$C_2=\langle\sigma\rangle$ acts on
$\bm{C}(w_1,w_2,w_3,w_4)$ by
\begin{align*}
\sigma: w_1\mapsto -w_1, w_2\mapsto \frac{w_4}{w_2},
w_3\mapsto \frac{(w_4-1)(w_4-w_1^2)}{w_3}, w_4\mapsto w_4.
\end{align*}
In particular, the unramified cohomology groups
$H_u^i(K,\bm{Q}/\bm{Z})$ of the fields $K$
in {\rm (1)}--{\rm (5)} coincide and 
${\rm Br}_u(K)\simeq \bm{Z}/2\bm{Z}$ 
and hence $K$ is not retract $k$-rational.
\end{theorem}
\begin{proof}
The equivalent (1) $\Leftrightarrow$ (3) $\Leftrightarrow$ (4)
$\Leftrightarrow$ (5) is already known
(see \cite[Theorems 6.2, 6.3 and 6.4, and their proofs]{HKK}).
We will show that (2) $\Leftrightarrow$ (4). Let $x_1, x_2, x_3, x_4$ be given in (2).
Define
$u_1=\frac{x_1x_4(x_3+1)}{x_4+1}$,
$u_2=\frac{x_2}{x_1}$,
$u_3=\frac{x_3+1}{x_3-1}\frac{x_4+1}{x_4-1}$,
$u_4=\frac{x_3+1}{x_3-1}/\frac{x_4+1}{x_4-1}$.
Because of $\sigma^2: x_1\mapsto x_1x_3x_4$,
$x_2\mapsto x_2x_3x_4$,
$x_3\mapsto\tfrac{1}{x_3}$, $x_4\mapsto\tfrac{1}{x_4}$,
we have $\bm{C}(x_1,x_2,x_3,x_4)^{\langle\sigma^2\rangle}
=\bm{C}(u_1,u_2,u_3,u_4)$
and the actions of $\sigma$ and $\tau$ on $\bm{C}(u_1,u_2,u_3,u_4)$
are given by
\begin{align*}
\sigma:\ &u_1\mapsto\frac{u_1u_2(u_3u_4-1)}{(u_3-u_4)u_4},
u_2\mapsto \frac{1}{u_2}, u_3\mapsto -u_3, u_4\mapsto -\frac{1}{u_4},\\
\tau:\ &u_1\mapsto \frac{1}{u_1u_2}, u_2\mapsto u_2,
u_3\mapsto u_3, u_4\mapsto\frac{1}{u_4}.
\end{align*}
Define
\begin{align*}
z_1=u_1,
z_2=\frac{u_1u_2(u_3u_4-1)}{(u_3-u_4)u_4},
z_3=\frac{1}{u_1u_2},
z_4=\frac{u_4+1}{u_4-1}.
\end{align*}
Then $\bm{C}(u_1,u_2,u_3,u_4)=\bm{C}(z_1,z_2,z_3,z_4)$
and the actions of $\sigma$ and $\tau$ on $\bm{C}(z_1,z_2,z_3,z_4)$
are exactly the same as in (4).
Note that the unramified Brauer groups are isomorphic for stably isomorphic fields by \cite[Proposition 2.2]{Sa4} or \cite[Proposition 1.2]{CTO}. The last statement follows from Theorem \ref{t1.5} (2).
\end{proof}

\section{Classification of elementary abelian groups $(C_2)^k$ in $GL_n(\bZ)$ with $n\leq 7$}\label{seC}

After we finished the proof of Theorem \ref{t1.5}, we examined the groups obtained in Section \ref{seTables}. It was found that the elementary abelian $2$-group $(C_2)^3$ didn't appear on the lists. But Theorem \ref{thB1} tells us that some faithful lattice $M$ of this group will produce an example that the unramified Brauer group of $\bm{C}(M)^G$ is non-trivial where $G=(C_2)^3$. So we searched $G$-lattices of rank $7$. To our surprise, such lattices had not been explored before. This led to the study of elementary abelian $2$-groups in $GL_n(\bZ)$ with $n\leq 7$ and the associated unramified Brauer groups.

In this section, we will classify
elementary abelian groups $(C_2)^k$ in $GL_n(\bZ)$ with $n\leq 7$; the study of unramified Brauer groups is postponed till the next section.

\medskip
First we recall some terminology. Two finite subgroups in $GL_n(\bZ)$ are called {\it $\bQ$-conjugate} if they are conjugate by some matrix in $GL_n(\bQ)$; they are {\it $\bZ$-conjugate} if they are conjugate by some matrix in $GL_n(\bZ)$. Thus $\bQ$-conjugation defines an equivalence relation on the set of finite subgroups of $GL_n(\bZ)$; an equivalence class is called a {\it $\bQ$-class} in $GL_n(\bZ)$. Similarly for a {\it $\bZ$-class} in $GL_n(\bZ)$.

Let $G_{2^n}$ be the subgroup in $GL_n(\bZ)$ generated by the $n$ diagonal matrices
${\rm Diag}(-1,1,\ldots,1)$, ${\rm Diag}(1,-1,1,$ $\ldots,1),\ldots, {\rm Diag}(1,\ldots,1,-1)$. We will show that any finite subgroup $G$ in $GL_n(\bZ)$ with $G \simeq (C_2)^k$ ($k$ is some positive integer) is $\bQ$-conjugate to a subgroup of $G_{2^n}$. In fact, all the matrices of $G$ can be diagonalized simultaneously within $GL_n(\bQ)$. In other words, $G$ is $\bQ$-conjugate to a subgroup consisting of diagonal matrices, which is just some subgroup of $G_{2^n}$.

The main idea of the proof of Theorem \ref{t5.1} is to show that $\bQ$-conjugation is the same as $S_n$-equivalence (for the definition of $S_7$-equivalence see Step 2 in the proof of Theorem \ref{t5.1}). The general result for any positive integer $n$ is proved in Theorem \ref{t5.2} which relies on a theorem proved in coding theory \cite{BBFKKW}. Although no essential idea of error-correcting codes is used in the proof of Theorem 6.8.4 of \cite[page 551]{BBFKKW}, we choose to present Theorem \ref{t5.1} for $n=7$ along the traditional mathematical arguments. It has a bonus that the proof provides for each $\bQ$-class a set of complete invariants, the trace vector $L_G$ and the numerical invariants $T^1_G$ and $T^3_G$: Any two elementary abelian $2$-groups $G_1$ and $G_2$ in $GL_7(\bZ)$ are $\bQ$-conjugate if and only if they have the same invariants (see Steps $3$ and $4$ in the proof of Theorem \ref{t5.1}).

\begin{theorem}\label{t5.1}~\\
{\rm (1)} There exist
$3$ $($resp. $2$$)$ $\bZ$-classes
forming
$2$ $($resp. $1$$)$ $\bQ$-classes
of groups
$(C_2)^k$ in $GL_2(\bZ)$ for $k=1$ $($resp. $2$$)$.\\
{\rm (2)} There exist
$5$ $($resp. $11$, $4$$)$ $\bZ$-classes
forming
$3$ $($resp. $3$, $1$$)$ $\bQ$-classes
of groups
$(C_2)^k$ in $GL_3(\bZ)$ for $k=1$ $($resp. $2$, $3$$)$.\\
{\rm (3)} There exist
$8$ $($resp. $37$, $39$, $8$$)$ $\bZ$-classes
forming
$4$ $($resp. $6$,  $4$, $1$$)$ $\bQ$-classes
of groups
$(C_2)^k$ in $GL_4(\bZ)$ for $k=1$ $($resp. $2$, $3$, $4$$)$.\\
{\rm (4)} There exist
$11$ $($resp. $99$, $263$, $138$, $16$$)$ $\bZ$-classes
forming
$5$ $($resp. $10$,  $10$, $5$, $1$$)$ $\bQ$-classes
of groups
$(C_2)^k$ in $GL_5(\bZ)$ for $k=1$ $($resp. $2$, $3$, $4$, $5$$)$.\\
{\rm (5)} There exist
$15$ $($resp. $255$, $1649$, $1947$, $511$, $36$$)$ $\bZ$-classes
forming
$6$ $($resp. $16$,  $22$, $16$, $6$, $1$$)$ $\bQ$-classes
of groups
$(C_2)^k$ in $GL_6(\bZ)$ for $k=1$ $($resp. $2$, $3$, $4$, $5$, $6$$)$.\\
{\rm (6)} There exist
$19$ $($resp. $608$, $10645$, $29442$, $15248$, $2016$, $80$$)$ $\bZ$-classes
forming
$7$ $($resp. $23$,  $43$, $43$, $23$, $7$, $1$$)$ $\bQ$-classes
of groups
$(C_2)^k$ in $GL_7(\bZ)$ for $k=1$ $($resp. $2$, $3$, $4$, $5$, $6$, $7$$)$.
\end{theorem}

\begin{proof}
(1)--(5): $2\leq n\leq 6$.
We can use the CARAT package \cite{CARAT}
of GAP which has a database of all $\bZ$-classes and
$\bQ$-classes of finite subgroup of $GL_n(\bZ)$ $(n\leq 6)$ up to conjugacy.
For GAP ID and CARAT ID, see Section \ref{seCarat}. 

(6): $n=7$. No database for rank $n=7$ in CARAT was created. We should work out it ourselves. Note that the main job is to find $\bQ$-classes. Once a $\bQ$-class is find, we may use the CARAT function {\tt ZClassRepsQClass} of GAP to split the $\bQ$-class into a union of $\bZ$-classes. So we will concentrate to find all the $\bQ$-classes of subgroups $\simeq (C_2)^k$ in the following steps.

Step 1.
Let $G_{128}$ be the diagonal subgroup of $GL_7(\bZ)$ of order $128$ which is generated by the diagonal matrices ${\rm Diag}(-1,1,1,1,1,1,1)$, ${\rm Diag}(1,-1,1,1,1,1,1),\ldots, {\rm Diag}(1,1,1,1,1,1,-1)$. 

Via the group isomorphism $\{-1, 1 \} \to \bF_2$, subgroups $G_{128}$ correspond bijective to vector subspaces of $(\bF_2)^7$. It is not difficult to determine the number of subgroups of $G_{128}$ with order $2^k$ (because we may determine that for the vector subspaces of $(\bF_2)^7$); for examples, the number of subgroups of order $4$ is $127 \cdot 126/((4-1)(4-2))= 2667$. It follows that there exist $127$
(resp. $2667$, $11811$, $11811$, $2667$, $127$) subgroups in $G_{128}$, which are
of order $2$ (resp. $4$, $8$, $16$, $32$, $64$).

\medskip
Step 2.

Consider the equivalence relation defined on the set of all subgroups of elementary abelian $2$-groups of $GL_7(\bZ)$. Since every such subgroup is $\bQ$-conjugate to some subgroup of $G_{128}$, we may as well consider the restriction of this equivalence relation to the set of all subgroups of $G_{128}$. Clearly the numbers of equivalence classes of these two equivalence relations are the same.

Now we consider a coarser equivalence relation on the set of subgroups of $G_{128}$. Let $S_7$ be the subgroup of $GL_7(\bZ)$ consisting of all the permutation matrices; $S_7$ is isomorphic to the symmetric group in $7$ letters as an abstract group. Two subgroups $G_1, G_2\subset G_{128}$ are called $S_7$-equivalent if there is some element $t \in S_7$ such that $G_2= t \cdot G_1 \cdot t^{-1}$. Note that $S_7$-equivalence implies $\bQ$-conjugation, but the converse is not obvious. The advantage of considering $S_7$-equivalence is that, (i) the group $S_7$ is a finite group, and (ii) for any $t \in S_7$, any subgroup $G$ of $G_{128}$, the group $t \cdot G \cdot t^{-1}$ is always a subgroup of $G_{128}$. Thus we may find all the $S_7$-equivalence classes easily with the help of computers. In Step 3 and Step 4 we will show that $\bQ$-conjugation for subgroups of $G_{128}$ is necessarily $S_7$-equivalent.

\medskip
Step 3.

For a group $G\subset G_{128}$, define the trace vector $L_G$ associated to $G$. Explicitly, $L_G=[ t_{-7}, t_{-5}, t_{-3}, t_{-1}, t_1, t_3, t_5, t_7 ]$
$\in \bm{Z}^8$ as follows: $t_m$
is the number of elements in $G$ whose trace is $m$. Note that the trace of an element in $G$ is always an odd integer; moreover, if two groups are $\bQ$-conjugate, then they have the same trace vector.

We claim that, if $G_1, G_2$ are subgroups of $G_{128}$ isomorphic to $(C_2)^k$ ($1 \le k\le 6$), then the following two statements are equivalent, (i) $G_1$ and $G_2$ are $S_7$-equivalent, (ii) $G_1$ and $G_2$ are $\bQ$-conjugate. In fact, we will prove this by showing that, in ``most" situations (including the cases $k \neq 3,4$), (i) and (ii) are equivalent to: (iii) $L_{G_1}= L_{G_2}$.

Clearly, (i) $\Rightarrow$ (ii), and (ii) $\Rightarrow$ (iii). It remains to show that (iii) $\Rightarrow$ (i) in ``most" situations, and to show that (ii) $\Rightarrow$ (i) in the remaining situations.

\medskip
Let $X(k)$ be the set of all subgroups of $G_{128}$ isomorphic to $(C_2)^k$ ($1 \le k \le 6$ and $k\neq 3, 4$). We define the $L_G$-equivalence on $X(k)$: If $G_1, G_2 \in X(k)$, $G_1$ and $G_2$ are $L_G$-equivalent if $L_{G_1}= L_{G_2}$. Thus $X(k)$ is partitioned into a finite disjoint union of $L_G$-equivalence classes; for simplicity, we call these $L_G$-equivalence classes the $L_G$-orbits. Clearly, each $L_G$-orbit is stable under the conjugation actions of $S_7$; namely, each $L_G$-orbit is a union of some $S_7$-equivalence classes. By computer computation, each $L_G$-orbit consists of only one $S_7$-equivalence class. This finishes the proof.

\medskip
Step 4.
The situation for subgroup $\simeq (C_2)^3$ is slightly different and the proof may be modified as follows (the proof for subgroup $\simeq (C_2)^4$ is similar and is omitted).

As before, let $X(3)$ be the subgroups $\simeq (C_2)^3$ in $G_{128}$ and define the $L_G$-orbits as in Step 3. Computer computation shows that there are forty-one $L_G$-orbits. Moreover, each $L_G$-orbit consists of only one $S_7$-equivalence class except two $L_G$-orbits. Thus the proof of (iii) $\Rightarrow$ (i) is finished for these good $L_G$-orbits. The exceptional cases are the $L_G$-orbits with trace vector $L_G= [0,1,2,1,2,1,0,1]$ and $L_G=[0,1,0,3,0,3,0,1]$.

\medskip
We will discuss the $L_G$-orbit with $L_G= [0,1,2,1,2,1,0,1]$ first.

{}From the trace vector, we know that, if $G$ belongs to this $L_G$-orbit, 
then there are precisely two distinct elements $x, y \in G$ with 
${\rm trace}(x)={\rm trace}(y)=1$ (where ${\rm trace}(x)$ 
denotes the trace of the matrix $x$). 
Define another invariant of the group $G$. 
Define $T_G ^1={\rm trace}(xy)$. 
Clearly $T_G ^1$ is the same for all $\bQ$-conjugate groups. 
In fact, if $G$ is a groups in this $L_G$-orbit, 
$T_G ^1$ is $3$ or $-1$. 
Thus the $L_G$-orbit is the disjoint union of two subsets 
$Y_1$ and $Y_2$ where each group $G$ in $Y_1$ (resp. $Y_2$) has 
$T_G ^1$ equal to $3$ (resp. $-1$). 
We find that $Y_i$ is a union of $\bQ$-classes for $i=1,2$. 
Now $S_7$ acts on elements of $Y_i$ by conjugation. 
By computer computation, $S_7$ is transitive on $Y_i$. 
In other words, $Y_i$ is just a single $\bQ$-class 
and is also a single $S_7$-equivalence class. 
This finishes the proof of (ii) $\Rightarrow$ (i).

\medskip
Now consider the $L_G$-orbit with trace vector $L_G=[0,1,0,3,0,3,0,1]$. 
If $G$ belongs to this $L_G$-orbit, there are precisely three distinct 
elements $x, y, z \in G$ with 
${\rm trace}(x)={\rm trace}(y)={\rm trace}(z)=3$. 
Note that the traces of $xy, xz, yz$ are the same. 
Define $T_G ^3={\rm trace}(xy)$. 
Note that $T_G ^3$ is the same for all $\bQ$-conjugate groups; 
it is either $3$ or $-1$. 
Thus this $L_G$-orbit is the disjoint union of two subsets 
$Z_1$ and $Z_2$. 
As before, each of $Z_i$ is a single $\bQ$-class and is also 
a single $S_7$-equivalence class. Done.

\medskip
Now we may understand $\bQ$-classes of subgroups isomorphic to $(C_2)^k$ via $S_7$-equivalence classes of subgroups of $G_{128}$.

Conclusion: There exist
$7$ $($resp. $23$,  $43$, $43$, $23$, $7$, $1$$)$ $\bQ$-classes
of groups
$(C_2)^k$ in $GL_7(\bZ)$ for $k=1$ $($resp. $2$, $3$, $4$, $5$, $6$, $7$$)$
with invariants $L_G$, $T^1_G$ and $T^3_G$ (see Table $4$).
\end{proof}
\newpage


Table $4$: All $\bZ$-classes of $(C_2)^k$ of rank $7$
with complete invariants $L_G$, $T^1_G$, $T^3_G$.\vspace*{2mm}\\
\begin{tabular}{lr}
$L_G$ with $G\simeq C_2$
& \# $\bm{Z}$-classes \\\hline 
$[ 1, 0, 0, 0, 0, 0, 0, 1 ]$ & 1 \\
$[ 0, 1, 0, 0, 0, 0, 0, 1 ]$ & 2 \\
$[ 0, 0, 1, 0, 0, 0, 0, 1 ]$ & 3 \\
$[ 0, 0, 0, 1, 0, 0, 0, 1 ]$ & 4 \\
$[ 0, 0, 0, 0, 1, 0, 0, 1 ]$ & 4 \\
$[ 0, 0, 0, 0, 0, 1, 0, 1 ]$ & 3 \\
$[ 0, 0, 0, 0, 0, 0, 1, 1 ]$ & 2 \\
\hline
Total & 19 
\\
\\
$L_G$ with $G\simeq (C_2)^2$
& \# $\bm{Z}$-classes \\\hline
$[ 1, 1, 0, 0, 0, 0, 1, 1 ]$ & 2 \\
$[ 1, 0, 1, 0, 0, 1, 0, 1 ]$ & 3 \\
$[ 1, 0, 0, 1, 1, 0, 0, 1 ]$ & 4 \\
$[ 0, 1, 1, 0, 1, 0, 0, 1 ]$ & 12 \\
$[ 0, 1, 1, 0, 0, 0, 1, 1 ]$ & 7 \\
$[ 0, 1, 0, 1, 0, 1, 0, 1 ]$ & 12 \\
$[ 0, 1, 0, 2, 0, 0, 0, 1 ]$ & 13 \\
$[ 0, 1, 0, 0, 2, 0, 0, 1 ]$ & 13 \\
$[ 0, 2, 0, 0, 0, 1, 0, 1 ]$ & 6 \\
$[ 0, 0, 1, 1, 1, 0, 0, 1 ]$ & 82 \\
$[ 0, 0, 1, 1, 0, 0, 1, 1 ]$ & 12 \\
$[ 0, 0, 1, 0, 1, 1, 0, 1 ]$ & 26 \\
$[ 0, 0, 2, 1, 0, 0, 0, 1 ]$ & 20 \\
$[ 0, 0, 2, 0, 0, 1, 0, 1 ]$ & 29 \\
$[ 0, 0, 0, 1, 1, 0, 1, 1 ]$ & 16 \\
$[ 0, 0, 0, 1, 2, 0, 0, 1 ]$ & 112 \\
$[ 0, 0, 0, 1, 0, 2, 0, 1 ]$ & 20 \\
$[ 0, 0, 0, 2, 0, 1, 0, 1 ]$ & 60 \\
$[ 0, 0, 0, 3, 0, 0, 0, 1 ]$ & 62 \\
$[ 0, 0, 0, 0, 1, 1, 1, 1 ]$ & 12 \\
$[ 0, 0, 0, 0, 2, 1, 0, 1 ]$ & 60 \\
$[ 0, 0, 0, 0, 0, 1, 2, 1 ]$ & 6 \\
$[ 0, 0, 0, 0, 0, 3, 0, 1 ]$ & 19 \\
\hline
Total & 608 
\\
\\
\\
\\
\\
\\
\\
\\
\\
\\
\\
\end{tabular}\hspace*{4mm}
\begin{tabular}{llr}
$L_G$ with $G\simeq (C_2)^3$
& & \# $\bm{Z}$-classes \\\hline
$[ 1, 1, 1, 1, 1, 1, 1, 1 ]$ & & 12\\
$[ 1, 1, 0, 2, 2, 0, 1, 1 ]$ & & 13\\
$[ 1, 2, 1, 0, 0, 1, 2, 1 ]$ & & 6\\
$[ 1, 0, 1, 2, 2, 1, 0, 1 ]$ & & 60\\
$[ 1, 0, 2, 1, 1, 2, 0, 1 ]$ & & 20\\
$[ 1, 0, 3, 0, 0, 3, 0, 1 ]$ & & 19\\
$[ 1, 0, 0, 3, 3, 0, 0, 1 ]$ & & 62\\
$[ 0, 1, 1, 1, 2, 1, 1, 1 ]$ & & 80\\
$[ 0, 1, 1, 2, 3, 0, 0, 1 ]$ & & 380\\
$[ 0, 1, 2, 1, 2, 1, 0, 1 ]$ & $T^1_G=3\ \ $ & 208\\
$[ 0, 1, 2, 1, 2, 1, 0, 1 ]$ & $T^1_G=-1$ & 107\\
$[ 0, 1, 2, 1, 0, 1, 2, 1 ]$ & & 28\\
$[ 0, 1, 2, 2, 1, 0, 1, 1 ]$ & & 80\\
$[ 0, 1, 0, 1, 4, 1, 0, 1 ]$ & & 279\\
$[ 0, 1, 0, 3, 0, 3, 0, 1 ]$ & $T^3_G=3\ \ $ & 113\\
$[ 0, 1, 0, 3, 0, 3, 0, 1 ]$ & $T^3_G=-1$  & 58\\
$[ 0, 1, 0, 5, 0, 1, 0, 1 ]$ & & 279\\
$[ 0, 2, 1, 0, 3, 1, 0, 1 ]$ & & 58\\
$[ 0, 2, 2, 0, 1, 1, 1, 1 ]$ & & 28\\
$[ 0, 2, 0, 3, 0, 2, 0, 1 ]$ & & 58\\
$[ 0, 3, 0, 1, 0, 3, 0, 1 ]$ & & 18\\
$[ 0, 0, 1, 1, 2, 2, 1, 1 ]$ & & 107\\
$[ 0, 0, 1, 2, 1, 1, 2, 1 ]$ & & 58\\
$[ 0, 0, 1, 2, 3, 1, 0, 1 ]$ & & 1919\\
$[ 0, 0, 1, 3, 2, 0, 1, 1 ]$ & & 415\\
$[ 0, 0, 1, 0, 3, 3, 0, 1 ]$ & & 221\\
$[ 0, 0, 2, 1, 2, 2, 0, 1 ]$ & & 429\\
$[ 0, 0, 2, 2, 1, 1, 1, 1 ]$ & & 208\\
$[ 0, 0, 2, 3, 2, 0, 0, 1 ]$ & & 1000\\
$[ 0, 0, 3, 2, 1, 1, 0, 1 ]$ & & 429\\
$[ 0, 0, 3, 3, 0, 0, 1, 1 ]$ & & 58\\
$[ 0, 0, 3, 0, 1, 3, 0, 1 ]$ & & 113\\
$[ 0, 0, 4, 1, 0, 2, 0, 1 ]$ & & 92\\
$[ 0, 0, 0, 1, 2, 2, 2, 1 ]$ & & 58\\
$[ 0, 0, 0, 1, 4, 2, 0, 1 ]$ & & 691\\
$[ 0, 0, 0, 1, 0, 6, 0, 1 ]$ & & 60\\
$[ 0, 0, 0, 2, 3, 1, 1, 1 ]$ & & 432\\
$[ 0, 0, 0, 3, 4, 0, 0, 1 ]$ & & 1077\\
$[ 0, 0, 0, 3, 0, 4, 0, 1 ]$ & & 221\\
$[ 0, 0, 0, 5, 0, 2, 0, 1 ]$ & & 691\\
$[ 0, 0, 0, 7, 0, 0, 0, 1 ]$ & & 269\\
$[ 0, 0, 0, 0, 1, 3, 3, 1 ]$ & & 18\\
$[ 0, 0, 0, 0, 3, 3, 1, 1 ]$ & & 113\\
\hline
Total & & 10645 
\end{tabular}

\vspace*{2mm}
\noindent
$L_G=[ t_{-7}, t_{-5}, t_{-3}, t_{-1}, t_1, t_3, t_5, t_7 ]$
$\in \bm{Z}^8$ is the list of $t_m$ which
is the number of elements of $G$ whose trace is $m$.\\
$T^i_G$ is the trace of the product of two distinct elements
with trace $i$.

\begin{tabular}{llr}
$L_G$ with $G\simeq (C_2)^4$
& & \# $\bm{Z}$-classes \\\hline
$[ 1, 1, 1, 5, 5, 1, 1, 1 ]$ & & 279 \\
$[ 1, 1, 3, 3, 3, 3, 1, 1 ]$ & $T^3_G=3\ \ $ & 113 \\
$[ 1, 1, 3, 3, 3, 3, 1, 1 ]$ & $T^3_G=-1$ & 58 \\
$[ 1, 2, 2, 3, 3, 2, 2, 1 ]$ & & 58 \\
$[ 1, 3, 3, 1, 1, 3, 3, 1 ]$ & & 18 \\
$[ 1, 0, 2, 5, 5, 2, 0, 1 ]$ & & 691 \\
$[ 1, 0, 4, 3, 3, 4, 0, 1 ]$ & & 221 \\
$[ 1, 0, 6, 1, 1, 6, 0, 1 ]$ & & 60 \\
$[ 1, 0, 0, 7, 7, 0, 0, 1 ]$ & & 269 \\
$[ 0, 1, 1, 3, 6, 3, 1, 1 ]$ & & 886 \\
$[ 0, 1, 2, 3, 4, 3, 2, 1 ]$ & & 270 \\
$[ 0, 1, 2, 5, 6, 1, 0, 1 ]$ & & 2823 \\
$[ 0, 1, 3, 3, 2, 3, 3, 1 ]$ & & 109 \\
$[ 0, 1, 3, 5, 4, 1, 1, 1 ]$ & & 1878 \\
$[ 0, 1, 4, 3, 4, 3, 0, 1 ]$ & $T^3_G=3\ \ $ & 883 \\
$[ 0, 1, 4, 3, 4, 3, 0, 1 ]$ & $T^3_G=-1$ & 670 \\
$[ 0, 1, 4, 5, 2, 1, 2, 1 ]$ & & 270 \\
$[ 0, 1, 0, 3, 8, 3, 0, 1 ]$ & & 789 \\
$[ 0, 1, 0, 7, 0, 7, 0, 1 ]$ & & 380 \\
$[ 0, 1, 0, 11, 0, 3, 0, 1 ]$ & & 789 \\
$[ 0, 2, 2, 3, 6, 2, 0, 1 ]$ & & 1170 \\
$[ 0, 2, 3, 3, 4, 2, 1, 1 ]$ & & 422 \\
$[ 0, 2, 4, 3, 2, 2, 2, 1 ]$ & & 141 \\
$[ 0, 2, 0, 9, 0, 4, 0, 1 ]$ & & 574 \\
$[ 0, 3, 2, 1, 6, 3, 0, 1 ]$ & & 144 \\
$[ 0, 3, 3, 1, 4, 3, 1, 1 ]$ & & 109 \\
$[ 0, 3, 0, 7, 0, 5, 0, 1 ]$ & & 144 \\
$[ 0, 4, 0, 5, 0, 6, 0, 1 ]$ & & 56 \\
$[ 0, 0, 1, 1, 6, 6, 1, 1 ]$ & & 380 \\
$[ 0, 0, 1, 3, 4, 4, 3, 1 ]$ & & 214 \\
$[ 0, 0, 1, 5, 6, 2, 1, 1 ]$ & & 3055 \\
$[ 0, 0, 2, 3, 6, 4, 0, 1 ]$ & & 2534 \\
$[ 0, 0, 2, 5, 4, 2, 2, 1 ]$ & & 1170 \\
$[ 0, 0, 3, 3, 4, 4, 1, 1 ]$ & & 883 \\
$[ 0, 0, 3, 7, 4, 0, 1, 1 ]$ & & 1037 \\
$[ 0, 0, 4, 1, 4, 6, 0, 1 ]$ & & 361 \\
$[ 0, 0, 4, 5, 4, 2, 0, 1 ]$ & & 2867 \\
$[ 0, 0, 5, 5, 2, 2, 1, 1 ]$ & & 670 \\
$[ 0, 0, 6, 3, 2, 4, 0, 1 ]$ & & 558 \\
$[ 0, 0, 0, 1, 4, 6, 4, 1 ]$ & & 56 \\
$[ 0, 0, 0, 3, 6, 4, 2, 1 ]$ & & 574 \\
$[ 0, 0, 0, 5, 0, 10, 0, 1 ]$ & & 145 \\
$[ 0, 0, 0, 9, 0, 6, 0, 1 ]$ & & 664 \\
\hline
Total & & 29442 
\end{tabular}\hspace*{4mm}
\begin{tabular}{lr}
$L_G$ with $G\simeq (C_2)^5$
& \# $\bm{Z}$-classes \\\hline
$[ 1, 1, 3, 11, 11, 3, 1, 1 ]$ & 789 \\
$[ 1, 1, 7, 7, 7, 7, 1, 1 ]$ & 380 \\
$[ 1, 2, 4, 9, 9, 4, 2, 1 ]$ & 574 \\
$[ 1, 3, 5, 7, 7, 5, 3, 1 ]$ & 144 \\
$[ 1, 4, 6, 5, 5, 6, 4, 1 ]$ & 56 \\
$[ 1, 0, 6, 9, 9, 6, 0, 1 ]$ & 664 \\
$[ 1, 0, 10, 5, 5, 10, 0, 1 ]$ & 145 \\
$[ 0, 1, 2, 7, 12, 7, 2, 1 ]$ & 1034 \\
$[ 0, 1, 4, 7, 8, 7, 4, 1 ]$ & 360 \\
$[ 0, 1, 6, 11, 8, 3, 2, 1 ]$ & 1860 \\
$[ 0, 1, 8, 7, 8, 7, 0, 1 ]$ & 983 \\
$[ 0, 1, 0, 15, 0, 15, 0, 1 ]$ & 206 \\
$[ 0, 2, 5, 9, 10, 4, 1, 1 ]$ & 2499 \\
$[ 0, 2, 7, 9, 6, 4, 3, 1 ]$ & 557 \\
$[ 0, 3, 4, 7, 12, 5, 0, 1 ]$ & 899 \\
$[ 0, 3, 6, 7, 8, 5, 2, 1 ]$ & 557 \\
$[ 0, 3, 0, 19, 0, 9, 0, 1 ]$ & 521 \\
$[ 0, 4, 5, 5, 10, 6, 1, 1 ]$ & 360 \\
$[ 0, 5, 0, 15, 0, 11, 0, 1 ]$ & 144 \\
$[ 0, 0, 1, 5, 10, 10, 5, 1 ]$ & 144 \\
$[ 0, 0, 3, 9, 10, 6, 3, 1 ]$ & 1183 \\
$[ 0, 0, 5, 5, 10, 10, 1, 1 ]$ & 557 \\
$[ 0, 0, 9, 9, 6, 6, 1, 1 ]$ & 632 \\
\hline
Total & 15248 
\\
\\
$L_G$ with $G\simeq (C_2)^6$
& \# $\bm{Z}$-classes \\\hline
$[ 1, 1, 15, 15, 15, 15, 1, 1 ]$ & 206 \\
$[ 1, 3, 9, 19, 19, 9, 3, 1 ]$ & 521 \\
$[ 1, 5, 11, 15, 15, 11, 5, 1 ]$ & 144 \\
$[ 0, 1, 6, 15, 20, 15, 6, 1 ]$ & 203 \\
$[ 0, 3, 12, 19, 16, 9, 4, 1 ]$ & 491 \\
$[ 0, 5, 10, 15, 20, 11, 2, 1 ]$ & 371 \\
$[ 0, 7, 0, 35, 0, 21, 0, 1 ]$ & 80 \\
\hline
Total & 2016 
\\
\\
$L_G$ with $G\simeq (C_2)^7$
& \# $\bm{Z}$-classes \\\hline
$[ 1, 7, 21, 35, 35, 21, 7, 1 ]$ & 80 \\
\hline
Total & 80 
\\
\\
\\
\\
\\
\\
\\
\end{tabular}

\vspace*{2mm}
\noindent
$L_G=[ t_{-7}, t_{-5}, t_{-3}, t_{-1}, t_1, t_3, t_5, t_7 ]$
$\in \bm{Z}^8$ is the list of $t_m$ which
is the number of elements of $G$ whose trace is $m$.\\
$T^i_G$ is the trace of the product of two distinct elements
with trace $i$.\\

The related functions we need are available from\\
{\tt https://www.math.kyoto-u.ac.jp/\~{}yamasaki/Algorithm/MultInvField/}.
For GAP ID and CARAT ID, see Section \ref{seCarat}. 
\footnote{
Before computation of CARAT,
we should change the limit 1000 to 10000
in the source files of the following folders of CARAT:
functions/Graph/lattices.c,
functions/Graph/super-k-groups-fcts.c,
functions/ZZ/ZZ.c. 
If you use BuildCarat.sh in our web page 
{\tt https://www.math.kyoto-u.ac.jp/\~{}yamasaki/Algorithm/MultInvField/}, 
the limit is fixed automatically.}

\begin{verbatim}
gap> Read("crystcat.gap");
gap> Read("caratnumber.gap");
gap> Read("KS.gap");

# (1)-(5): 2<=n<=6
gap> GL2Q:=Concatenation(List([1..4],x->List([1..NrQClassesCrystalSystem(2,x)],
> y->[2,x,y])));;
# all the finite subgroups of GL2(Z) up to Q-conjugate
gap> Length(GL2Q);
10
gap> C2_GL2Q:=Filtered(GL2Q,x->IdSmallGroup(MatGroupZClass(x[1],x[2],x[3],1))=[2,1]);
[ [ 2, 1, 2 ], [ 2, 2, 1 ] ]
gap> List(last,x->NrZClassesQClass(x[1],x[2],x[3]));
[ 1, 2 ]
gap> [Length(last),Sum(last)]; # # of Q-classes and Z-classes of C2 respectively
[ 2, 3 ]
gap> C2xC2_GL2Q:=Filtered(GL2Q,
> x->IdSmallGroup(MatGroupZClass(x[1],x[2],x[3],1))=[4,2]);
[ [ 2, 2, 2 ] ]
gap> List(last,x->NrZClassesQClass(x[1],x[2],x[3]));
[ 2 ]
gap> [Length(last),Sum(last)]; # # of Q-classes and Z-classes of (C2)^2 respectively
[ 1, 2 ]

gap> GL3Q:=Concatenation(List([1..7],x->List([1..NrQClassesCrystalSystem(3,x)],
> y->[3,x,y])));;
# all the finite subgroups of GL3(Z) up to Q-conjugate
gap> Length(GL3Q);
32
gap> C2_GL3Q:=Filtered(GL3Q,
> x->IdSmallGroup(MatGroupZClass(x[1],x[2],x[3],1))=[2,1]);
[ [ 3, 1, 2 ], [ 3, 2, 1 ], [ 3, 2, 2 ] ]
gap> List(last,x->NrZClassesQClass(x[1],x[2],x[3]));
[ 1, 2, 2 ]
gap> [Length(last),Sum(last)]; # # of Q-classes and Z-classes of C2 respectively
[ 1, 5 ]
gap> C2xC2_GL3Q:=Filtered(GL3Q,
> x->IdSmallGroup(MatGroupZClass(x[1],x[2],x[3],1))=[4,2]);
[ [ 3, 2, 3 ], [ 3, 3, 1 ], [ 3, 3, 2 ] ]
gap> List(last,x->NrZClassesQClass(x[1],x[2],x[3]));
[ 2, 4, 5 ]
gap> [Length(last),Sum(last)]; # # of Q-classes and Z-classes of (C2)^2 respectively
[ 3, 11 ]
gap> C2xC2xC2_GL3Q:=Filtered(GL3Q,
> x->IdSmallGroup(MatGroupZClass(x[1],x[2],x[3],1))=[8,5]);
[ [ 3, 3, 3 ] ]
gap> List(last,x->NrZClassesQClass(x[1],x[2],x[3]));
[ 4 ]
gap> [Length(last),Sum(last)]; # # of Q-classes and Z-classes of (C2)^3 respectively
[ 1, 4 ]

gap> GL4Q:=Concatenation(List([1..33],x->List([1..NrQClassesCrystalSystem(4,x)],
> y->[4,x,y])));;
# all the finite subgroups of GL4(Z) up to Q-conjugate
gap> Length(GL4Q);
227
gap> C2_GL4Q:=Filtered(GL4Q,x->IdSmallGroup(MatGroupZClass(x[1],x[2],x[3],1))=[2,1]);
[ [ 4, 1, 2 ], [ 4, 2, 1 ], [ 4, 2, 2 ], [ 4, 3, 1 ] ]
gap> List(last,x->NrZClassesQClass(x[1],x[2],x[3]));
[ 1, 2, 2, 3 ]
gap> [Length(last),Sum(last)]; # # of Q-classes and Z-classes of C2 respectively
[ 4, 8 ]
gap> C2xC2_GL4Q:=Filtered(GL4Q,
> x->IdSmallGroup(MatGroupZClass(x[1],x[2],x[3],1))=[4,2]);
[ [ 4, 2, 3 ], [ 4, 3, 2 ], [ 4, 4, 1 ], [ 4, 4, 2 ], [ 4, 4, 3 ],
  [ 4, 5, 1 ] ]
gap> List(last,x->NrZClassesQClass(x[1],x[2],x[3]));
[ 2, 3, 6, 7, 6, 13 ]
gap> [Length(last),Sum(last)]; # # of Q-classes and Z-classes of (C2)^2 respectively
[ 6, 37 ]
gap> C2xC2xC2_GL4Q:=Filtered(GL4Q,
> x->IdSmallGroup(MatGroupZClass(x[1],x[2],x[3],1))=[8,5]);
[ [ 4, 4, 4 ], [ 4, 5, 2 ], [ 4, 6, 1 ], [ 4, 6, 2 ] ]
gap> List(last,x->NrZClassesQClass(x[1],x[2],x[3]));
[ 6, 9, 12, 12 ]
gap> [Length(last),Sum(last)]; # # of Q-classes and Z-classes of (C2)^3 respectively
[ 4, 39 ]
gap> C2xC2xC2xC2_GL4Q:=Filtered(GL4Q,
> x->IdSmallGroup(MatGroupZClass(x[1],x[2],x[3],1))=[16,14]);
[ [ 4, 6, 3 ] ]
gap> List(last,x->NrZClassesQClass(x[1],x[2],x[3]));
[ 8 ]
gap> [Length(last),Sum(last)]; # # of Q-classes and Z-classes of (C2)^4 respectively
[ 1, 8 ]

gap> C2_GL5Q:=Filtered([1..955],x->Order(CaratMatGroupZClass(5,x,1))=2);
[ 2, 3, 4, 6, 7 ]
gap> List(last,x->CaratNrZClasses(5,x));
[ 1, 2, 2, 3, 3 ]
gap> [Length(last),Sum(last)]; # # of Q-classes and Z-classes of C2 respectively
[ 5, 11 ]
gap> C2xC2_GL5Q:=Filtered([1..955],x->Order(CaratMatGroupZClass(5,x,1))=4
> and IdSmallGroup(CaratMatGroupZClass(5,x,1))=[4,2]);
[ 5, 8, 9, 10, 11, 13, 14, 15, 18, 19 ]
gap> List(last,x->CaratNrZClasses(5,x));
[ 2, 3, 6, 7, 6, 9, 9, 11, 28, 18 ]
gap> [Length(last),Sum(last)]; # # of Q-classes and Z-classes of (C2)^2 respectively
[ 10, 99 ]
gap> C2xC2xC2_GL5Q:=Filtered([1..955],x->Order(CaratMatGroupZClass(5,x,1))=8
> and IdSmallGroup(CaratMatGroupZClass(5,x,1))=[8,5]);
[ 12, 16, 20, 21, 22, 23, 24, 30, 31, 32 ]
gap> List(last,x->CaratNrZClasses(5,x));
[ 6, 9, 18, 17, 17, 27, 27, 33, 49, 60 ]
gap> [Length(last),Sum(last)]; # # of Q-classes and Z-classes of (C2)^3 respectively
[ 10, 263 ]
gap> C2xC2xC2xC2_GL5Q:=Filtered([1..955],x->Order(CaratMatGroupZClass(5,x,1))=16
> and IdSmallGroup(CaratMatGroupZClass(5,x,1))=[16,14]);
[ 17, 25, 26, 27, 28 ]
gap> List(last,x->CaratNrZClasses(5,x));
[ 17, 33, 42, 16, 30 ]
gap> Sum(last);
138
gap> [Length(last),Sum(last)]; # # of Q-classes and Z-classes of (C2)^4 respectively
[ 5, 138 ]
gap> C2xC2xC2xC2xC2_GL5Q:=Filtered([1..955],x->Order(CaratMatGroupZClass(5,x,1))=32
> and IdSmallGroup(CaratMatGroupZClass(5,x,1))=[32,51]);
[ 29 ]
gap> List(last,x->CaratNrZClasses(5,x));
[ 16 ]
gap> [Length(last),Sum(last)]; # # of Q-classes and Z-classes of (C2)^5 respectively
[ 1, 16 ]

gap> C2_GL6Q:=Filtered([1..7103],x->Order(CaratMatGroupZClass(6,x,1))=2);
[ 1, 2, 4, 5, 11, 2710 ]
gap> List(last,x->CaratNrZClasses(6,x));
[ 2, 2, 3, 3, 4, 1 ]
gap> [Length(last),Sum(last)]; # # of Q-classes and Z-classes of (C2)^3 respectively
[ 6, 15 ]
gap> C2xC2_GL6Q:=Filtered([1..7103],x->Order(CaratMatGroupZClass(6,x,1))=4
> and IdSmallGroup(CaratMatGroupZClass(6,x,1))=[4,2]);
[ 3, 6, 7, 8, 9, 12, 13, 14, 15, 16, 4618, 4619, 4625, 4626, 4629, 4630 ]
gap> List(last,x->CaratNrZClasses(6,x));
[ 2, 3, 7, 6, 6, 4, 12, 12, 12, 12, 19, 29, 12, 17, 50, 52 ]
gap> Sum(last);
255
gap> [Length(last),Sum(last)]; # # of Q-classes and Z-classes of (C2)^2 respectively
[ 16, 255 ]
gap> C2xC2xC2_GL6Q:=Filtered([1..7103],x->Order(CaratMatGroupZClass(6,x,1))=8
> and IdSmallGroup(CaratMatGroupZClass(6,x,1))=[8,5]);
[ 10, 17, 4620, 4621, 4622, 4623, 4624, 4627, 4631, 4632, 4633, 4634, 4635,
  4644, 4645, 4646, 4647, 4648, 4649, 4650, 4664, 4665 ]
gap> List(last,x->CaratNrZClasses(6,x));
[ 6, 12, 19, 18, 18, 28, 28, 12, 39, 48, 50, 50, 48, 101, 83, 53, 101, 189,
  101, 189, 325, 131 ]
gap> [Length(last),Sum(last)]; # # of Q-classes and Z-classes of (C2)^3 respectively
[ 22, 1649 ]
gap> C2xC2xC2xC2_GL6Q:=Filtered([1..7103],x->Order(CaratMatGroupZClass(6,x,1))=16
> and IdSmallGroup(CaratMatGroupZClass(6,x,1))=[16,14]);
[ 4617, 4628, 4636, 4637, 4638, 4639, 4640, 4641, 4642, 4651, 4652, 4653,
  4654, 4655, 4656, 4657 ]
gap> List(last,x->CaratNrZClasses(6,x));
[ 18, 37, 53, 101, 97, 97, 49, 49, 127, 86, 76, 273, 189, 273, 313, 109 ]
gap> [Length(last),Sum(last)]; # # of Q-classes and Z-classes of (C2)^4 respectively
[ 16, 1947 ]
gap> C2xC2xC2xC2xC2_GL6Q:=Filtered([1..7103],x->Order(CaratMatGroupZClass(6,x,1))=32
> and IdSmallGroup(CaratMatGroupZClass(6,x,1))=[32,51]);
[ 4643, 4658, 4659, 4660, 4661, 4662 ]
gap> List(last,x->CaratNrZClasses(6,x));
[ 49, 37, 134, 75, 141, 75 ]
gap> [Length(last),Sum(last)]; # # of Q-classes and Z-classes of (C2)^5 respectively
[ 6, 511 ]
gap> C2xC2xC2xC2xC2xC2_GL6Q:=Filtered([1..7103],x->Order(CaratMatGroupZClass(6,x,1))=64
> and IdSmallGroup(CaratMatGroupZClass(6,x,1))=[64,267]);
[ 4663 ]
gap> List(last,x->CaratNrZClasses(6,x));
[ 36 ]
gap> [Length(last),Sum(last)]; # # of Q-classes and Z-classes of (C2)^6 respectively
[ 1, 36 ]
\end{verbatim}

\bigskip



\bigskip

The remaining part of this section is devoted to the study of $\bQ$-classes
of elementary abelian $2$-groups of $GL_n(\bZ)$ where $n$ is any positive
integer.

Recall that $G_{2^n}$ is the subgroup of $GL_n(\bZ)$ generated by the diagonal
matrices ${\rm Diag}(-1,1,1,1,1,1,1)$, 
$\ldots$, ${\rm Diag}(1,1,1,1,1,1,-1)$.
Since any elementary abelian $2$-group of $GL_n(\bZ)$ is $\bQ$-conjugate to
a subgroup of $G_{2^n}$, it suffices to study $\bQ$-classes of subgroups
in $G_{2^n}$.

As in Step 2 of the proof of Theorem \ref{t5.1}, we define $S_n$-equivalence
where $S_n$ is the subgroup of all the permutation matrices in $GL_n(\bZ)$.
Explicitly, two subgroups $H_1$ and $H_2$ of $G_{2^n}$ are $S_n$-equivalent
if there exists some element $t\in S_n$ such that
$H_2=t\cdot H_1\cdot t^{-1}$.

Clearly, ``$S_n$-equivalent'' $\Rightarrow$ ``$\bQ$-conjugation''.
The following theorem shows that the converse is also true.
Once we have this result, we may find all the $S_n$-equivalence classes,
and thus all the $\bQ$-classes, with the aid of the computer.
\begin{theorem}\label{t5.2}
Let $k$ be an integer with $1\leq k\leq n$,
$X(n,k):=\{H : H\ {\rm is\ a\ subgroup\ of\ }G_{2^n}\
{\rm and\ }H\simeq(C_2)^k\}$.
For any $H_1,H_2\in X(n,k)$, $H_1$ and $H_2$ are $\bQ$-conjugate
if and only if they are $S_n$-equivalent.
\end{theorem}
\begin{proof}
Step 1.
Let $\Phi:G_{2^n}\rightarrow(\bF_2)^n$ be the correspondence induced
by the group isomorphism $\{-1,1\}\rightarrow (\bF_2)^n$ as in Step 1
of the proof of Theorem \ref{t5.1}.
If $H\in X(n,k)$, then $\Phi(H)$ is a vector subspace of dimension $k$
in $(\bF_2)^n$.
Define $H(n,k):=\{W:W\ {\rm is\ a\ subspace\ of}\ (\bF_2)^n\ {\rm and}\ {\rm dim}_{\bF_2}W=k\}$.
For any vector $w\in (\bF_2)^n$, define the weight of $w$, denoted by
${\rm wt}(w)$, as ${\rm wt}(w)=\#\{w_i:w_i\neq 0\}$
if $w=(w_1,\ldots,w_n)\in(\bF_2)^n$.
It is easy to see that if $x\in G_{2^n}$, then ${\rm wt}(\Phi(x))=l
\Leftrightarrow {\rm trace}(x)=n-2l$.
We remark that if
$W\in H(n,k)$, $W$ is called a binary $(n,k)$-code in \cite{BBFKKW}.
For any $W_1, W_2\in H(n,k)$, a local isometry of $W_1$ and $W_2$ is a
linear isomorphism $\varphi:W_1\rightarrow W_2$ such that
${\rm wt}(\varphi(w))={\rm wt}(w)$ for any $w\in W_1$ \cite[page 549]{BBFKKW}.
Because we consider binary codes, i.e. the base field is $\bF_2$, there is no difference between linear maps and semilinear maps, which
are emphasized in \cite{BBFKKW}.
The reader may find the notion of isometries (instead of local isometries!) in
\cite[page 30 and page 40]{BBFKKW}; we will not use isometries in our proof.

\medskip
Step 2.
By abusing the notation, we will denote $S_n$ the symmetric group of
permutations on $n$ letters $\{1,2,\ldots,n\}$.
Thus, if $\sigma\in S_n$, then $\sigma$ induces a linear isomorphism on
$(\bF_2)^n$ by $\sigma(v)=(v_{\sigma(1)},\ldots,v_{\sigma(n)})$ if
$v=(v_1,\ldots,v_n)\in (\bF_2)^n$.
Using this linear isomorphism, we define an equivalence relation on $H(n,k)$.
For any $W_1, W_2\in H(n,k)$, $W_1$ is $S_n$-equivalent to $W_2$
if there is some $\sigma\in S_n$ such that $\sigma(W_1)=W_2$.
It is easy to verify if $t\in S_n$ is a permutation matrix associated to
the permutation $\sigma$, and $H_1, H_2\in X(n,k)$, then
$H_2=t\cdot H_1\cdot t^{-1}$ if and only if $\sigma(\Phi(H_1))=\Phi(H_2)$.

\medskip
Step 3.
Let $W_1,W_2\in H(n,k)$.
If $W_1$ is $S_n$-equivalent to $W_2$, it is clear that $W_1$ and $W_2$
are locally isometric.
The converse is true by \cite[page 551, Theorem 6.8.4]{BBFKKW}.
That is, $W_1$ is $S_n$-equivalent to $W_2$ if and only if they are
locally isometric.
This result is the key tool for Step 4..

\medskip
Step 4.
Now we come to the proof of this theorem.
It suffices to show that, if $H_1$ and $H_2$ are $\bQ$-conjugate,
then they are $S_n$-equivalent.
Suppose that $g\in GL_n(\bQ)$ and $H_2=g\cdot H_1\cdot g^{-1}$.
Then ${\rm trace}(x)={\rm trace}(g\cdot x\cdot g^{-1})$ for all $x\in H_1$.
It follows that the conjugation map $x\mapsto g\cdot x\cdot g^{-1}$
induces a local isometry from $\Phi(H_1)$ to $\Phi(H_2)$.
By the result in Step 3 (i.e. \cite[page 551, Theorem 6.8.4]{BBFKKW}),
$\Phi(H_1)$ and $\Phi(H_2)$ are $S_n$-equivalent.
It follows that $H_1$ and $H_2$ are $S_n$-equivalent by Step 2.
\end{proof}
\begin{remark}
{}From the above proof, we find that the following numbers are the same:\\ 
(i) the number of $\bQ$-classes in $X(n,k)$;\\ 
(ii) the number of $S_n$-equivalence classes of $X(n,k)$;\\
(iii) the number of local isometry classes of $H(n,k)$;\\ 
(iv) the number of $S_n$-equivalence classes of $H(n,k)$.\\
A list of the number of the local isometry classes of $H(n,k)$ (for $1\leq n\leq 14$)
may be found in \cite[page 446, Table 6.2]{BBFKKW}.
\end{remark}

\section{The case $G=(C_2)^3$ with $H_u^2(G,M)\neq 0$}\label{seAbelian}

In Section \ref{seC} we have found the list of all the groups isomorphic to $(C_2)^3$ up to conjugation in $GL_7(\bZ)$. We will find the unramified Brauer groups of the multiplicative invariant fields associated to these groups in this section.

Note that, if $G$ is abelian, then $\bC(G)$ is $\bC$-rational and $H_u^2(G,\bm{Q}/\bm{Z})=0$ by Theorem \ref{thFi}. Thus ${\rm Br}_u(\bm{C}(M)^G)=H_u^2(G,M)$.
We will find that there exist exactly $9$ groups
$G\leq GL_7(\bm{Z})$ such that $G \simeq (C_2)^3$ and $H_u^2(G,M)\neq 0$.

\begin{theorem}\label{t5.4}
Let $G$ be an elementary abelian group of order $2^n$ in $GL_7(\bm{Z})$ and
$M$ be the associated $G$-lattice of rank $7$.
Then 
${\rm Br}_u(\bm{C}(M)^G)\neq 0$ if and only if $G$ is isomorphic up to conjugation to one of the nine groups $G_1,\ldots,G_9$ in $GL_7(\bm{Z})$ where each of $G_i$ is isomorphic to $(C_2)^3$ as an abstract group. Moreover, ${\rm Br}_u(\bm{C}(M)^{G_i})=H_u^2(G_i,M)\simeq \bZ/2\bZ$
$($resp. $\bZ/2\bZ\oplus\bZ/2\bZ$$)$ for $1\leq i\leq 8$
$($resp. $i=9$$)$. These groups $G_1,\ldots,G_9$ are defined as follows:
\begin{align*}
G_1&=
\left\langle\left(
\begin{smallmatrix}
0 & 1 & -1 & 0 & 0 & 0 & 0 \\
1 & 0 & 1 & 0 & 0 & 0 & 0 \\
0 & 0 & 1 & 0 & 0 & 0 & 0 \\
0 & 0 & 0 & 1 & 0 & 0 & 0 \\
0 & 0 & -1 & 0 & -1 & 0 & 1 \\
0 & 0 & 0 & 0 & 0 & -1 & 1 \\
0 & 0 & 0 & 0 & 0 & 0 & 1
\end{smallmatrix}\right),\left(
\begin{smallmatrix}
0 & -1 & 0 & 0 & 0 & 0 & 0 \\
-1 & 0 & 0 & 0 & 0 & 0 & 0 \\
0 & 0 & 1 & 0 & 0 & 0 & 0 \\
0 & 0 & 1 & -1 & 0 & 0 & 0 \\
0 & 0 & 0 & 0 & 1 & 0 & -1 \\
0 & 0 & 0 & 0 & 0 & 1 & -1 \\
0 & 0 & 0 & 0 & 0 & 0 & -1
\end{smallmatrix}\right),\left(
\begin{smallmatrix}
-1 & 0 & 0 & 0 & 0 & 0 & 0 \\
0 & -1 & 0 & 0 & 0 & 0 & 0 \\
0 & 0 & -1 & 0 & 0 & 0 & 0 \\
0 & 0 & -1 & 1 & 0 & 0 & -1 \\
0 & 0 & 1 & 0 & 1 & 0 & -1 \\
0 & 0 & 0 & 0 & 0 & -1 & 0\\
0 & 0 & 0 & 0 & 0 & 0 & -1
\end{smallmatrix}\right)
\right\rangle,\\
G_2&=
\left\langle\left(
\begin{smallmatrix}
0 & -1 & 0 & 0 & 0 & 0 & 1\\
-1 & 0 & 0 & 0 & 0 & 0 & -1\\
0 & 0 & -1 & 0 & 0 & 0 & 0\\
0 & 0 & 1 & 0 & -1 & 0 & -1\\
0 & 0 & -1 & -1 & 0 & 0 & 1\\
0 & 0 & 0 & 0 & 0 & 1 & -1\\
0 & 0 & 0 & 0 & 0 & 0 & -1
\end{smallmatrix}\right),\left(
\begin{smallmatrix}
0 & -1 & 0 & 0 & 0 & 0 & 1\\
-1 & 0 & 0 & 0 & 0 & 0 & -1\\
0 & 0 & 0 & 1 & 1 & 0 & 0\\
0 & 0 & 1 & 0 & -1 & 0 & -1\\
0 & 0 & 0 & 0 & 1 & 0 & 1\\
0 & 0 & 0 & 0 & 0 & -1 & 0\\
0 & 0 & 0 & 0 & 0 & 0 & -1
\end{smallmatrix}\right),\left(
\begin{smallmatrix}
0 & 1 & 0 & 0 & 0 & 0 & 0\\
1 & 0 & 0 & 0 & 0 & 0 & 0\\
0 & 0 & 0 & 1 & 1 & 0 & 0\\
0 & 0 & 0 & -1 & 0 & 0 & 0\\
0 & 0 & 1 & 1 & 0 & 0 & 0\\
0 & 0 & 0 & 0 & 0 & 1 & -1\\
0 & 0 & 0 & 0 & 0 & 0 & -1
\end{smallmatrix}\right)
\right\rangle,\\
G_3&=
\left\langle\left(
\begin{smallmatrix}
-1 & 0 & 0 & 0 & 0 & 0 & 0\\
0 & 0 & -1 & -1 & 0 & 0 & 0\\
0 & 0 & -1 & 0 & 0 & 0 & 0\\
0 & -1 & 1 & 0 & 0 & 0 & 0\\
0 & 0 & 0 & 0 & 0 & 1 & 1\\
0 & 0 & 0 & 0 & 1 & 0 & -1\\
0 & 0 & 0 & 0 & 0 & 0 & 1
\end{smallmatrix}\right),\left(
\begin{smallmatrix}
-1 & 0 & 0 & 0 & 0 & 0 & 0\\
0 & 0 & -1 & -1 & 0 & 0 & 0\\
0 & -1 & 0 & -1 & -1 & 1 & 1\\
0 & 0 & 0 & 1 & 1 & -1 & -1\\
0 & 0 & 0 & 0 & -1 & 0 & 0\\
0 & 0 & 0 & 0 & 0 & -1 & 0\\
0 & 0 & 0 & 0 & 0 & 0 & -1
\end{smallmatrix}\right),\left(
\begin{smallmatrix}
1 & 0 & 0 & 0 & 0 & 0 & 0\\
0 & 0 & 1 & 1 & 0 & 0 & 0\\
1 & 0 & 1 & 0 & -1 & 1 & 1\\
-1 & 1 & -1 & 0 & 1 & -1 & -1\\
1 & 0 & 0 & 0 & 0 & 1 & 1\\
-1 & 0 & 0 & 0 & 0 & -1 & 0\\
0 & 0 & 0 & 0 & 1 & 1 & 0
\end{smallmatrix}\right)
\right\rangle,\\
G_4&=
\left\langle\left(
\begin{smallmatrix}
-1 & 0 & 0 & 0 & 0 & 0 & -1\\
0 & -1 & 0 & 0 & 0 & 0 & 1\\
0 & 0 & -1 & 0 & 0 & 0 & -1\\
0 & 0 & -1 & 0 & -1 & 0 & -1\\
0 & 0 & 1 & -1 & 0 & 0 & 0\\
0 & 0 & 0 & 0 & 0 & 1 & 0\\
0 & 0 & 0 & 0 & 0 & 0 & 1
\end{smallmatrix}\right),\left(
\begin{smallmatrix}
0 & -1 & 0 & 0 & 0 & 0 & 0\\
-1 & 0 & 0 & 0 & 0 & 0 & 0\\
0 & 0 & -1 & 0 & 0 & 0 & -1\\
0 & 0 & 0 & -1 & 0 & 0 & -1\\
0 & 0 & 0 & 0 & -1 & 0 & 0\\
0 & 0 & 0 & 0 & 0 & -1 & 1\\
0 & 0 & 0 & 0 & 0 & 0 & 1
\end{smallmatrix}\right),\left(
\begin{smallmatrix}
1 & 0 & 0 & 0 & 0 & 0 & 1\\
0 & 1 & 0 & 0 & 0 & 0 & -1\\
0 & 0 & 0 & 1 & 1 & 0 & 1\\
0 & 0 & 0 & 1 & 0 & 0 & 1\\
0 & 0 & 1 & -1 & 0 & 0 & 0\\
0 & 0 & 0 & 0 & 0 & -1 & 0\\
0 & 0 & 0 & 0 & 0 & 0 & -1
\end{smallmatrix}\right)
\right\rangle,\\
G_5&=
\left\langle\left(
\begin{smallmatrix}
0 & -1 & 0 & 0 & 0 & 0 & 0\\
-1 & 0 & 0 & 0 & 0 & 0 & 0\\
0 & 0 & -1 & 0 & 0 & 0 & 1\\
0 & 0 & 1 & 0 & -1 & 0 & 0\\
0 & 0 & -1 & -1 & 0 & 0 & 1\\
0 & 0 & 0 & 0 & 0 & -1 & 1\\
0 & 0 & 0 & 0 & 0 & 0 & 1
\end{smallmatrix}\right),\left(
\begin{smallmatrix}
-1 & 0 & 0 & 0 & 0 & 0 & 1\\
0 & -1 & 0 & 0 & 0 & 0 & -1\\
0 & 0 & -1 & 0 & 0 & 0 & 1\\
0 & 0 & 0 & -1 & 0 & 0 & 0\\
0 & 0 & 0 & 0 & -1 & 0 & 1\\
0 & 0 & 0 & 0 & 0 & 1 & 0\\
0 & 0 & 0 & 0 & 0 & 0 & 1
\end{smallmatrix}\right),\left(
\begin{smallmatrix}
0 & 1 & 0 & 0 & 0 & 0 & 0\\
1 & 0 & 0 & 0 & 0 & 0 & 0\\
0 & 0 & 0 & 1 & 1 & 0 & -1\\
0 & 0 & 1 & 0 & -1 & 0 & 0\\
0 & 0 & 0 & 0 & 1 & 0 & -1\\
0 & 0 & 0 & 0 & 0 & 1 & -1\\
0 & 0 & 0 & 0 & 0 & 0 & -1
\end{smallmatrix}\right)
\right\rangle,\\
G_6&=
\left\langle\left(
\begin{smallmatrix}
0 & 1 & 1 & 1 & 0 & 0 & 0\\
1 & 0 & 1 & 0 & 0 & 0 & 0\\
0 & 0 & -1 & -1 & 0 & 0 & 0\\
0 & 0 & 0 & 1 & 0 & 0 & 0\\
0 & 0 & 0 & -1 & -1 & 0 & 0\\
0 & 0 & 0 & -1 & -1 & 0 & 1\\
0 & 0 & 0 & 0 & -1 & 1 & 0
\end{smallmatrix}\right),\left(
\begin{smallmatrix}
0 & -1 & -1 & 0 & 1 & 1 & -1\\
0 & -1 & 0 & -1 & 0 & 0 & 0\\
-1 & 1 & 0 & -1 & -1 & -1 & 1\\
0 & 0 & 0 & 1 & 0 & 0 & 0\\
0 & 0 & 0 & -1 & -1 & 0 & 0\\
0 & 0 & 0 & -1 & 0 & -1 & 0\\
0 & 0 & 0 & 0 & 0 & 0 & -1
\end{smallmatrix}\right),\left(
\begin{smallmatrix}
0 & -1 & -1 & -1 & 0 & 0 & 0\\
-1 & 0 & -1 & 0 & 0 & 0 & 0\\
0 & 0 & 1 & 1 & 0 & 0 & 0\\
0 & 0 & 0 & -1 & 0 & 0 & 0\\
0 & 0 & 0 & 1 & 0 & 1 & -1\\
0 & 0 & 0 & 1 & 0 & 1 & 0\\
0 & 0 & 0 & 0 & -1 & 1 & 0
\end{smallmatrix}\right)
\right\rangle,\\
G_7&=
\left\langle\left(
\begin{smallmatrix}
0 & 1 & 0 & 0 & 0 & 0 & 0\\
1 & 0 & 0 & 0 & 0 & 0 & 0\\
0 & 0 & -1 & 0 & 0 & 0 & 0\\
0 & 0 & 0 & -1 & 0 & 0 & 0\\
0 & 0 & 0 & 0 & -1 & 0 & 0\\
0 & 0 & 0 & 0 & 0 & 1 & -1\\
0 & 0 & 0 & 0 & 0 & 0 & -1
\end{smallmatrix}\right),\left(
\begin{smallmatrix}
0 & 1 & 0 & 0 & 0 & 0 & 0\\
1 & 0 & 0 & 0 & 0 & 0 & 0\\
0 & 0 & 0 & 1 & 1 & 0 & 0\\
0 & 0 & 0 & -1 & 0 & 0 & 0\\
0 & 0 & 1 & 1 & 0 & 0 & 0\\
0 & 0 & 0 & 0 & 0 & -1 & 0\\
0 & 0 & 0 & 0 & 0 & 0 & -1
\end{smallmatrix}\right),\left(
\begin{smallmatrix}
0 & -1 & 0 & 0 & 0 & 0 & 1\\
-1 & 0 & 0 & 0 & 0 & 0 & -1\\
0 & 0 & 0 & 1 & 1 & 0 & 0\\
0 & 0 & 1 & 0 & -1 & 0 & -1\\
0 & 0 & 0 & 0 & 1 & 0 & 1\\
0 & 0 & 0 & 0 & 0 & 1 & -1\\
0 & 0 & 0 & 0 & 0 & 0 & -1
\end{smallmatrix}\right)
\right\rangle,\\
G_8&=
\left\langle\left(
\begin{smallmatrix}
-1 & 0 & 0 & 0 & 0 & 0 & -1\\
0 & -1 & 0 & 0 & 0 & 0 & 1\\
0 & 0 & -1 & 0 & 0 & 0 & -1\\
0 & 0 & 0 & -1 & 0 & 0 & 0\\
0 & 0 & 0 & 0 & -1 & 0 & 1\\
0 & 0 & 0 & 0 & 0 & 1 & 0\\
0 & 0 & 0 & 0 & 0 & 0 & 1
\end{smallmatrix}\right),\left(
\begin{smallmatrix}
-1 & 0 & 0 & 0 & 0 & 0 & -1\\
0 & -1 & 0 & 0 & 0 & 0 & 1\\
0 & 0 & 0 & -1 & -1 & 0 & 0\\
0 & 0 & 0 & -1 & 0 & 0 & 0\\
0 & 0 & -1 & 1 & 0 & 0 & 0\\
0 & 0 & 0 & 0 & 0 & -1 & 1\\
0 & 0 & 0 & 0 & 0 & 0 & 1
\end{smallmatrix}\right),\left(
\begin{smallmatrix}
0 & -1 & 0 & 0 & 0 & 0 & 1\\
-1 & 0 & 0 & 0 & 0 & 0 & -1\\
0 & 0 & 0 & -1 & -1 & 0 & 1\\
0 & 0 & -1 & 0 & -1 & 0 & 0\\
0 & 0 & 0 & 0 & 1 & 0 & -1\\
0 & 0 & 0 & 0 & 0 & 1 & -1\\
0 & 0 & 0 & 0 & 0 & 0 & -1
\end{smallmatrix}\right)
\right\rangle,\\
G_9&=
\left\langle\left(
\begin{smallmatrix}
1 & -1 & -1 & 1 & 0 & 1 & 0\\
-1 & 0 & 0 & 0 & -1 & -1 & 0\\
-1 & 0 & -1 & -1 & 0 & 1 & -1\\
-1 & 1 & 0 & -1 & -1 & -1 & 0\\
0 & 0 & 1 & 0 & 0 & -1 & 1\\
-1 & 0 & 0 & -1 & 0 & 0 & -1\\
0 & 0 & 1 & 0 & 1 & -1 & 0
\end{smallmatrix}\right),\left(
\begin{smallmatrix}
0 & 1 & 1 & 0 & 0 & 0 & 1\\
1 & 1 & 0 & -1 & 0 & 1 & 1\\
1 & 0 & 0 & 1 & -1 & 0 & 0\\
0 & 0 & 0 & -1 & 0 & 0 & 0\\
0 & 1 & 0 & -1 & 0 & 0 & 1\\
0 & 0 & -1 & 0 & -1 & 0 & -1\\
-1 & -1 & 0 & 0 & 1 & -1 & -1
\end{smallmatrix}\right),\left(
\begin{smallmatrix}
0 & 1 & 0 & 0 & -1 & -1 & 0\\
1 & 0 & 1 & 0 & 0 & 0 & 1\\
0 & 1 & 1 & -1 & 1 & 0 & 1\\
1 & 0 & 0 & 0 & 0 & 1 & 1\\
0 & 0 & 0 & 0 & -1 & 0 & 0\\
0 & 0 & 1 & 0 & 1 & 0 & 1\\
0 & -1 & -1 & 1 & 0 & 1 & -1
\end{smallmatrix}\right)
\right\rangle.
\end{align*}

\end{theorem}
\begin{proof}
As noted before, $G$ is an abelian group and therefore $B_0(G)=H_u^2(G,\bm{Q}/\bm{Z})$ is trivial. Hence ${\rm Br}_u(\bm{C}(M)^G)=H_u^2(G,M)$. From now on we focus on the computation of $H_u^2(G,M)$.

When ${\rm rank}_\bm{Z} M\leq 3$, it follows from Theorem \ref{thHaj87} and
Theorem \ref{thHKHR} that $\bm{C}(M)^G$ is $\bm{C}$-rational.
Hence $H_u^2(G,M)=0$.
When $4\leq {\rm rank}_\bm{Z} M\leq 6$, it follows from Theorem \ref{t1.5}
(see Theorems \ref{thm1}, \ref{thm2}, \ref{thm3} for details)
that $H_u^2(G,M)=0$.

\medskip
Hence we assume that ${\rm rank}_\bm{Z} M=7$ from now on.
For $|G|\leq 4$, it follows from Theorem \ref{thB1} that $H_u^2(G,M)=0$.
For $8\leq |G|\leq 128$, we can apply the function {\tt H2nrM} given
in Section \ref{seGAP} to those groups $\simeq (C_2)^k$ classified in Theorem \ref{t5.1}.
Then we obtain that only the following $9$ groups $G\simeq (C_2)^3$
satisfy $H_u^2(G,M)\neq 0$.
\begin{center}
\begin{tabular}{lcccc}
$L_G$ with $G\simeq (C_2)^3$
& & \# $\bm{Z}$-classes & \# $G$'s with $H_u^2(G,M)\neq 0$ & $H_u^2(G,M)$  \\\hline
$[ 0, 1, 1, 2, 3, 0, 0, 1 ]$ & & 380 & 1 & $\bm{Z}/2\bm{Z}$ \\
$[ 0, 1, 0, 5, 0, 1, 0, 1 ]$ & & 279 & 1 & $\bm{Z}/2\bm{Z}$ \\
$[ 0, 0, 2, 3, 2, 0, 0, 1 ]$ & & 1000 & 4 & $\bm{Z}/2\bm{Z}$ \\
$[ 0, 0, 3, 2, 1, 1, 0, 1 ]$ & & 429 & 2 & $\bm{Z}/2\bm{Z}$ \\
$[ 0, 0, 0, 7, 0, 0, 0, 1 ]$ & & 269 & 1 & $(\bm{Z}/2\bm{Z})^{\oplus 2}$ \\
\end{tabular}
\end{center}
\end{proof}

\begin{verbatim}
gap> LoadPackage("carat");
true
gap> Read("res.gap");

gap> Length(G8Qc); # G8Qc can be obtained in (6) of Section 5
43
gap> G8Zc:=List(G8Qc,ZClassRepsQClass);;
gap> List(G8Zc,Length);
[ 12, 13, 6, 60, 20, 19, 62, 80, 380, 208, 107, 28, 80, 279, 113, 58, 279,
  58, 28, 58, 18, 107, 58, 1919, 415, 221, 429, 208, 1000, 429, 58, 113, 92,
  58, 691, 60, 432, 1077, 221, 691, 269, 18, 113 ]
gap> Sum(last); # # of Z-classes of (C2)^3
10645
gap> G8H2nr:=List(G8Zc,x->Filtered(x,y->H2nrM(y).H2nrM<>[]));; 
# choosing only the cases with non-trivial H2nr(G,M)
gap> G8H2nr;
[ [  ], [  ], [  ], [  ], [  ], [  ], [  ], [  ],
  [ <matrix group of size 8 with 3 generators> ], [  ], [  ], [  ], [  ],
  [  ], [  ], [  ], [ <matrix group of size 8 with 3 generators> ], [  ],
  [  ], [  ], [  ], [  ], [  ], [  ], [  ], [  ], [  ], [  ],
  [ <matrix group of size 8 with 3 generators>,
      <matrix group of size 8 with 3 generators>,
      <matrix group of size 8 with 3 generators>,
      <matrix group of size 8 with 3 generators> ],
  [ <matrix group of size 8 with 3 generators>,
      <matrix group of size 8 with 3 generators> ], [  ], [  ], [  ], [  ],
  [  ], [  ], [  ], [  ], [  ], [  ],
  [ <matrix group of size 8 with 3 generators> ], [  ], [  ] ]
gap> List(G8H2nr,Length);
[ 0, 0, 0, 0, 0, 0, 0, 0, 1, 0, 0, 0, 0, 0, 0, 0, 1, 0, 0, 0, 0, 0, 0, 0, 0,
  0, 0, 0, 4, 2, 0, 0, 0, 0, 0, 0, 0, 0, 0, 0, 1, 0, 0 ]
gap> Sum(last); # #=9 with non-trivial H2nr(G,M)
9
gap> G8H2nrs:=Flat(G8H2nr);;
gap> List(G8H2nrs,x->H2nrM(x).H2nrM); # structure of H2nr(G,M)
[ [ 2 ], [ 2 ], [ 2 ], [ 2 ], [ 2 ], [ 2 ], [ 2 ], [ 2 ], [ 2, 2 ] ]
gap> for i in G8H2nrs do Print(GeneratorsOfGroup(i),"\n"); od; 
# we omit the output of the generators of 9 groups (see Theorem 6.1)

gap> G16Zc:=List(G16Qc,ZClassRepsQClass);;
gap> List(G16Zc,Length);
[ 279, 113, 58, 58, 18, 691, 221, 60, 269, 886, 270, 2823, 109, 1878, 883,
  670, 270, 789, 380, 789, 1170, 422, 141, 574, 144, 109, 144, 56, 380, 214,
  3055, 2534, 1170, 883, 1037, 361, 2867, 670, 558, 56, 574, 145, 664 ]
gap> Sum(last);
29442
gap> G16H2nr:=List(G16Zc,x->Filtered(x,y->H2nrM(y).H2nrM<>[]));;
# choosing only the cases with non-trivial H2nr(G,M)
gap> G16H2nr;
[ [  ], [  ], [  ], [  ], [  ], [  ], [  ], [  ], [  ], [  ], [  ], [  ],
  [  ], [  ], [  ], [  ], [  ], [  ], [  ], [  ], [  ], [  ], [  ], [  ],
  [  ], [  ], [  ], [  ], [  ], [  ], [  ], [  ], [  ], [  ], [  ], [  ],
  [  ], [  ], [  ], [  ], [  ], [  ], [  ] ]

gap> G32Zc:=List(G32Qc,ZClassRepsQClass);;
gap> List(G32Zc,Length);
[ 789, 380, 574, 144, 56, 664, 145, 1034, 360, 1860, 983, 206, 2499, 557,
  899, 557, 521, 360, 144, 144, 1183, 557, 632 ]
gap> Sum(last);
15248
gap> G32H2nr:=List(G32Zc,x->Filtered(x,y->H2nrM(y).H2nrM<>[]));;
# choosing only the cases with non-trivial H2nr(G,M)
gap> G32H2nr;
[ [  ], [  ], [  ], [  ], [  ], [  ], [  ], [  ], [  ], [  ], [  ], [  ],
  [  ], [  ], [  ], [  ], [  ], [  ], [  ], [  ], [  ], [  ], [  ] ]

gap> G64Zc:=List(G64Qc,ZClassRepsQClass);;
gap> List(G64Zc,Length);
[ 206, 521, 144, 203, 491, 371, 80 ]
gap> Sum(last);
2016
gap> G64H2nr:=List(G64Zc,x->Filtered(x,y->H2nrM(y).H2nrM<>[]));;
# choosing only the cases with non-trivial H2nr(G,M)
gap> G64H2nr;
[ [  ], [  ], [  ], [  ], [  ], [  ], [  ] ]

gap> G128Zc:=List(G128Qc,ZClassRepsQClass);;
gap> List(G128Zc,Length);
[ 80 ]
gap> Sum(last);
80
gap> G128H2nr:=List(G128Zc,x->Filtered(x,y->H2nrM(y).H2nrM<>[]));;
# choosing only the cases with non-trivial H2nr(G,M)
gap> G128H2nr;
[ [  ] ]
\end{verbatim}

\bigskip

\section{The case $G=A_6$ with $H_u^2(G,M)\neq 0$ and Noether's problem for $N\rtimes A_6$}\label{seA6}

Because 
all the $G$-lattices $M$ with ${\rm rank}_\bZ M\leq 6$ and $H_u^2(G,M)\neq 0$
in Theorem \ref{t1.5}, Theorems \ref{thm1}, \ref{thm2}, \ref{thm3} 
(resp. Theorem \ref{t5.4}) satisfy the condition that 
$G$ is non-abelian and solvable 
(resp. elementary abelian), 
one may wonder what happens 
in the case of finite non-abelian simple groups. 

In this section, we give an example of two $G$-lattices $M$ with
${\rm Br}_u(\bC(M)^G)=H_u^2(G,M)\neq 0$ 
where $G \simeq A_6$ is the alternating group of degree $6$
and ${\rm rank}_\bZ M =9$. 
Note that $B_0(G)={\rm Br}_u(\bC(G))=0$ for any non-abelian simple group $G$ 
by Kunyavskii's theorem \cite{Ku}; 
our example illustrates the situation of the multiplicative invariant fields.


\begin{theorem}\label{t7.1}
Let $A_6$ be the alternating group of degree $6$. 
Embed $A_6$ into $S_{10}$ through the isomorphism $A_6\simeq PSL_2(\bF_9)$, 
which acts on the projective line $\bm{P}^1 _{\bF_9}$ via 
fractional linear transformations. 
Thus we may regard $A_6$ as a transitive subgroup of $S_{10}$.  
Let $P=\oplus_{1\leq i\leq 10}\bZ\cdot x_i$ be the 
permutation $S_{10}$-lattice defined by $\sigma \cdot x_i = x_{\sigma(i)}$ 
for any $\sigma \in S_{10}$; it becomes a permutation $A_6$-lattice by 
restricting the action of $S_{10}$ to $A_6$. 
Define $J=P/(\bZ\cdot \sum_{i=1}^{10}x_i)$ with ${\rm rank}_\bZ J=9$.
There exist exactly six $A_6$-lattices $J=M_1$, $M_2,\ldots,M_6$
which are $\bQ$-conjugate but not $\bZ$-conjugate to each other. 
Then we have
\begin{align*}
H_u^2(A_6,M_1)\simeq H_u^2(A_6,M_3)\simeq\bZ/2\bZ,\quad
H_u^2(A_6,M_i)=0\ {\rm for}\ i=2,4,5,6.
\end{align*}

In particular, $\bC(M_1)^{A_6}$ and $\bC(M_3)^{A_6}$
are not retract $\bC$-rational.

Furthermore, the lattices $M_1$ and $M_3$ may be distinguished by the Tate cohomology groups,
\begin{align*}
H^1(A_6,M_1)=0,\ \widehat{H}^{-1}(A_6,M_1)=\bZ/10\bZ,\
H^1(A_6,M_3)=\bZ/5\bZ,\ \widehat{H}^{-1}(A_6,M_3)=\bZ/2\bZ.
\end{align*}
\end{theorem}
\begin{proof}
Let $G\leq S_{10}$ be the $26$th transitive subgroup $10T26$
of $S_{10}$ which is isomorphic to the alternating group $A_6$ of degree $6$
(see \cite{BM}, \cite{GAP}).
We take the corresponding $A_6$-lattice $P=\oplus_{1\leq i\leq 10}\bZ\cdot x_i$
with ${\rm rank}_\bZ P=10$ and 
$J=P/(\bZ\cdot \sum_{i=1}^{10}x_i)$ with ${\rm rank}_\bZ J=9$.

We remark that 
$P\simeq \bZ[A_6/H]$ where $H\simeq (C_3\times C_3)\rtimes C_4$ 
and 
$J\simeq J_{G/H}$
where $J_{G/H}$ is the character lattice of the norm one torus 
$R_{K/k}^{(1)}(\bG_m)$ of some non-Galois extension $K/k$ 
of degree $10$ (see \cite[Chapter 8]{HY}): 
\begin{align*}
&0\to I_{G/H}\to P\xrightarrow{\varepsilon}\bZ\to 0,\\
&0\to \bZ\xrightarrow{\varepsilon^\circ} P\to J_{G/H}\to 0
\end{align*}
where $\varepsilon$ is the argumentation map.
We also remark that $M_1=J\simeq J_{G/H}=(I_{G/H})^\circ\simeq (M_6)^\circ$, 
$M_2=(M_5)^\circ$, $M_3=(M_4)^\circ$ where 
$M^\circ={\rm Hom}_\bZ(M,\bZ)$ is the dual lattice of a $G$-lattice $M$. 

We can take $J\simeq J_{G/H}$ in GAP
by using the command ${\tt Norm1TorusJ}(10,26)$
in ${\tt FlabbyResolution.gap}$ which is available from
{\tt https://www.math.kyoto-u.ac.jp/\~{}yamasaki/Algorithm/MultInvField/}.
Then just applying the function ${\tt H2nrM}$,
we find that $H_u^2(A_6,J)\simeq\bZ/2\bZ$.

Using the command {\tt ZClassRepsQClass} of CARAT package \cite{CARAT}
of GAP, we can split the $\bQ$-class of $M$ into six $\bZ$-classes.
Then just apply the function ${\tt H2nrM}$ to determine $H_u^2(A_6,M_i)$
$(1\leq i\leq 6)$.
We can distinguish the six $\bZ$-classes according to their
cohomologies $H^1$ and $\widehat{H}^{-1}$ 
(see a demonstration of the GAP functions in the end of this section).
\end{proof}

In some cases, 
Saltman \cite{Sa4} gave an application of 
the non-vanishing $H_u^2(G,M)\neq 0$ to 
a counter-example of Noether's problem for $N\rtimes G$ 
over $\bC$ where $N$ is some abelian group: 

%
\begin{theorem}[{Saltman \cite[Corollary 3.3]{Sa4}}]\label{thSa4}
Let $M$ be a faithful $G$-lattice and $M^\prime$ be a $G$-lattice 
with $[M^\prime]^{fl}=0$. 
Assume that $P^\prime=(M\oplus M^\prime)\otimes_\bZ \bQ$ 
is a permutation $\bQ[G]$-lattice. 
Then $P^\prime=P\otimes_\bZ \bQ$ 
and $M\oplus M^\prime\leq P$ with finite index
for some permutation $G$-lattice $P$. 
Let $G^\prime=N\rtimes G$ where 
$N={\rm Hom}_\bZ(P/(M\oplus M^\prime),\bQ/\bZ)$. 
Then $\bC(G^\prime)$ is stably isomorphic to $\bC(M)^G$. 
In particular, if $H_u^2(G,M)\neq 0$, then ${\rm Br}_u(\bC(G^\prime))\neq 0$ 
and $\bC(G^\prime)$ is not retract $\bC$-rational. 
\end{theorem}
Note that the unramified Brauer groups are isomorphic 
for stably isomorphic fields by \cite[Proposition 2.2]{Sa4} 
or \cite[Proposition 1.2]{CTO}. 
In our situation $G=A_6$ as in Theorem \ref{t7.1}, 
we obtain: 
%
\begin{lemma}\label{lemN}
Let $M_i$ $(1\leq i\leq 6)$ 
be $A_6$-lattices with ${\rm rank}_\bZ M_i=9$ as in Theorem \ref{t7.1}. 
Then $M_i\oplus\bZ\leq \bZ[A_6/H]$ with finite index 
where $H\simeq (C_3\times C_3)\rtimes C_4$ 
and $N_i={\rm Hom}_\bZ(\bZ[A_6/H]/(M_i\oplus\bZ),\bQ/\bZ)$ 
$(1\leq i\leq 6)$ satisfy 
\begin{align*}
&N_1\simeq (\bZ/10\bZ)^{\oplus 9},\\
&N_2\simeq (\bZ/5\bZ)^{\oplus 4}\oplus(\bZ/10\bZ)^{\oplus 5},\\
&N_3\simeq (\bZ/2\bZ)^{\oplus 8}\oplus \bZ/10\bZ,\\
&N_4\simeq (\bZ/5\bZ)^{\oplus 8}\oplus\bZ/10\bZ,\\
&N_5\simeq (\bZ/2\bZ)^{\oplus 4}\oplus \bZ/10\bZ,\\
&N_6\simeq \bZ/10\bZ.
\end{align*}
\end{lemma}
\begin{proof}
We can check the assertion by applying the GAP function 
{\tt MultInvFieldToNoetherProblem} 
which is abailable from 
{\tt https://www.math.kyoto-u.ac.jp/\~{}yamasaki/Algorithm/MultInvField/} 
(see a demonstration of the GAP functions in the end of this section).
\end{proof}

By combining Theorem \ref{t7.1}, Theorem \ref{thSa4} 
and Lemma \ref{lemN}, we obtain a negative answer to 
Noether's problem for $N_1\rtimes A_6$ and $N_3\rtimes A_6$ 
over $\bC$ where $N_1$ and $N_3$ are as in Lemma \ref{lemN}: 
\begin{theorem}\label{thA6}
Let $N_1\simeq (C_{10})^9$ and 
$N_3\simeq (C_2)^8\times C_{10}$ be $A_6$-modules 
as in Lemma \ref{lemN}. 
Then, for $i=1,3$, ${\rm Br}_u(\bC(N_i\rtimes A_6))\simeq \bZ/2\bZ$ and 
Noether's problem for $N_i\rtimes A_6$ 
over $\bC$ has a negative answer. 
Moreover, $\bC(N_i\rtimes A_6)$ $(i=1,3)$ is not retract $\bC$-rational 
and hence it is not $($stably$)$ $\bC$-rational. 
\end{theorem}
\begin{remark}
For 
$N_6\simeq C_{10}$ 
as in Lemma \ref{lemN}, 
we see that $N_6$ is the trivial $A_6$-module. 
Hence we have 
$\bC(N_6\rtimes A_6)=\bC(C_{10}\times A_6)$. 
Because $\bC(C_{10})$ is $\bC$-rational, 
it follows from Kang and Plans \cite[Theorem 1.3]{KP} 
that 
$\bC(C_{10}\times A_6)$ is rational over $\bC(A_6)$. 
Hence, by Theorem \ref{thSa4}, 
our result $H_u^2(A_6,M_6)=0$ as in Theorem \ref{t7.1} 
gives an alternative proof of $B_0(A_6)=0$ 
provided by Kunyavskii \cite{Ku}. 
\end{remark}
The following is a demonstration of the GAP functions 
which is used in the proofs of Theorem \ref{t7.1} and 
Lemma \ref{lemN}. 
%
Before proceeding to the computations of GAP \cite{GAP} below,
we should read the related data from\\
{\tt https://www.math.kyoto-u.ac.jp/\~{}yamasaki/Algorithm/MultInvField/}.
\footnote{
We need 
CARAT package \cite{CARAT} of GAP for  
the command {\tt ZClassRepsQClass} and it 
works on Linux or macOS, but not on Windows.}
\\

\begin{verbatim}
gap> Read("MultInvField.gap");

gap> J:=Norm1TorusJ(10,26); # J=A6
<matrix group with 3 generators>
gap> StructureDescription(J);
"A6"
gap> H2nrM(J).H2nrM; # H2nr(A6,J)=Z/2Z
[ 2 ]
gap> Filtered(H1(J),x->x>1); # H1(A6,J)=0
[  ]
gap> Filtered(Hminus1(J),x->x>1); # H^{-1}(A6,J)=Z/10Z
[ 10 ]

gap> Jz:=ZClassRepsQClass(G);; # all the 6 Z-classes forming the same Q-class
gap> List(Jz,x->H2nrM(x).H2nrM); # structure of H2nr(G,M) 
[ [ 2 ], [  ], [ 2 ], [  ], [  ], [  ] ]
gap> List(Jz,x->Filtered(H1(x),y->y>1)); # structure of H1(G,M) 
[ [  ], [  ], [ 5 ], [ 2 ], [ 5 ], [ 10 ] ]
gap> List(Jz,x->Filtered(Hminus1(x),y->y>1)); # structure of H^{-1}(G,M) 
[ [ 10 ], [ 5 ], [ 2 ], [ 5 ], [  ], [  ] ]

gap> ZZ:=Group([[[1]],[[1]],[[1]]]); # Z
Group([ [ [ 1 ] ], [ [ 1 ] ], [ [ 1 ] ] ])
gap> JzZ:=List(Jz,x->DirectSumMatrixGroup([x,ZZ])); # J+Z
[ <matrix group with 3 generators>, <matrix group with 3 generators>, 
  <matrix group with 3 generators>, <matrix group with 3 generators>, 
  <matrix group with 3 generators>, <matrix group with 3 generators> ]
gap> JzZgen:=List(JzZ,GeneratorsOfGroup);; # generators of J+Z

gap> T10_26gen:=GeneratorsOfGroup(TransitiveGroup(10,26)); # generators of 10T26=A6
[ (1,2,10)(3,4,5)(6,7,8), (1,3,2,6)(4,5,8,7), (1,2)(4,7)(5,8)(9,10) ]
gap> Pgen:=List(T10_26gen,x->PermutationMat(x,10));; # generators of P

gap> NP:=List([1..6],x->MultInvFieldToNoetherProblem(JzZgen[x],Pgen));;
gap> List(NP,x->x.snf); # structures of Ni's
[ [ 1, 10, 10, 10, 10, 10, 10, 10, 10, 10 ], 
  [ 1, 5, 5, 5, 5, 10, 10, 10, 10, 10 ], 
  [ 1, 2, 2, 2, 2, 2, 2, 2, 2, 10 ], 
  [ 1, 5, 5, 5, 5, 5, 5, 5, 5, 10 ], 
  [ 1, 1, 1, 1, 1, 2, 2, 2, 2, 10 ], 
  [ 1, 1, 1, 1, 1, 1, 1, 1, 1, 10 ] ]

gap> NP[1].act; # action of A6 on N1
[ [ [ 1, 0, 0, 0, 0, 0, 0, 0, 0, 0 ], 
      [ 0, 0, 0, 0, 0, 0, 0, 0, 0, 1 ], 
      [ 0, 0, 0, 1, 0, 0, 0, 0, 0, 0 ], 
      [ 0, 0, 0, 0, 1, 0, 0, 0, 0, 0 ], 
      [ 0, 0, 1, 0, 0, 0, 0, 0, 0, 0 ], 
      [ 0, 0, 0, 0, 0, 0, 1, 0, 0, 0 ], 
      [ 0, 0, 0, 0, 0, 0, 0, 1, 0, 0 ], 
      [ 0, 0, 0, 0, 0, 1, 0, 0, 0, 0 ], 
      [ 0, 0, 0, 0, 0, 0, 0, 0, 1, 0 ], 
      [ 0, -1, -1, -1, -1, -1, -1, -1, -1, -1 ] ], 
  [ [ 1, 0, 0, 0, 0, 0, 0, 0, 0, 0 ], 
      [ 0, 0, 0, 0, 0, 1, 0, 0, 0, 0 ], 
      [ 0, 1, 0, 0, 0, 0, 0, 0, 0, 0 ], 
      [ 0, 0, 0, 0, 1, 0, 0, 0, 0, 0 ], 
      [ 0, 0, 0, 0, 0, 0, 0, 1, 0, 0 ], 
      [ 0, -1, -1, -1, -1, -1, -1, -1, -1, -1 ], 
      [ 0, 0, 0, 1, 0, 0, 0, 0, 0, 0 ], 
      [ 0, 0, 0, 0, 0, 0, 1, 0, 0, 0 ], 
      [ 0, 0, 0, 0, 0, 0, 0, 0, 1, 0 ], 
      [ 0, 0, 0, 0, 0, 0, 0, 0, 0, 1 ] ], 
  [ [ 1, 0, 0, 0, 0, 0, 0, 0, 0, 0 ], 
      [ 0, -1, -1, -1, -1, -1, -1, -1, -1, -1 ], 
      [ 0, 0, 1, 0, 0, 0, 0, 0, 0, 0 ], 
      [ 0, 0, 0, 0, 0, 0, 1, 0, 0, 0 ], 
      [ 0, 0, 0, 0, 0, 0, 0, 1, 0, 0 ], 
      [ 0, 0, 0, 0, 0, 1, 0, 0, 0, 0 ], 
      [ 0, 0, 0, 1, 0, 0, 0, 0, 0, 0 ], 
      [ 0, 0, 0, 0, 1, 0, 0, 0, 0, 0 ], 
      [ 0, 0, 0, 0, 0, 0, 0, 0, 0, 1 ], 
      [ 0, 0, 0, 0, 0, 0, 0, 0, 1, 0 ] ] ]
gap> NP[3].act; # action of A6 on N3
[ [ [ 1, 0, 0, 0, 0, 0, 0, 0, 0, 0 ], 
      [ 0, 0, 0, 0, 0, 0, 0, 0, 0, 5 ], 
      [ 0, 0, 0, 1, 0, 0, 0, 0, 0, 0 ], 
      [ 0, 0, 0, 0, 1, 0, 0, 0, 0, 0 ], 
      [ 0, 0, 1, 0, 0, 0, 0, 0, 0, 0 ], 
      [ 0, 0, 0, 0, 0, 0, 1, 0, 0, 0 ], 
      [ 0, 0, 0, 0, 0, 0, 0, 1, 0, 0 ], 
      [ 0, 0, 0, 0, 0, 1, 0, 0, 0, 0 ], 
      [ 0, 0, 0, 0, 0, 0, 0, 0, 1, 0 ], 
      [ 0, 1, 1, 1, 1, 1, 1, 1, 1, 1 ] ], 
  [ [ 1, 0, 0, 0, 0, 0, 0, 0, 0, 0 ], 
      [ 0, 0, 0, 0, 0, 1, 0, 0, 0, 0 ], 
      [ 0, 1, 0, 0, 0, 0, 0, 0, 0, 0 ], 
      [ 0, 0, 0, 0, 1, 0, 0, 0, 0, 0 ], 
      [ 0, 0, 0, 0, 0, 0, 0, 1, 0, 0 ], 
      [ 0, 1, 1, 1, 1, 1, 1, 1, 1, 5 ], 
      [ 0, 0, 0, 1, 0, 0, 0, 0, 0, 0 ], 
      [ 0, 0, 0, 0, 0, 0, 1, 0, 0, 0 ], 
      [ 0, 0, 0, 0, 0, 0, 0, 0, 1, 0 ], 
      [ 0, 0, 0, 0, 0, 0, 0, 0, 0, 1 ] ], 
  [ [ 1, 0, 0, 0, 0, 0, 0, 0, 0, 0 ], 
      [ 0, 1, 1, 1, 1, 1, 1, 1, 1, 5 ], 
      [ 0, 0, 1, 0, 0, 0, 0, 0, 0, 0 ], 
      [ 0, 0, 0, 0, 0, 0, 1, 0, 0, 0 ], 
      [ 0, 0, 0, 0, 0, 0, 0, 1, 0, 0 ], 
      [ 0, 0, 0, 0, 0, 1, 0, 0, 0, 0 ], 
      [ 0, 0, 0, 1, 0, 0, 0, 0, 0, 0 ], 
      [ 0, 0, 0, 0, 1, 0, 0, 0, 0, 0 ], 
      [ 0, 0, 0, 0, 0, 0, 0, 0, 0, 5 ], 
      [ 0, 0, 0, 0, 0, 0, 0, 0, 1, -4 ] ] ]
gap> NP[6].act=[IdentityMat(10),IdentityMat(10),IdentityMat(10)];
# A6 acts on N6 trivially
true
\end{verbatim}

\bigskip

\section{Some lattices of rank $2n+2, 4n$, and $p(p-1)$}\label{seHigher}

Motivated by Theorem \ref{t1.5} we will construct in this section lattices of rank $2n+2, 4n, p(p-1)$ with non-trivial unramified Brauer groups (where $n$ is any positive integer and $p$ is any odd prime number).

Before the construction, we note a variation of Theorem \ref{thSa4}.
{}From Theorem \ref{thSa4}, $H_u^2(G,M)$ may be  obtained as
\[
H_u^2(G,M)=\bigcap_{A} {\rm Ker}(
{\rm res} : H^2(G,M)\rightarrow H^2(A,M))
\]
where $A$ runs over {\it maximal} bicyclic subgroups of $G$. If $B\leq A$, then
${\rm Ker}(
{\rm res} : H^2(G,M)\rightarrow H^2(A,M))\leq
{\rm Ker}(
{\rm res} : H^2(G,M)\rightarrow H^2(B,M))$
(see also Definition \ref{d1.4}).
Hence in order to find $H_u^2(G,M)$, it suffices to evaluate
${\rm Ker}(
{\rm res} : H^2(G,M)\rightarrow H^2(H,M))$
for each maximal bicyclic subgroup $H\leq G$.

\begin{definition}\label{defD4n}
Let $G=\langle \sigma,\tau: \sigma^{4n}=\tau^2=1,
\tau^{-1}\sigma\tau=\sigma^{-1}\rangle\simeq D_{4n}$,
the dihedral group of order $8n$ where $n$ is any positive integer.
Define a $G$-lattice of rank $2n+2$ as follows:
$M=(\oplus_{1\leq i\leq 2n}\bZ\cdot x_i)\oplus(\oplus_{1\leq j\leq 2}\bZ
\cdot y_j)$ such that
\begin{align*}
\sigma&:x_1\mapsto x_2\mapsto\cdots\mapsto x_{2n}\mapsto -x_1,
y_1\mapsto y_2+x_1, y_2\mapsto y_1+x_1,\\
\tau&:x_i\leftrightarrow -x_{2n+1-i}\ (1\leq i\leq 2n),
y_1\mapsto -y_2, y_2\mapsto -y_1.
\end{align*}
In matrix forms with respect to the cases $x_1,\ldots,x_{2n}, y_1,y_2$,
the actions of $\sigma$ and $\tau$ are given as
\begin{align*}
\sigma=\left(
\begin{array}{cccc|cc}
 & 1 & & & & \\
 & & \ddots & & &\\
 &  & & 1 & &\\
 -1 & & & & &\\\hline
1 & 0 & \cdots & 0& 0 &1 \\
1 & 0 & \cdots & 0& 1 & 0
\end{array}
\right),
\tau=
\left(
\begin{array}{cccc|cc}
 &  & & -1 & &  \\
 & & \reflectbox{$\ddots$} & & &\\
 & \reflectbox{$\ddots$} & &  &  &\\
-1 & & & &  & \\\hline
 &  & & & 0 &-1 \\
 &  & & & -1 & 0
\end{array}
\right)\in GL_{2n+2}(\bZ).
\end{align*}
\end{definition}

The following theorem gives a generalization
of the case of the dihedral groups $D_4$ and $D_8$
with GAP ID (4,12,4,12) and CARAT ID (6,5574,11)
in Theorem \ref{t1.5} (2) and (4) respectively
(see Table 1 and Table 3-1-1). When $G=D_m$ the dihedral group of order $2m$, note that $\bm{C}(G)$ is $\bm{C}$-rational by \cite[Proposition 2.6]{CHK}. 
Thus $B_0(G)=0$ and ${\rm Br}_u(\bm{C}(M)^G)=H^2_u(G, M)$.


\begin{theorem}\label{thmD4n}
Let $G=\langle \sigma,\tau: \sigma^{4n}=\tau^2=1,
\tau^{-1}\sigma\tau=\sigma^{-1}\rangle\simeq D_{4n}$,
the dihedral group of order $8n$ where $n$ is any positive integer.
Let $M$ be the $G$-lattice of rank $2n+2$ defined in {\rm Definition} $\ref{defD4n}$.
Then $H_u^2(G,M)\simeq \bZ/2\bZ$.
Consequently, $\bC(M)^G$ is not retract $\bC$-rational.
\end{theorem}
\begin{proof}
As mentioned above, we will use the formula
\begin{align*}
H_u^2(G,M)=\bigcap_{H} {\rm Ker}(
{\rm res} : H^2(G,M)\rightarrow H^2(H,M))
\end{align*}
where $H$ runs over all the maximal bicyclic subgroups of $G$.
Note that there are three kinds of maximal bicyclic subgroups
$H$ of $G$: $H=\langle\sigma\rangle\simeq C_{4n}$,
$H=\langle\sigma^{2n},\sigma^{2i}\tau\rangle$ where $0\leq i\leq 2n$, and
$H=\langle\sigma^{2n},\sigma^{2i+1}\tau\rangle$ where $0\leq i\leq 2n$.
We will evaluate the subgroup
${\rm Ker}({\rm res}: H^2(G,M)\rightarrow H^2(H,M))$ for each of such
maximal bicyclic subgroups.
Define $M^\prime=\oplus_{1\leq i\leq 2n}\bZ\cdot x_i$ be the
$G$-sublattice of $M$.
We will prove $H_u^2(G,M)\simeq\bZ/2\bZ$ by showing\\
(a) $H^2(G,M^\prime)\simeq\bZ/2\bZ$, and\\
(b) the natural homomorphism $H^2(G,M^\prime)\rightarrow H^2(G,M)$
induces an isomorphism $H^2(G,M^\prime)\simeq H^2_u(G,M)$.\\

Step 1.
Let $N$ be a normal subgroup of $G$.
By Theorem \ref{thDHW}, we have the $7$-term exact sequence
\begin{align*}
0 &\to H^1(G/N,M^N) \to H^1(G,M)\to H^1(N,M)^{G/N} \to H^2(G/N,M^N) \\
&\to H^2(G,M)_1
\to H^1(G/N,H^1(N,M))
\to H^3(G/N,M^N)
\end{align*}
where $H^2(G,M)_1={\rm Ker}({\rm res}: H^2(G,M)\to H^2(N,M))$.

We recall a formula to evaluate $H^1(\pi,M)$ where
$\pi=\langle\sigma\rangle\simeq C_m$ and $M$ is a $\bZ[\pi]$-module.
Define a morphism ${\rm Norm}:M\rightarrow M$ by
$\mathrm{Norm}(x)=x+\sigma\cdot x+\sigma^2\cdot x+\cdots+\sigma^{m-1}\cdot x$
for any $x\in M$.
Let ${\rm Ker}(\mathrm{Norm})=\{x\in M:\mathrm{Norm}(x)=0\}$,
${\rm Image}(\sigma-1)=\{\sigma\cdot x-x:x\in M\}$.
If $\beta$ is a $1$-cocycle in $Z^1(\pi,M)$,
we denote by $[\beta]$ the associated cohomology class of $\beta$ in
$H^1(\pi,M)$.
The following result is well-known:
$H^1(\pi,M)\simeq{\rm Ker(Norm)}/{\rm Image}(\sigma-1)$.
The explicit isomorphism is given as follows:
For any $x\in {\rm Ker(Norm)}$,
define a $1$-cocycle $\beta_x$ by $\beta_x(\sigma)=x$,
$\beta_x(\sigma^i)=x+\sigma\cdot x+\cdots+\sigma^{i-1}\cdot x$
for $1\leq i\leq m$.
Then the isomorphism is given by $[\beta_x]\mapsto \overline{x}$.
For details, see Step 3 of the proof of Theorem 5.4 in
\cite[page 706]{HKKu}.

Now we turn to the group $\pi=\langle\sigma,\tau:\sigma^m=\tau^2=1,
\tau^{-1}\sigma\tau=\sigma^{-1}\rangle\simeq D_m$,
the dihedral group of order $2m$.
Then $\pi$ acts on $H^1(\langle\sigma\rangle,M)$ with
$\sigma$ acts on $H^1(\langle\sigma\rangle,M)$ trivially.
The action of $\tau$ on ${\rm Ker(Norm)}/{\rm Image}(\sigma-1)$ is
induced from that of $\tau$ on $H^1(\langle\sigma\rangle,M)$.
In particular, if $x\in {\rm Ker(Norm)}$ and $\beta$ is the $1$-cocyle
corresponding to $x$, i.e. $\beta=\beta_x$, then
${}^\tau\beta(\sigma)=\tau\cdot(\beta(\tau^{-1}\sigma\tau))
=\tau\cdot\beta(\sigma^{-1})=\tau(x+\sigma\cdot x+\cdots+\sigma^{n-2}\cdot x)
=-\tau\cdot(\sigma^{-1}\cdot x)$.
Note that $\tau\cdot(\sigma^{-1}\cdot x)=\sigma\cdot(\tau\cdot x)=[\sigma\cdot(\tau\cdot x)-\tau\cdot x]+\tau\cdot x$.
It follows that the cohomology class $[{}^\tau\beta]$ corresponds to
$-\overline{\tau\cdot x}\in {\rm Ker(Norm)}/{\rm Image}(\sigma-1)$.

Similarly, if $\pi=\langle\sigma,\tau:\sigma^2=\tau^2=1, \sigma\tau=\tau\sigma\rangle$ is the Klein four group and ${\rm Norm}(x)=x+\sigma\cdot x$ is the norm
map associated to the subgroup $\langle\sigma\rangle$ of $\pi$, then for any $x\in {\rm Ker(Norm)}$,
$\tau\cdot\overline{x}=\overline{\tau\cdot x}\in {\rm Ker(Norm)}/{\rm Image}(\sigma-1)$.\\

Step 2.
{}From now on till the end of the proof of this theorem,
$G=\langle\sigma,\tau:\sigma^{4n}
=\tau^2=1,\tau^{-1}\sigma\ta=\sigma^{-1}\rangle\simeq D_{4n}$ and $M$ is the $G$-lattice
defined in the theorem.

We claim that the natural map $H^2(G,M^\prime)\rightarrow H^2(G,M)$
is injective; in fact, this injection induces an isomorphism
$H^2(G,M^\prime)\xrightarrow{\sim}{\rm Ker}({\rm res}:H^2(G,M)\rightarrow H^2(\langle\sigma\rangle,M))$.
Moreover, $H^2(G,M^\prime)\simeq\bZ/2\bZ$.

(i) We will prove that $H^2(G,M^\prime)\rightarrow H^2(G,M)$ is injective.

Define $M^{\prime\prime}=M/M^\prime=\bZ\cdot y_1\oplus\bZ\cdot y_2$ where
$\sigma:y_1\leftrightarrow y_2$, $\tau:y_1\mapsto -y_2, y_2\mapsto -y_1$
(note that these $y_1,y_2\in M^{\prime\prime}$ are the images of
$y_1,y_2\in M$).

The exact sequence $0\rightarrow M^\prime\rightarrow M\rightarrow M^{\prime\prime}\rightarrow 0$ gives rise to the exact sequence
\begin{align*}
H^0(G,M^{\prime\prime})\rightarrow H^1(G,M^\prime)\rightarrow
H^1(G,M)\rightarrow H^1(G,M^{\prime\prime})\rightarrow H^2(G,M^\prime)\rightarrow H^2(G,M)\rightarrow\cdots.
\end{align*}
Note that $H^0(G,M^{\prime\prime})={M^{\prime\prime}}^G=0$.
As to $H^1(G,M^{\prime\prime})$, we use the exact sequence
\begin{align*}
0\rightarrow H^1(G/\langle\sigma\rangle, {M^{\prime\prime}}^{\langle\sigma\rangle})\rightarrow H^1(G,M^{\prime\prime})\rightarrow H^1(\langle\sigma\rangle,M^{\prime\prime})^{G/\langle\sigma\rangle}.
\end{align*}
Apply Step 1.
Note that $1+\sigma+\sigma^2+\cdots+\sigma^{4n-1}(y_i)=2n(y_1+y_2)$.
Hence $H^1(\langle\sigma\rangle,M^{\prime\prime})=0$.
It follows that $H^1(G, M^{\prime\prime})\simeq H^1(G/\langle\sigma\rangle, {M^{\prime\prime}}^{\langle\sigma\rangle})=H^1(G/\langle\sigma\rangle, \bZ\cdot(y_1+y_2))\simeq\bZ/2\bZ$.

Now we compute $H^1(G,M^\prime)$.
Because $\sigma^{2n}:x_i\mapsto -x_i$ for $1\leq i\leq 2n$,
it follows that ${M^\prime}^{\langle\sigma\rangle}=0$.
{}From the sequence
\begin{align*}
H^1(G/\langle\sigma\rangle, {M^{\prime}}^{\langle\sigma\rangle})\rightarrow H^1(G, M^{\prime})\rightarrow H^1(\langle\sigma\rangle, M^{\prime})^{G/\langle\sigma\rangle}\rightarrow H^2(G/\langle\sigma\rangle, {M^\prime}^{\langle\sigma\rangle}),
\end{align*}
we find that $H^1(G, M^\prime)\simeq H^1(\langle\sigma\rangle, M^\prime)^{G/\langle\sigma\rangle}$.
By Step 1, $H^1(\langle\sigma\rangle, M^\prime)\simeq{\rm Ker(Norm)}/{\rm Image}(\sigma-1)=\langle x_i:1\leq i\leq 2n\rangle/\langle x_1-x_2,\ldots,x_{2n-1}-x_{2n},x_{2n}+x_1\rangle\simeq\bZ/2\bZ\cdot \overline{x_1}$.
The action of $G/\langle\sigma\rangle\simeq\langle\tau\rangle$ is given by
$\tau\cdot\overline{x_1}=\overline{\tau\cdot x_1}=\overline{x_2}
=\overline{x_1}$ by Step 1 again.
Hence $H^1(G,M^\prime)\simeq H^1(\langle\sigma\rangle,M^\prime)^{G/\langle\sigma\rangle}\simeq\bZ/2\bZ$.

We turn to $H^1(G,M)$.
It is not difficult to verify that
$M^{\langle\sigma\rangle}=\bZ\cdot w$ where
$w=y_1+y_2+\sum_{1\leq i\leq 2n}x_i$.
Note that $\tau\cdot w=-w$.
{}From the exact sequence
\begin{align*}
0\rightarrow H^1(G/\langle\sigma\rangle, M^{\langle\sigma\rangle})\rightarrow H^1(G, M)\rightarrow H^1(\langle\sigma\rangle, M)^{G/\langle\sigma\rangle}\rightarrow H^2(G/\langle\sigma\rangle, M^{\langle\sigma\rangle}),
\end{align*}
we find that $H^1(G/\langle\sigma\rangle, M^{\langle\sigma\rangle})
=\bZ/2\bZ\cdot\overline{w}\simeq\bZ/2\bZ$,
$H^2(G/\langle\sigma\rangle, M^{\langle\sigma\rangle})
\simeq \widehat H^0(G/\langle\sigma\rangle, \bZ\cdot w)=0$,
$H^1(\langle\sigma\rangle, M)={\rm Ker(Norm)}/{\rm Image}(\sigma-1)
=\langle x_i,y_1-y_2:1\leq i\leq 2n\rangle/\langle x_1-x_2,\ldots, x_{2n-1}-x_{2n},x_{2n}+x_1,y_1-y_2+x_1\rangle\simeq \bZ/2\bZ\cdot\overline{x_1}$
with $\tau\cdot\overline{x_1}=\overline{x_1}$.
Hence $H^1(\langle\sigma\rangle, M)^{G/\langle\sigma\rangle}\simeq\bZ/2\bZ$
.
It follows that $H^1(G,M)$ is an abelian group of over $4$.

Now, in the exact sequence
\begin{align*}
 H^0(G,M^{\prime\prime})\rightarrow
H^1(G,M^\prime)\rightarrow
H^1(G,M)\xrightarrow{g_1}
H^1(G,M^{\prime\prime})\xrightarrow{g_2}
H^2(G,M^{\prime})\xrightarrow{g_3} H^2(G,M),
\end{align*}
we have $H^0(G,M^{\prime\prime})=0$, $H^1(G,M^\prime)\simeq\bZ/2\bZ\simeq
H^1(G,M^{\prime\prime})$, $|H^1(G,M)|=4$.
Thus $g_1$ is surjective and $g_2$ is the zero map.
It follows that $g_3$ is injective.

(ii) We will show that the above $g_3$ induces an injection of
$H^2(G,M^\prime)$ into
${\rm Ker}({\rm res}: H^2(G,M)\rightarrow H^2(\langle\sigma\rangle, M))$.

For this assertion, it suffices to show that the composite map
$H^2(G, M^\prime)\rightarrow H^2(G, M)\rightarrow H^2(\langle\sigma\rangle, M)$ is the zero map.
Consider the commutative diagram
\begin{align*}
\xymatrix@R=15pt@C=8pt{
& H^2(G,M^{\prime}) \ar[r] \ar[d]^{\rm res} & H^2(G,M) \ar[d]^{\rm res} \\
& H^2(\langle\sigma\rangle,M^{\prime}) \ar[r] & H^2(\langle\sigma\rangle,M).
}
\end{align*}
Since $H^2(\langle\sigma\rangle,M^\prime)\simeq\widehat H^0(\langle\sigma\rangle, {M^\prime}^{\langle\sigma\rangle})=0$
(remember that ${M^\prime}^{\langle\sigma\rangle}=0$, a fact verified in (i)),
it follows that $H^2(G,M^\prime)\rightarrow H^2(\langle\sigma\rangle, M)$
is the zero map.

(iii) We will show that $H^2(G,M^\prime)\simeq\bZ/2\bZ$.

Use the exact sequence
\begin{align*}
H^2(G/\langle\sigma\rangle, {M^\prime}^{\langle\sigma\rangle})&\to
L_1:={\rm Ker}({\rm res}:H^2(G,M^\prime)\to H^2(\langle\sigma\rangle, M^\prime))\\
&\to H^1(G/\langle\sigma\rangle, H^1(\langle\sigma\rangle, M^\prime))
\to H^3(G/\langle\sigma\rangle, {M^\prime}^{\langle\sigma\rangle}).
\end{align*}
As shown in (ii), $H^2(\langle\sigma\rangle, M^\prime)=0$.
Hence $L_1=H^2(G,M^\prime)$.
Since ${M^\prime}^{\langle\sigma\rangle}=0$, we find that
$H^2(G,M)\simeq H^1(G/\langle\sigma\rangle, H^1(\langle\sigma\rangle, M^\prime))$.

It remains to show that $H^1(G/\langle\sigma\rangle, H^1(\langle\sigma\rangle, M^\prime))\simeq\bZ/2\bZ$.
We have proved in (i) that $H^1(\langle\sigma\rangle, M^\prime)\simeq\bZ/2\bZ\cdot \overline{x_1}$ and $\tau\cdot\overline{x_1}=\overline{x_1}$.
Thus $H^1(G/\langle\sigma\rangle, H^1(\langle\sigma\rangle, M^\prime))=H^1(G/\langle\sigma\rangle, \bZ/2\bZ\cdot\overline{x_1})=\bZ/2\bZ\cdot \overline{x_1}\simeq \bZ/2\bZ$.

(iv)  We will show that $H^2(G, M^\prime)\rightarrow {\rm Ker}({\rm res}: H^2(G,M)\rightarrow H^2(\langle\sigma\rangle, M))$ is an isomorphism.

Call $L_2:={\rm Ker}({\rm res}:H^2(G,M)\rightarrow H^2(\langle\sigma\rangle, M))$.

By (ii), $H^2(G,M^\prime)$ is a subgroup of $L_2$.
Thus $|L_2|\geq 2$.
We will show that $|L_2|\leq 2$.
This will finish the proof that $H^2(G,M^\prime)\simeq L_2$.

Use the exact sequence
\begin{align*}
H^2(G/\langle\sigma\rangle, M^{\langle\sigma\rangle})\to
L_2:={\rm Ker}({\rm res}:H^2(G,M)\to H^2(\langle\sigma\rangle, M))
\to H^1(G/\langle\sigma\rangle, H^1(\langle\sigma\rangle, M)).
\end{align*}
We have proved in (i) that
$M^{\langle\sigma\rangle}=\bZ\cdot w$ where
$w=y_1+y_2+\sum_{1\leq i\leq 2n}x_i$ and $\tau\cdot w=-w$.
Hence $H^2(G/\langle\sigma\rangle, M^{\langle\sigma\rangle})
=H^2(G/\langle\sigma\rangle, \bZ\cdot w)\simeq\widehat H^0(\langle\tau\rangle, \bZ\cdot w)=0$.
Recall also the result $H^1(\langle\sigma\rangle, M)=\bZ/2\bZ\cdot\overline{x_1}$ with $\tau\cdot\overline{x_1}=\overline{x_1}$ which is proved in (i).
We find that $H^1(G/\langle\sigma\rangle, H^1(\langle\sigma\rangle, M))
\simeq H^1(\langle\tau\rangle, \bZ/2\bZ\cdot\overline{x_1})\simeq\bZ/2\bZ$.
Hence the result. \\

Step 3.
We will show that the injection
$g_3:H^2(G, M^\prime)\rightarrow H^2(G,M)$ (see (i) of Step 2)
induces an injection of $H^2(G, M^\prime)$ into ${\rm Ker}({\rm res}:
H^2(G,M)\rightarrow H^2(\langle\sigma^{2n},\sigma^{2i}\tau\rangle, M))$
for any $1\leq i\leq 2n$.

Define $L_3={\rm Ker}({\rm res}: H^2(G,M)\rightarrow H^2(\langle\sigma^{2n},\tau\rangle, M))$.

We will show that $H^2(G, M^\prime)$ is mapped into $L_3$ via $g_3$.
The proof for the situation of $\langle\sigma^{2n},\sigma^{2i}\tau\rangle$
(where $1\leq i\leq 2n-1$) is similar and is omitted.

To show that $g_3(H^2(G, M^\prime))\subset L_3$,
it suffices to show that $H^2(G, M^\prime)\xrightarrow{g_3} H^2(G,M)
\xrightarrow{\rm res} H^2(\langle\sigma^{2n},\tau\rangle, M)$ is the zero map.
Consider the commutative diagram
\begin{align*}
\xymatrix@R=15pt@C=8pt{
& H^2(G,M^{\prime}) \ar[r] \ar[d]^{\rm res} & H^2(G,M) \ar[d]^{\rm res} \\
& H^2(\langle\sigma^{2n},\tau\rangle,M^{\prime}) \ar[r] & H^2(\langle\sigma^{2n},\tau\rangle,M).
}
\end{align*}
We will show that $H^2(\langle\sigma^{2n},\tau\rangle, M^\prime)=0$.

Use the exact sequence
\[
H^2(\langle\sigma^{2n},\tau\rangle/\langle\sigma^{2n}\rangle, {M^\prime}^{\langle\sigma^{2n}\rangle})\rightarrow H^2(\langle\sigma^{2n},\tau\rangle, M^\prime)\rightarrow H^1(\langle\sigma^{2n},\tau\rangle/\langle\sigma^{2n}\rangle, H^1(\langle\sigma^{2n}\rangle, M^\prime)).
\]
Note that $\sigma^{2n}:x_i\mapsto -x_i$.
Hence ${M^\prime}^{\langle\sigma^{2n}\rangle}=0$ and
$H^2(\langle\sigma^{2n},\tau\rangle/\langle\sigma^{2n}\rangle, {M^\prime}^{\langle\sigma^{2n}\rangle})=0$.
On the other hand, $H^1(\langle\sigma^{2n}\rangle, M^\prime)\simeq
{\rm Ker}({\rm Norm})/{\rm Image}(\sigma-1)\simeq\oplus_{1\leq i\leq 2n}\bZ/2\bZ\cdot\overline{x_i}$ where $\tau\cdot\overline{x_i}=\overline{x_{2n-i}}$.

Thus $H^1(\langle\sigma^{2n},\tau\rangle/\langle\sigma^{2n}\rangle, H^1(\langle\sigma^{2n}\rangle, M^\prime))\simeq H^1(\langle\tau\rangle, \oplus_{1\leq i\leq 2n}\bZ/2\bZ\cdot\overline{x_i})=0$. Done.\\

Step 4.
We will show that the injection $g_3:H^2(G,M^\prime)\rightarrow H^2(G,M)$
induces an injection of $H^2(G,M^\prime)$ into $L_4:={\rm Ker}({\rm res}:H^2(G,M)\rightarrow H^2(\langle\sigma^{2n}, \sigma\tau\rangle,M))$.
The situation for $\langle\sigma^{2n}, \sigma^{2i+1}\tau\rangle$
(where $1\leq i\leq 2n-1$) is similar and is omitted.

Once this step is finishes, we find that
$g_3(H^2(G, M^\prime))\subset H_u^2(G,M)$.
On the other hand, by Step 2, we know that $g_3(H^2(G,M^\prime))=
{\rm Ker}({\rm res}:H^2(G,M)\rightarrow H^2(\langle\sigma\rangle, M))\supset H_u^2(G,M)$.
Hence $H_u^2(G,M)=g_3(H^2(G,M^\prime))\simeq\bZ/2\bZ$.

To show that $G_3(H^2(G, M^\prime))\subset L_4$, it suffices to prove that
$H^2(G, M^\prime)\xrightarrow{g_3}H^2(G,M)\xrightarrow{\rm res} H^2(\langle\sigma^{2n},\sigma\tau\rangle, M)$ is the zero map.
Consider the commutative diagram
\begin{align*}
\xymatrix@R=15pt@C=8pt{
& H^2(G,M^{\prime}) \ar[r]^{g_3} \ar[d]^{\rm res} & H^2(G,M) \ar[d]^{\rm res} \\
& H^2(\langle\sigma^{2n},\sigma\tau\rangle,M^{\prime}) \ar[r]^{h_3} & H^2(\langle\sigma^{2n},\sigma\tau\rangle,M).
}
\end{align*}
We will show that $h_3$ is the zero map.
Once we obtain this result, the proof is finished .

For convenience, write $H:=\langle\sigma^{2n},\sigma\tau\rangle$
and denote $\lambda_1=\sigma^{2n}$, $\lambda_2=\sigma\tau$,
$\lambda_3=\lambda_1\lambda_2$.
{}From the exact sequence $0\rightarrow M^\prime\rightarrow M\rightarrow M^{\prime\prime}\rightarrow 0$, we have the exact sequence
\[
\cdots\rightarrow H^1(H, M^{\prime\prime})\xrightarrow{h_2} H^2(H, M^\prime)
\xrightarrow{h_3} H^2(H, M)\rightarrow\cdots.
\]
We will prove that the connecting homomorphism $h_2$ is surjective.
Thus $h_3$ is the zero map.

(i) We recall the action of $H=\langle\lambda_1,\lambda_2\rangle$ on $M$.
It is given by
\begin{align*}
\lambda_1 &: x_i\mapsto -x_i\ (1\leq i\leq 2n),
y_j\mapsto y_j+\sum_{1\leq i\leq 2n}x_i\ (j=1,2),\\
\lambda_2 &: x_1\mapsto x_1, x_i\mapsto -x_{2n+2-i}\ (2\leq i\leq 2n),
y_j\mapsto -y_j-x_1\ (j=1,2),\\
\lambda_3 &: x_1\mapsto -x_1, x_i\mapsto x_{2n+2-i}\ (2\leq i\leq 2n),
y_j\mapsto -y_j+\sum_{2\leq i\leq 2n}x_i\ (j=1,2).
\end{align*}
The action of $H$ on $M^{\prime\prime}$ is even simpler:
For $1\leq j\leq 2$,
\begin{align*}
\lambda_1 : y_j\mapsto y_j,\quad
\lambda_2=\lambda_3 : y_j\mapsto -y_j.
\end{align*}

(ii) We will show that $H^2(H,M^\prime)\simeq\bZ/2\bZ$.

By the exact sequence
\begin{align*}
H^2(H/\langle\lambda_2\rangle, {M^\prime}^{\langle\lambda_2\rangle})&\to
{\rm Ker}({\rm res}:H^2(H,M^\prime)\to H^2(\langle\lambda_2\rangle, M^\prime))\\
&\to H^1(H/\langle\lambda_2\rangle, H^1(\langle\lambda_2\rangle, M^\prime))
\to H^3(H/\langle\lambda_2\rangle, {M^\prime}^{\langle\lambda_2\rangle})\to\cdots,
\end{align*}
since ${M^\prime}^{\langle\lambda_2\rangle}=0$, we find that
$H^2(H,M^\prime)\simeq H^1(H/\langle\lambda_2\rangle, H^1(\langle\lambda_2\rangle, M^\prime))$.

Using the same method as in Step 1, it is not difficult to prove
that $H^1(H/\langle\lambda_2\rangle, H^1(\langle\lambda_2\rangle, M^\prime))
\simeq\bZ/2\bZ$.

(iii) We will show that $H^1(H, M^{\prime\prime})\simeq (\bZ/2\bZ)^{\oplus 2}$
and find explicitly a $1$-cocycle $\gamma:H\rightarrow M^{\prime\prime}$
such that $[\gamma]\neq 0$ in $H^1(H,M^{\prime\prime})$.

{}From the exact sequence
\begin{align*}
0 \rightarrow H^1(H/\langle\lambda_2\rangle, {M^{\prime\prime}}^{\langle\lambda_2\rangle})\rightarrow
H^1(H, M^{\prime\prime})\xrightarrow{h_1}
H^1(\langle\lambda_2\rangle, {M^{\prime\prime}})^{H/\langle\lambda_2\rangle}
\rightarrow H^2(H/\langle\lambda_2\rangle, {M^{\prime\prime}}^{\langle\lambda_2\rangle})\rightarrow \cdots,
\end{align*}
since ${M^{\prime\prime}}^{\langle\lambda_2\rangle}=0$,
we find that $H^1(H,M^{\prime\prime})\xrightarrow{h_1} H^1(\langle\lambda_2\rangle, M^{\prime\prime})^{H/\langle\lambda_2\rangle}$ is an isomorphism.
Note that $h_1$ is the restriction map.

As in Step 1, it can be shown that $H^1(\langle\lambda_2\rangle, M^{\prime\prime})^{H/\langle\lambda_2\rangle}\simeq\bZ/2\bZ\cdot\overline{y_1}\oplus\bZ/2\bZ\cdot\overline{y_2}$.

The $1$-cochain $\overline{\gamma}:\langle\lambda_2\rangle\rightarrow M^{\prime\prime}$ defined by $\overline{\gamma}(1)=0$ and
$\overline{\gamma}(\lambda_2)=y_1$ is a $1$-cocycle and it corresponds to
the element $\overline{y_1}\in {\rm Ker}({\rm Norm})/{\rm Image}(\lambda_2-1)$.
We will lift the $1$-cocycle $\overline{\gamma}$ to a $1$-cocycle
$\gamma : H\rightarrow M^{\prime\prime}$.
Define $\gamma(1)=\gamma(\lambda_1)=0$, $\gamma(\lambda_2)=\gamma(\lambda_3)=y_1$.
It is easy to verify that $\gamma$ is a $1$-cocycle and its restriction to
$\langle\lambda_2\rangle$ is $\overline{\gamma}$.
Thus $[\gamma]\neq 0$.

(iv) We will find a $2$-cocycle $\alpha:H\times H\rightarrow M^\prime$
such that $h_2([\gamma])=[\alpha]\in H^2(H,M^\prime)$
(remember that $h_2$ is the connecting homomorphism).

By the standard method (see \cite[page 122]{HS}, for example),
the $2$-cocycle $\alpha$ can be found as follows.
For each $\lambda\in H$, find a preimage of $\gamma(\lambda)$ in $M$;
in other words, define a $1$-cochain $\beta:H\rightarrow M$ such that
$\beta(1)=0$ and the image of $\beta(\lambda)$ in $M^{\prime\prime}$
is $\gamma(\lambda)$.
Then define $\alpha:H\times H\rightarrow M$ by
\begin{align}
\alpha(\lambda_i,\lambda_j)=\lambda_i\cdot\beta(\lambda_j)-\beta(\lambda_i\lambda_j)+\beta(\lambda_i)\label{FormulaA}
\end{align}
where $\lambda_i, \lambda_j\in H$.
It can be shown that $\alpha(\lambda_i,\lambda_j)\in M^\prime$
and $\alpha$ is a $2$-cocycle, i.e. $[\alpha]\in H^2(H,M^\prime)$.

Return to the concrete case $\gamma:H\rightarrow M^{\prime\prime}$
where $\gamma(1)=\gamma(\lambda_1)=0$,
$\gamma(\lambda_2)=\gamma(\lambda_3)=y_1$.
We define $\beta:H\rightarrow M$ by $\beta(1)=\beta(\lambda_1)=0$,
$\beta(\lambda_2)=\beta(\lambda_3)=y_1$.

Remember that $\beta(\lambda_i)\in M$ (instead of $M^{\prime\prime})$.
By Formula (\ref{FormulaA}), we find that the normalized $2$-cocycle
$\alpha: H\times H\rightarrow M^\prime$ is given by
\begin{align*}
&\alpha(\lambda_1,\lambda_1)=0,\
\alpha(\lambda_2,\lambda_2)=-x_1,\
\alpha(\lambda_3,\lambda_3)=\sum_{2\leq i\leq 2n}x_i,\\
&\alpha(\lambda_1,\lambda_2)=\sum_{1\leq i\leq 2n}x_i,\
\alpha(\lambda_2,\lambda_1)=0,\
\alpha(\lambda_1,\lambda_3)=\sum_{1\leq i\leq 2n}x_i,\\
&\alpha(\lambda_3,\lambda_1)=0,\
\alpha(\lambda_2,\lambda_3)=-x_1,\
\alpha(\lambda_3,\lambda_2)=\sum_{2\leq i\leq 2n}x_i.
\end{align*}

(v) We will show that $[\alpha]\neq 0$ in $H^2(H,M^\prime)$.

Assume this result.
Since $H^2(H,M^\prime)\simeq\bZ/2\bZ$ by (ii), it follows that $H^2(H,M^\prime)$ is generated by $[\alpha]$.
As $[\alpha]=h_2([\gamma])$, we find that $h_2$ is surjective.

It remains to show that $\alpha$ is not a $2$-coboundary. Suppose not.
There is a normalized $1$-cochain $\delta:H\rightarrow M^\prime$
whose differential is $\alpha$.
Write $\delta(\lambda_i)=v_i\in M^\prime$.
Then the differential of $\delta$ is given by
\begin{align*}
(\lambda_i,\lambda_j)\mapsto\lambda_i\cdot v_j-\delta(\lambda_i\lambda_j)+v_i.
\end{align*}

Write $v_2=\sum_{1\leq i\leq 2n}c_ix_i$. Because $\delta$ is a normalized $1$-cochain, we have $\delta (\lambda^2)= \delta(1)=0$. Thus the differential of $\delta$ for $(\lambda_2,\lambda_2)$ is given by $(\lambda_2,\lambda_2)\mapsto\lambda_2\cdot (\sum_{1\leq i\leq 2n}c_ix_i)+\sum_{1\leq i\leq 2n}c_ix_i=2c_1x_1$.
If $\alpha$ is the differential of $\delta$, we get
$\alpha(\lambda_2,\lambda_2)=2c_1x_1$.
But we have shown that $\alpha(\lambda_2,\lambda_2)=-x_1\not\in 2M^\prime$.
Hence $[\alpha]\neq 0$ in $H^2(H,M^\prime)$.
\end{proof}

Now we turn to the quasi-dihedral group of order $16n$,
$QD_{8n}=\langle \sigma,\tau: \sigma^{8n}=\tau^2=1,
\tau^{-1}\sigma\tau=\sigma^{4n-1}\rangle$
where $n$ is any positive integer.
Note that the subgroup $\langle\sigma^2,\sigma\tau\rangle$ of $QD_{8n}$
satisfies the relations $(\sigma^2)^{4n}=(\sigma\tau)^4=1$,
$(\sigma^2)^{2n}=(\sigma\tau)^2$ and
$(\sigma\tau)^{-1}\cdot\sigma^2\cdot(\sigma\tau)=(\sigma^2)^{-1}$.
Thus the generalized quaternion group of order $8n$,
denoted by $Q_{8n}$, may also be defined as the subgroup
$\langle\sigma^2,\sigma\tau\rangle$ of $QD_{8n}$.

\begin{remark}
Some authors define the quasi-dihedral groups for $2$-groups only. However, we allow the order of quasi-dihedral groups may be $16, 32, 48, 64, 80,$ etc..
\end{remark}


\begin{definition}\label{defQDQ}
(1) Let $n$ be any positive integer and $G=\langle \sigma,\tau: \sigma^{8n}=\tau^2=1,
\tau^{-1}\sigma\tau=\sigma^{4n-1}\rangle\simeq QD_{8n}$
be the quasi-dihedral group of order $16n$.
Define a $G$-lattice $M$ of rank $4n$ as follows:
$M=\oplus_{1\leq i\leq 2n}(\bZ\cdot x_i\oplus\bZ\cdot y_i)$
such that
\begin{align*}
\sigma:&\  x_1\mapsto y_1\mapsto x_2\mapsto y_2\mapsto\cdots\mapsto
x_{2n}\mapsto y_{2n}\mapsto -x_1,\\
\tau:&\ x_1\mapsto (x_1-y_1)+(x_2-y_2)+\cdots+(x_{2n-1}-y_{2n-1})+x_{2n},\\
&\ y_1\mapsto (x_1-y_1)+(x_2-y_2)+\cdots+(x_{2n-1}-y_{2n-1})+y_{2n},\\
&\ x_i\mapsto (x_1-y_1)+\cdots+(x_{2n-i}-y_{2n-i})+x_{2n+1-i}-(x_{2n+2-i}-y_{2n+2-i})-\cdots-(x_{2n}-y_{2n}),\\
&\ y_i\mapsto (x_1-y_1)+\cdots+(x_{2n-i}-y_{2n-i})+y_{2n+1-i}-(x_{2n+2-i}-y_{2n+2-i})-\cdots-(x_{2n}-y_{2n}),\\
&{\rm where}\ 2\leq i\leq 2n-1,\\
&\ x_{2n}\mapsto x_1-(x_2-y_2)-\cdots-(x_{2n}-y_{2n}),\\
&\ y_{2n}\mapsto y_1-(x_2-y_2)-\cdots-(x_{2n}-y_{2n}).
\end{align*}
In matrix forms with respect to the bases $x_1,y_1,\ldots,x_{2n},y_{2n}$,
the actions of $\sigma$ and $\tau$ are given as
\begin{align*}
\sigma=\left(
\begin{array}{cccc}
  & 1 &  &  \\
& & \ddots &\\
& & & 1 \\
-1 &  & &
\end{array}
\right),\
\tau=
\left(
\begin{array}{cccc}
 A & \cdots & A & I \\
 \vdots & \reflectbox{$\ddots$} & \reflectbox{$\ddots$} & -A \\
 A & \reflectbox{$\ddots$} & \reflectbox{$\ddots$} & \vdots \\
I & -A & \cdots & -A
\end{array}
\right),
\end{align*}
where
\begin{align*}
A=
\left(
\begin{array}{cc}
1 & -1\\
1 & -1
\end{array}
\right),\
I=
\left(
\begin{array}{cc}
1 & 0\\
0 & 1
\end{array}
\right)\in GL_2(\bZ).
\end{align*}
(2) Let $\widehat G=\langle\sigma^2,\sigma\tau\rangle\simeq Q_{8n}\leq G$
be the generalized quaternion group of order $8n$ where $n$ is any positive integer.
Define a $G$-lattice $\widehat M$ by
the restriction of $G$-lattice $M$ to $\widehat G$, i.e.
$\widehat M={\rm Res}^G_{\widehat G}(M)$.
\end{definition}
%

The following theorem is a generalization
of the cases $QD_8$ and $Q_8$
of the GAP IDs (4,32,3,2) and (4,32,1,2)
in Theorem \ref{t1.5} (2) respectively (see Table 1). When $G=QD_{8n}$ the quasi-dihedral group of order $16n$, note that $\bm{C}(G)$ is $\bm{C}$-rational by similar proof as in \cite[Proposition 2.6]{CHK}. Thus $B_0(G)=0$ and ${\rm Br}_u(\bm{C}(M)^G)=H^2_u(G, M)$. The similar result and argument are valid for the case $G=Q_{8n}$ the generalized quaternion group of order $8n$. In short, we will concentrate on $H^2_u(G, M)$ for both groups $G=QD_{8n}$ and $G=Q_{8n}$.

\begin{theorem}\label{thmQDQ}
{\rm (1)} Let $n$ be any positive integer and $G=\langle \sigma,\tau: \sigma^{8n}=\tau^2=1,
\tau^{-1}\sigma\tau=\sigma^{4n-1}\rangle\simeq QD_{8n}$
be the quasi-dihedral group of order $16n$.
Let $M$ be the $G$-lattice of rank $4n$ defined in {\rm Definition} $\ref{defQDQ}$.
Then $H_u^2(G,M)\simeq \bm{Z}/2\bm{Z}$.
Consequently, $\bm{C}(M)^{G}$ is not retract $\bm{C}$-rational.\\
{\rm (2)}
Let $\widehat G=\langle\sigma^2,\sigma\tau\rangle\simeq Q_{8n}\leq G$ be the
generalized quaternion group of order $8n$.
Let $\widehat M={\rm Res}^G_{\widehat G}(M)$ be the $\widehat G$-lattice
of rank $4n$ defined in {\rm Definition} $\ref{defQDQ}$.
Then $H_u^2(\widehat G,\widehat M)\simeq\bm{Z}/2\bm{Z}\oplus\bm{Z}/2\bm{Z}$.
Consequently,
$\bm{C}(\widehat M)^{\widehat G}$ is not retract $\bm{C}$-rational.
\end{theorem}
\begin{proof}
The proof is similar to that of Theorem \ref{thmD4n}.

(1) The case $G=\langle \sigma, \tau \rangle$.

There are three kinds of maximal bicyclic subgroups $H$ of $G$:
$H=\langle\sigma\rangle\simeq C_{8n}$,
$H=\langle\sigma^{4k+1}\tau\rangle\simeq C_4$ or
$\langle\sigma^{4k+3+4n}\tau\rangle\simeq C_4$, and
$H=\langle\sigma^{4n},\sigma^{2k}\tau\rangle\simeq C_2\times C_2$.

Note that $\sigma^{4n}:x_i\mapsto -x_i$, $y_i\mapsto -y_i$.
Thus $M^{\langle\sigma^{4n}\rangle}=0$.

By the same method as in the proof of Theorem \ref{thmD4n},
we find that $H^2(G,M)={\rm Ker}({\rm res}:H^2(G,M)\rightarrow H^2(\langle\sigma\rangle,M))$ and ${\rm Ker}({\rm res}:H^2(G,M)\rightarrow H^2(\langle\sigma\rangle, M))\simeq H^1(G/\langle\sigma\rangle, H^1(\langle\sigma\rangle,M))\simeq\bZ/2\bZ\cdot\overline{x_1}\simeq\bZ/2\bZ$.

Since $(\sigma^{4k+1}\tau)^2=\sigma^{4n}=(\sigma^{4k+3+4n}\tau)^2$,
it follows that $H^2(\langle\sigma^{4k+1}\tau\rangle,M)=0=H^2(\langle\sigma^{4k+3-4n}\tau\rangle, M)$.
Hence $H^2(G,M)={\rm Ker}({\rm res}:H^2(G,M)\rightarrow H^2(H,M))$
where $H=\langle\sigma^{4k+1}\tau\rangle$ or $\langle\sigma^{4k+3+4n}\tau\rangle$.

Finally, consider the situation ${\rm Ker}({\rm res}:H^2(G,M)\rightarrow H^2(\langle\sigma^{4n},\sigma^{2k}\tau\rangle,M)$.
We will show that $H^2(\langle\sigma^{4n},\tau\rangle, M)=0$.
The proof of $H^2(\langle\sigma^{4n},\sigma^{2k}\tau\rangle, M)=0$
(where $1\leq k\leq 4n-1$)
is similar; alternatively, since $\langle\sigma^{4n},\tau\rangle$ and
$\langle\sigma^{4n},\sigma^{2k}\tau\rangle$ are conjugate, we may apply
\cite[page 116, Proposition 3]{Se} to finish the proof.

Once the above assertion is proved, we find that
$H^2(G,M)={\rm Ker}({\rm res}:H^2(G,M)\rightarrow H^2(\langle\sigma^{4n},\sigma^{2k}\tau\rangle,M))$ again.
Hence $H^2(G,M)=H_u^2(G,M)$ and $H_u^2(G,M)\simeq\bZ/2\bZ$.

Now we proceed to prove $H^2(\langle\sigma^{4n},\tau\rangle,M)=0$.

Use the exact sequence
\begin{align*}
\cdots&\rightarrow H^2(\langle\sigma^{4n},\tau\rangle/\langle\sigma^{4n}\rangle,M^{\langle\sigma^{4n}\rangle})\rightarrow H^2(\langle\sigma^{4n},\tau\rangle,M)\rightarrow H^1(\langle\sigma^{4n},\tau\rangle/\langle\sigma^{4n}\rangle, H^1(\langle\sigma^{4n}\rangle,M))\\
&\rightarrow H^3(\langle\sigma^{4n},\tau\rangle/\langle\sigma^{4n}\rangle,M^{\langle\sigma^{4n}\rangle})\rightarrow\cdots.
\end{align*}
As $M^{\langle\sigma^{4n}\rangle}=0$, it follows that $H^2(\langle\sigma^{4n},\tau\rangle,M)\simeq H^1(\langle\sigma^{4n},\tau\rangle/\langle\sigma^{4n}\rangle,H^1(\sigma^{4n}\rangle,M))$.

Note that $H^1(\langle\sigma^{4n}\rangle,M)\simeq{\rm Ker}(1+\sigma^{4n})/{\rm Image}(1-\sigma^{4n})=\langle x_i,y_i:1\leq i\leq 2n\rangle/\langle 2x_i,2y_i: 1\leq i\leq 2n\rangle\simeq\oplus_{1\leq i\leq 2n}(\bZ/2\bZ\cdot\overline{x_i}\oplus\bZ/2\bZ\cdot\overline{y_i})$ where $\tau\cdot\overline{x_i}=\overline{\tau\cdot x_i}$, $\tau\cdot\overline{y_i}=\overline{\tau\cdot y_i}$ by Step 1
in the proof of Theorem \ref{thmD4n}.

Since the coefficient ring of $\oplus_{1\leq i\leq 2n}(\bZ/2\bZ\cdot\overline{x_i}\oplus\bZ/2\bZ\cdot\overline{y_i})$ is $\bZ/2\bZ$, we find that
$\overline{\tau\cdot x_i}=\overline{x_{2n+1-i}}+\sum_{j\neq 2n+1-i}\overline{x_j+y_j}$.
Similarly for $\overline{\tau\cdot y_i}$.
Define $\overline{u_i}=\overline{x_{2n+1-i}}+\sum_{1\leq j\leq 2n\atop j\neq 2n+1-i}(\overline{x_j}+\overline{y_j})$ for $1\leq i\leq n$, $\overline{v_i}=\overline{y_{2n+1-i}}+\sum{1\leq j\leq 2n\atop j\neq 2n+1-i}(\overline{x_j}+\overline{y_j})$ for $1\leq i\leq n$.
Then $\tau:\overline{x_i}\leftrightarrow \overline{u_i}$,
$\overline{y_i}\leftrightarrow\overline{v_i}$.

Let $T$ be the coefficient matrix of $x_1$, $y_1$, $x_2$, $y_2,\ldots, x_n$, $y_n$, $u_1$, $v_1$, $u_2$, $v_2,\ldots,u_n$, $v_n$ with respect to the basis
$x_1$, $y_1,\ldots,x_{2n}$, $y_{2n}$.
For example, when $n=2$, $T$ is of the form
\begin{align*}
\left(
\begin{array}{cccc|cccc}
1&0&0&0&0&0&0&0\\
0&1&0&0&0&0&0&0\\
0&0&1&0&0&0&0&0\\
0&0&0&1&0&0&0&0\\\hline
1&1&1&1&1&1&1&0\\
1&1&1&1&1&1&0&1\\
1&1&1&1&1&0&1&1\\
1&1&1&1&0&1&1&1
\end{array}
\right).
\end{align*}
It is easy to verify that ${\rm det}(T)=1\in\bZ/2\bZ$.
Thus $\langle\overline{x_1},\overline{y_1},\ldots,\overline{x_n},\overline{y_n},\overline{u_1},\overline{v_1},\ldots,\overline{u_n},\overline{v_n}\rangle=\langle\overline{x_1},\overline{y_1},\ldots,\overline{x_{2n}},\overline{y_{2n}}\rangle$ over $\bZ/2\bZ$.

Note that $H^1(\langle\tau\rangle,\bZ/2\bZ\cdot\overline{x_i}\oplus\bZ/2\bZ\cdot\overline{u_i})=0=H^1(\langle\tau\rangle,\bZ/2\bZ\cdot\overline{y_i}\oplus\bZ/2\bZ\cdot\overline{v_i})$ for $1\leq i\leq n$.

Hence $H^1(\langle\sigma^{4n},\tau\rangle/\langle\sigma^{4n}\rangle, H^1(\langle\sigma^{4n}\rangle, M))=0$.

\medskip
(2) Now consider the case $\widehat{G}=\langle\sigma^2,\sigma\tau\rangle$ and $\widehat{M}={\rm Res}^G_{\widehat{G}}(M)$.

 There are three kinds of maximal bicyclic subgroups $H$ of $\widehat{G}$:
$H=\langle\sigma^2\rangle\simeq C_{4n}$,
$H=\langle\sigma^{4k+1}\tau\rangle\simeq C_4$, and
$H=\langle\sigma^{4k+3}\tau\rangle\simeq C_4$.

It is not difficult to show that $H^2(\widehat{G}, \widehat{M})={\rm Ker}({\rm res}:H^2(\widehat{G}, \widehat{M})\rightarrow H^2(\langle\sigma^2\rangle, \widehat{M}))\simeq H^1(\widehat{G}/\langle\sigma^2\rangle, H^1(\langle\sigma^2\rangle, \widehat{M}))$.
But $H^1(\langle\sigma^2\rangle, \widehat{M})\simeq\bZ/2\bZ\cdot\overline{x_1}\oplus\bZ/2\bZ\cdot\overline{y_1}$ with $\sigma\tau:\overline{x_1}\mapsto\overline{x_1}$, $\overline{y_1}\mapsto\overline{y_1}$.
Hence $H^1(\widehat{G}/\langle\sigma^2\rangle, H^1(\langle\sigma^2\rangle, \widehat{M})\simeq\bZ/2\bZ\oplus\bZ/2\bZ$.

Moreover, $H^2(\langle\sigma^{4k+1}\tau\rangle, \widehat{M})=0=H^2(\langle\sigma^{4k+3}\tau\rangle, \widehat{M})$ because $(\sigma^{4k+1}\tau)^2=\sigma^{4n}=(\sigma^{4k+3}\tau)^2$ and $\widehat{M}^{\langle\sigma^{4n}\rangle}=0$.
Thus $H_u^2(\widehat{G}, \widehat{M})=H^2(\widehat{G}, \widehat{M})\simeq\bZ/2\bZ\oplus\bZ/2\bZ$.
\end{proof}

We will construct a $G$-lattice of rank $p(p-1)$ 
with ${\rm Br}_u(\bm{C}(M)^G)=H^2_u(G, M)\simeq \bZ/2\bZ$ 
where $G\simeq C_{p^2}\rtimes C_p$. 

\begin{definition}\label{defCp2p}
Let $p$ be an odd prime number,
$G=\langle \sigma,\tau: \sigma^{p^2}=\tau^p=1,
\tau^{-1}\sigma\tau=\sigma^{p+1}\rangle\simeq C_{p^2}\rtimes C_p$.
Define a $G$-lattice $M$ of rank $p(p-1)$ as follows:
$M=\oplus_{0\leq i\leq p-1}(\oplus_{1\leq j\leq p-1}\bZ\cdot x_j^{(i)})$
and $x_0^{(i)}=-\sum_{1\leq j\leq p-1}x_j^{(i)}$ such that
\begin{align*}
\sigma:&\  x_j^{(i)}\mapsto x_j^{(i+1)}\quad {\rm for}\quad 0\leq i\leq p-2,\
0\leq j\leq p-1,\\
&\ x_j^{(p-1)}\mapsto x_{j+1}^{(0)}\quad {\rm for}\quad 0\leq j\leq p-1,\\
\tau:&\ x_j^{(0)}\mapsto -\sum_{1\leq i\leq p-1}x_j^{(i)},\\
&\ x_j^{(k)}\mapsto -\sum_{0\leq i\leq k-1}x_{j+k+1}^{(i)}-\sum_{k+1\leq i\leq p-1}x_{j+k}^{(i)}\quad {\rm for}\quad 1\leq k\leq p-2,\ 0\leq j\leq p-1,\\
&\ x_j^{(p-1)}\mapsto -\sum_{0\leq i\leq p-2}x_j^{(i)}\quad {\rm for}\quad 0\leq j\leq p-1
\end{align*}
where the $i,i$ in $x_j^{(i)}$ are taken modulo $p$
(note that $\sum_{0\leq j\leq p-1}x_j^{(i)}=0$ for any $0\leq i\leq p-1$).

In matrix forms with respect to the bases $x_1^{(0)},\ldots,x_{p-1}^{(0)},x_1^{(1)},\ldots,x_{p-1}^{(1)},\ldots,x_1^{(p-1)},\ldots,x_{p-1}^{(p-1)}$,
the actions of $\sigma$ and $\tau$ are given as
\begin{align*}
\sigma=\left(
\begin{array}{cccc}
  & I &  &  \\
& & \ddots &\\
& & & I \\
A &  & &
\end{array}
\right),\
\tau=
\left(
\begin{array}{cccccc}
 O & -I & -I & \cdots & -I & -I  \\
 -A^2 & O & -A & \cdots & -A & -A\\
 -A^3 & -A^3 & O & -A^2 & \cdots & -A^2\\
 \vdots & \vdots & \ddots & \ddots & \ddots & \vdots \\
-A^{p-1} & -A^{p-1} & \cdots & -A^{p-1} & O & -A^{p-2}\\
-I & -I & \cdots & -I & -I & O
\end{array}
\right)
\end{align*}
where
\begin{align*}
A=
\left(
\begin{array}{cccc}
  & 1 &  &  \\
& & \ddots &\\
& & & 1 \\
-1 & -1 & \cdots & -1
\end{array}
\right)\in GL_{p-1}(\bZ)
\end{align*}
and $I$ (resp. $O$) $\in GL_{p-1}(\bZ)$
is the identity (resp. the zero) matrix.
\end{definition}

The following theorem gives a generalization
of the case $G=C_9\rtimes C_3$ of the CARAT ID (6,2865,3)
of Theorem \ref{t1.5} (4) (see Table 3-1-1). 
When $G$ is the group order $p^3$ in Theorem \ref{thmCp2p}, 
note that $\bm{C}(G)$ is $\bm{C}$-rational by \cite[Theorem 2.3]{CK}. 
Thus $B_0(G)=0$ and ${\rm Br}_u(\bm{C}(M)^G)=H^2_u(G, M)$.

\begin{theorem}\label{thmCp2p}
Let $p$ be an odd prime number and
$G=\langle \sigma,\tau: \sigma^{p^2}=\tau^p=1,
\tau^{-1}\sigma\tau=\sigma^{p+1}\rangle\simeq C_{p^2}\rtimes C_p$.
Let $M$ be the $G$-lattice of rank $p(p-1)$
defined in {\rm Definition} $\ref{defCp2p}$.
Then $H_u^2(G,M)\simeq\bZ/p\bZ$.
Consequently, $\bm{C}(M)^{G}$ is not retract $\bm{C}$-rational.\\
\end{theorem}
\begin{proof}
There are $p+1$ maximal bicyclic subgroups of $G$: $\langle\sigma\rangle\simeq C_{p^2}$, $\langle\sigma^p,\sigma^i\tau\rangle\simeq C_p\times C_p$ where $0\leq i\leq p-1$.

\medskip
Step 1. We will prove that $H^2(G,M)={\rm Ker}({\rm res}:H^2(G,M)\rightarrow H^2(\langle\sigma\rangle, M))\simeq\bZ/p\bZ$.

Note that $\sigma^p:x_j^{(i)}\mapsto x_{j+1}^{(i)}$ for $0\leq i,j\leq p-1$.
However, the linear map $y_1\mapsto y_2\mapsto \cdots\mapsto y_{p-1}\mapsto -y_1-y_2-\cdots-y_{p-1}$ has no eigen value equal to $1$.
Hence $M^{\langle\sigma^p\rangle}=0$.

Use the exact sequence
\begin{align*}
\cdots\to
&\ H^2(G/\langle\sigma^p\rangle, M^{\langle\sigma\rangle})\to
{\rm Ker}({\rm res}:H^2(G, M)\to H^2(\langle\sigma^p\rangle, M))\\
\to &\ H^1(G/\langle\sigma^p\rangle, H^1(\langle\sigma^p\rangle, M))
\to H^3(G/\langle\sigma^p\rangle, M^{\langle\sigma^p\rangle}).
\end{align*}
We find that $H^2(G,M)={\rm Ker}({\rm res}:H^2(G,M)\rightarrow H^2(\langle\sigma^p\rangle, M))\xrightarrow{\sim} H^1(G/\langle\sigma^p\rangle, H^1(\langle\sigma^p\rangle,M))$.

It remains to show that $H^1(G/\langle\sigma^p\rangle, H^1(\langle\sigma^p\rangle, M))\simeq\bZ/p\bZ$.

\medskip
By Step 1 in the proof of Theorem \ref{thmD4n}, $H^1(\langle\sigma^p\rangle,M)={\rm Ker}({\rm Norm})/{\rm Image}(\sigma^p-1)=\langle x_j^{(i)}:0\leq i,j\leq p-1\rangle/\langle x_j^{(i)}-x_{j+1}^{(i)}:0\leq i,j\leq p-1\rangle\simeq\oplus_{0\leq i\leq p-1}\bZ/p\bZ\cdot\overline{x_1^{(i)}}$
(note that $x_{p-1}^{(i)}-x_0^{(i)}=2x_{p-1}^{(i)}+ \sum_{1 \le j \le p-2}x_j^{(i)}$ and $p x_{p-1}^{(i)}\in {\rm Image}(\sigma^p-1))$ and
$\sigma\cdot\overline{x_1^{(i)}}=\overline{x_1^{(i+1)}}$,
$\tau\cdot\overline{x_1^{(i)}}=-\sum_{0\leq l\leq p-1\atop l\neq i}\overline{x_1^{(l)}}$.

Now we compute $H^1(G/\langle\sigma^p\rangle, \oplus_{0\leq i\leq p-1}\bZ/p\bZ\cdot\overline{x_1^{(i)}})$.
Call $M_1=\oplus_{0\leq i\leq p-1}\bZ/p\bZ\cdot\overline{x_1^{(i)}}$.
We use the exact sequence
\begin{align*}
0\rightarrow H^1(G/\langle\sigma\rangle, M_1^{\langle\sigma\rangle})\rightarrow H^1(G/\langle\sigma^p\rangle, M_1)\rightarrow H^1(\langle\sigma\rangle/\langle\sigma^p\rangle, M_1)^{G/\langle\sigma\rangle}.
\end{align*}
Define $w:=\sum_{0\leq i\leq p-1}\overline{x_1^{(i)}}$.
It is routine to verify that $M_1^{\langle\sigma\rangle}=\bZ/p\bZ\cdot w$ and
$\tau\cdot w=w$.
Hence $H^1(G/\langle\sigma\rangle, M_1^{\langle\sigma\rangle})\simeq\bZ/p\bZ\cdot w\simeq\bZ/p\bZ$.

On the other hand, it is also routine to show that if $v=\sum_{0\leq i\leq p-1}c_i\overline{x_1^{(i)}}\in\oplus_{0\leq i\leq p-1}\bZ/p\bZ\cdot\overline{x_1^{(i)}}$ and $(1+\sigma+\cdots+\sigma^{p-1})(v)=0$,
then $\sum_{0\leq i\leq p-1}c_i=0$.
It follows that $H^1(\langle\sigma\rangle,M_1)={\rm Ker}({\rm Norm})/{\rm Image}(\sigma-1)=\langle \overline{x_1^{(i)}}-\overline{x_1^{(i+1)}}:0\leq i\leq p-1\rangle/\langle \overline{x_1^{(i)}}-\overline{x_1^{(i+1)}}:0\leq i\leq p-1\rangle=0$.
Thus $H^1(G/\langle\sigma^p\rangle, M_1)\simeq H^1(G/\langle\sigma\rangle, M_1^{\langle\sigma\rangle})\simeq\bZ/p\bZ$.

\medskip
Step 2.
We will prove that $H^2(\langle\sigma^p,\tau\rangle, M)\simeq(\bZ/p\bZ)^{\oplus (p-1)}$ and $H^2(\langle\sigma^p,\sigma^i\tau\rangle, M)=0$ for $1\leq i\leq p-1$.

Remember that $M^{\langle\sigma^p\rangle}=0$.
Hence $H^2(\langle\sigma^p,\tau\rangle, M)\simeq H^1(\langle\sigma^p,\tau\rangle/\langle\sigma^p\rangle, H^1(\langle\sigma^p\rangle, M))$.

Note that $H^1(\langle\sigma^p\rangle, M)\simeq\oplus_{0\leq i\leq p-1}\bZ/p\bZ\cdot\overline{x_1^{(i)}}$ with $\tau\cdot\overline{x_1^{(i)}}=-\sum_{l\neq i}\overline{x_1^{(l)}}$, which is shown in the previous step.
Write $M_1=\oplus_{0\leq i\leq p-1}\bZ/p\bZ\cdot\overline{x_1^{(i)}}$.
We will prove that $H^1(\langle\sigma^p,\tau\rangle/\langle\sigma^p\rangle, M_1)\simeq (\bZ/p\bZ)^{\oplus (p-1)}$.

Define $w=\sum_{0\leq i\leq p-1}\overline{x_1^{(i)}}$.
Then $\tau\cdot\overline{x_1^{(i)}}=\overline{x_1^{(i)}}-w$.
Hence $(1+\tau+\cdots+\tau^{p-1})(\overline{x_1^{(i)}})=0$ for all $0\leq i\leq p-1$.
Thus $H^1(\langle\sigma^p,\tau\rangle/\langle\sigma^p\rangle, M_1)\simeq{\rm Ker}({\rm Norm})/{\rm Image}(\tau-1)=\langle\overline{x_1^{(i)}}:0\leq i\leq p-1\rangle/\langle w\rangle\simeq \oplus_{1\leq i\leq p-1}\bZ/p\bZ\cdot\overline{x_1^{(i)}}\simeq (\bZ/p\bZ)^{\oplus (p-1)}$.

We turn to $H^2(\langle\sigma^p,\sigma^i\tau\rangle, M)$ where $1\leq i\leq p-1$.
As before, $H^2(\langle\sigma^p,\sigma^i\tau\rangle, M)\simeq H^1(\langle\sigma^p,\sigma^i\tau\rangle/\langle\sigma^p\rangle, H^1(\langle\sigma^p\rangle,$ $M))\simeq H^1(\langle\sigma^p,\sigma^i\tau\rangle/\langle\sigma^p\rangle, M_1)$ where $M_1=\oplus_{0\leq i\leq p-1}\bZ/p\bZ\cdot\overline{x_1^{(i)}}$ and $\sigma\cdot\overline{x_1^{(i)}}=\overline{x_1^{(i+1)}}$, $\tau\cdot\overline{x_1^{(i)}}=\overline{x_1^{(i)}}-w$ with $w=\sum_{0\leq l\leq p-1}\overline{x_1^{(l)}}$ as before.
Then $\sigma^i\tau:\overline{x_1^{(l)}}\mapsto\overline{x_1^{(l+i)}}-w$.
Note that $(1+\sigma^i\tau+(\sigma^i\tau)^2+\cdots+(\sigma^i\tau)^{p-1})\cdot \overline{x_1^{(l)}}=w$.
Thus, if $v=\sum_{0\leq l\leq p-1}c_l\overline{x_1^{(l)}}\in M_1$,
then $(1+\sigma^i\tau+(\sigma^i\tau)^2+\cdots+(\sigma^i\tau)^{p-1})\cdot v=0$ if and only if $\sum_{0\leq l\leq p-1}c_l=0$.
Thus $H^1(\langle\sigma^p,\sigma^i\tau\rangle/\langle\sigma^p\rangle, M_1)\simeq{\rm Ker}({\rm Norm})/{\rm Image}(\sigma^i\tau-1)=\langle\overline{x_1^{(l)}}-\overline{x_1^{(0)}}:1\leq l\leq p-1\rangle/\langle\overline{x_1^{(l+i)}}-\overline{x_1^{(l)}}-w:0\leq l\leq p-1\rangle$.
In the next step, we will show that $\langle\overline{x_1^{(l)}}-\overline{x_1^{(0)}}:1\leq l\leq p-1\rangle=\langle\overline{x_1^{(l+i)}}-\overline{x_1^{(l)}}-w:0\leq l\leq p-1\rangle$.
Thus $H^1(\langle\sigma^p,\sigma^i\tau\rangle/\langle\sigma^p\rangle, M_1))=0$.

\medskip
Step 3.
Suppose that $V$ is a vector subspace of the vector space $\oplus_{0\leq l\leq p-1}\bZ/p\bZ\cdot\overline{x_1^{(l)}}$ over $\bZ/p\bZ$ defined as $V=\langle\overline{x_1^{(l)}}-\overline{x_1^{(0)}}:1\leq l\leq p-1\rangle$ and $W$ is a subspace of $V$ defined as $W=\langle\overline{x_1^{(l+i)}}-\overline{x_1^{(l)}}-w:0\leq l\leq p-1\rangle$ where $w=\sum_{0\leq l\leq p-1}\overline{x_1^{(l)}}$ and $i$ is an integer $1\leq i\leq p-1$ (${\rm mod} \, p$).
We will show that $V=W$.

For simplicity, we consider the case $i=1$.
The proof for the case $2\leq i\leq p-1$ is similar (by adjusting the indices, for example).

Write $y_l=\overline{x_1^{(l)}}-\overline{x_1^{(0)}}$ for $1\leq l\leq p-1$.
Then $\{y_1,\ldots,y_{p-1}\}$ is a basis of $V$ over $\bZ/p\bZ$.

Define $v_0=\overline{x_1^{(1)}}-\overline{x_1^{(0)}}-w$, $v_l=\overline{x_1^{(l+1)}}-\overline{x_1^{(l)}}-w$ (where $1\leq l\leq p-2$), $v_{p-1}=\overline{x_1^{(0)}}-\overline{x_1^{(p-1)}}-w$.
Then $W=\langle v_l:0\leq l\leq p-1\rangle$.
Note that $w=\sum_{0\leq l\leq p-1}\overline{x_1^{(l)}}=\sum_{1\leq l\leq p-1}(\overline{x^{(l)}}-\overline{x_0^{(0)}})=\sum_{1\leq l\leq p-1}y_l$.

Define $z_0=v_0=y_1-w=-y_2-y_3-\cdots -y_{p-1}$,
$z_1=-v_1+v_2=y_1-2y_2+y_3$,
$z_l=-v_l+v_{l+1}=y_l-2y_{l+1}+y_{l+2}$ (where $1\leq l\leq p-3$),
$z_{p-2}=-v_{p-2}+v_{p-1}=y_{p-2}-2y_{p-1}$.

Consider the quotient space $V/W$.
We will prove by induction that $y_k\equiv(p-k)y_{p-1}\ ({\rm modulo}\ W)$
for $1\leq k\leq p-1$.

The case $y_{p-1}\equiv y_{p-1}\ ({\rm modulo}\ W)$ is automatic.
The case $y_{p-2}\equiv 2y_{p-1}\ ({\rm modulo}\ W)$ follows from $y_{p-2}-2y_{p-1}=z_{p-2}\in W$.
Inductively, if $y_{k+l}\equiv (p-k-l)y_{p-1}\ ({\rm modulo}\ W)$ for $l=1$ and $2$, then we use the formula $z_k=y_k-2y_{k+1}+y_{k+2}\in W$.
Thus $y_k\equiv 2y_{k+1}-y_{k+2}\equiv (p-k)y_{p-1}\ ({\rm modulo}\ W)$.
Finally, we get $y_1\equiv (p-1)y_{p-1}\ ({\rm modulo}\ W)$.

Now we use the fact that $z_0=-y_2-y_3-\cdots-y_{p-1}\in W$.
We find that $-[(p-2)+(p-3)+\cdots+2+1]y_{p-1}\equiv 0\ ({\rm modulo}\ W)$,
i.e. $-((p-1)(p-2)/2)\cdot y_{p-1}\in W$.
Since $(p-1)(p-2)/2\neq 0$ in $\bZ/p\bZ$, we find that $y_{p-1}\in W$.
Once we know that $y_{p-1}\in W$, we know $y_l\in W$ for $1\leq l\leq p-2$ by
using the formulae $z_{p-2}, z_{p-3},\ldots, z_2, z_1$.
Done.

\medskip
Step 4.
We will show that the restriction map ${\rm res}:H^2(G,M)\rightarrow H^2(\langle\sigma^p,\sigma^i\tau\rangle, M)$ is the zero map where $0\leq i\leq p-1$.
Once we finish it, then we find that $H^2(G,M)=H_u^2(G,M)\simeq \bZ/p\bZ$
(by combining the result of Step 1).

When $1\leq i\leq p-1$, the map ${\rm res}:H^2(G,M)\rightarrow H^2(\langle\sigma^p,\sigma^i\tau\rangle, M)$ is the zero map because $H^2(\langle\sigma^p,\sigma^i\tau\rangle, M)=0$ by Step 2.

It remains to show that ${\rm res}:H^2(G,M)\rightarrow H^2(\langle\sigma^p,\tau\rangle, M)$ is the zero map also.
We will show that ${\rm Ker}({\rm res}:H^2(G,M)\rightarrow H^2(\langle\sigma^p,\tau\rangle, M))\simeq\bZ/p\bZ$.
Assume this.
Since $H^2(G,M)\simeq \bZ/p\bZ$ by Step 1, it follows that ${\rm res}:H^2(G,M)\rightarrow H^2(\langle\sigma^p,\tau\rangle, M)$ is the zero map.

We will use the exact sequence
\begin{align*}
\cdots\to
&\ H^2(G/\langle\sigma^p,\tau\rangle, M^{\langle\sigma^p,\tau\rangle})\to
{\rm Ker}({\rm res}:H^2(G, M)\to H^2(\langle\sigma^p,\tau\rangle, M))\\
\to &\ H^1(G/\langle\sigma^p,\tau\rangle, H^1(\langle\sigma^p,\tau\rangle, M))
\to H^3(G/\langle\sigma^p,\tau\rangle, M^{\langle\sigma^p,\tau\rangle}).
\end{align*}

Since $M^{\langle\sigma^p\rangle}=0$, it follows that ${\rm Ker}({\rm res}:H^2(G,M)\rightarrow H^2(\langle\sigma^p,\tau\rangle, M))\simeq H^1(G/\langle\sigma^p,\tau\rangle, H^1(\langle\sigma^p,\tau\rangle, M))$.

In Step 2, we have shown that $H^1(\langle\sigma^p\rangle,M)\simeq\oplus_{0\leq i\leq p-1}\bZ/p\bZ\cdot\overline{x_1^{(i)}}$ with $\sigma\cdot\overline{x_1^{(i)}}=\overline{x_1^{(i+1)}}$, $\tau\cdot\overline{x_1^{(i)}}=-\sum_{0\leq l\leq p-1\atop l\neq i}\overline{x_1^{(l)}}$.
Use the exact sequence
\begin{align*}
0\to H^1(\langle\sigma^p,\tau\rangle/\langle\sigma^p\rangle, M^{\langle\sigma^p\rangle})\to H^1(\langle\sigma^p,\tau\rangle, M)\to H^1(\langle\sigma^p\rangle,M)^{\langle\sigma^p,\tau\rangle/\langle\sigma^p\rangle}\to H^2(\langle\sigma^p,\tau\rangle/\langle\sigma^p\rangle, M^{\langle\sigma^p\rangle})\to \cdots.
\end{align*}
Since $M^{\langle\sigma^p\rangle}=0$,
it follows that $H^1(\langle\sigma^p,\tau\rangle, M)\xrightarrow{\sim} H^1(\langle\sigma^p\rangle, M)^{\langle\sigma^p,\tau\rangle/\langle\sigma^p\rangle}\simeq(\oplus_{0\leq i\leq p-1}\bZ/p\bZ\cdot\overline{x_1^{(i)}})^{\langle\tau\rangle}=\oplus_{1\leq \i\leq p-1}\bZ/p\bZ\cdot y_i$ where $y_i=\overline{x_1^{(i)}}-\overline{x_1^{(i-1)}}$ for $1\leq i\leq p-1$.

Let $\overline{\sigma}$ be the image of $\sigma$ in
$G/\langle\sigma^p,\tau\rangle$.
Note that $\overline{\sigma}:y_1\mapsto y_2\mapsto\cdots\mapsto y_{p-1}\mapsto -y_1-y_2-\cdots-y_{p-1}$.
Hence $(1+\overline{\sigma}+\overline{\sigma}^2+\cdots+\overline{\sigma}^{p-1})\cdot y_i=0$ for all $1\leq i\leq p-1$.
It follows that $H^1(G/\langle\sigma^p,\tau\rangle,H^1(\langle\sigma^p,\tau\rangle,M))\simeq H^1(\langle\overline{\sigma}\rangle,\oplus_{1\leq i\leq p-1}\bZ/p\bZ\cdot y_i)\simeq {\rm Ker}({\rm Norm})/{\rm Image}(\overline{\sigma}-1)=\langle y_i:1\leq i\leq p-1\rangle/\langle y_i-y_{i+1}:1\leq i\leq p-1\rangle\simeq\bZ/p\bZ\cdot\overline{y_1}\simeq\bZ/p\bZ$.
\end{proof}

Finally we will relate the lattices we consider with their flabby classes. Recall some notion of the theory of flabby (flasque) $G$-lattices developed by Endo and Miyata, Voskresnskii, Colliot-Th\'el\`ene and Sansuc \cite[Vo, CTS]{EM2}. The reader is referred to \cite{Sw} for a quick review of this theory.

\begin{defn} \label{d7.1}
Let $G$ be a finite group and $M$ be a $G$-lattice.
$M$ is called a {\it permutation lattice} if $M$ has a $\bm{Z}$-basis permuted by $G$.
A $G$-lattice $M$ is called an {\it invertible lattice} if it is a direct summand of some permutation lattice.
A $G$-lattice $M$ is called a {\it flabby lattice} if $H^{-1}(G',M)=0$ for all subgroup $G'$ of $G$;
it is called a {\it coflabby lattice} if $H^1(G',M)=0$ for all subgroups $G'$ of $G$.
For the basic properties of $G$-lattices, see \cite[Sw]{CTS}.
\end{defn}

\begin{defn} \label{d7.2}
Let $G$ be a finite group.
Two $G$-lattices $M_1$ and $M_2$ are called {\it similar}, denoted by $M_1\sim M_2$,
if $M_1\oplus Q_1\simeq M_2\oplus Q_2$ for some permutation lattices $Q_1$ and $Q_2$.
The flabby class monoid $F_G$ consists of all the similarity classes of flabby $G$-lattices under the addition described below.
Explicitly, if $M$ is a flabby $G$-lattice,
then $[M]\in F_G$ denotes the similarity class containing $M$;
the addition in $F_G$ is defined as: $[M_1]+[M_2]=[M_1\oplus M_2]$.
Note that $[M]=0$ in $F_G$,
i.e.\ $[M]$ is the zero element in $F_G$, if and only if $M\oplus Q$ is isomorphic to a permutation lattice where $Q$ is some permutation lattice.
See \cite{Sw} for details.
\end{defn}

\begin{defn} \label{d7.3}
Let $G$ be a finite group, $M$ be a $G$-lattice.
The $M$ have a flabby resolution,
i.e.\ there is an exact sequence of $G$-lattices:
$0\to M\to Q\to E\to 0$ where $Q$ is a permutation lattice and $E$ is a flabby lattice \cite[Lemma 1.1; CTS; Sw]{EM2}.

Although the above flabby resolution is not unique, the class $[E]\in F_G$ is uniquely determined by $M$.
Thus we define the flabby class of $M$, denoted as $[M]^{fl}$, by $[M]^{fl}=[E]\in F_G$ (see \cite{Sw}).
Sometimes we say that $[M]^{fl}$ is permutation or invertible if the class $[E]$ contains a permutation lattice or an invertible lattice.
\end{defn}

Recall the multiplicative fixed fields associated to algebraic tori. In Definition \ref{d1.2}, for any $G$-lattice $M$ and any field $k$, we define the field $k(M)$ on which $G$ acts multiplicatively; note that $G$ acts trivially on $k$. Suppose that we replace the field $k$ by another field $K$ such that $K$ is a finite Galois extension of a field $k$ with ${\rm Gal}(K/k)=G$. We may define the field $K(M)$ on which $G$ acts multiplicatively in a similar way; this time $G$ acts faithfully on $K$. In fact, the fixed field $K(M)^G$ is the function field of an algebraic torus which is defined over $k$,
split by $K$ and with character module $M$ (see \cite[Sw]{Vo}). 
A rationality criterion for algebraic tori is known:

\begin{theorem} \label{t7.4}
Let $K/k$ be a finite Galois extension with $G={\rm Gal}(K/k)$.
Let $M$ be a $G$-lattice.
\begin{enumerate}
\item[$(1)$] {\rm (\cite[Theorem 1.6; Vo; Le, Theorem 1.7]{EM1})}
The fixed field $K(M)^G$ is stably $k$-rational if and only if $[M]^{fl}=0$ in $F_G$.
\item[$(2)$] {\rm (\cite[Theorem 1.3]{Sa4})}
Assume that $k$ is an infinite field.
Then the fixed field $K(M)^G$ is retract $k$-rational if and only if $[M]^{fl}$ is invertible.
\end{enumerate}
\end{theorem}

Here is an analogous result for $\bm{C}(M)^G$
(be aware that $G$ acts trivially on $\bC$).

\begin{theorem} \label{t7.5}
Let $G$ be a finite group and $M$ be a faithful $G$-lattice.
\begin{enumerate}
\item[$(1)$] {\rm (\cite[Theorem 5.4]{Ka})}
Assume that $[M]^{fl}$ is invertible. Then $\bC(M)^G$ is retract $\bC$-rational if and only if $\bC(G)$ is retract $\bC$-rational.
\item[$(2)$] Assume that $[M]^{fl}=0$ and $\bC(G)$ is stably $\bC$-rational. Then $\bC(M)^G$ is also stably $\bC$-rational.
\end{enumerate}
\end{theorem}

\begin{proof}
For the proof of (2), consider the fixed field $\bC(M \oplus \bZ [G])^G$. It is not difficult to see that $\bC(M)^G$ is stably isomorphic to $\bC(\bZ [G])(M)^G$ (see the first paragraph of the proof in \cite[Proposition 2.2]{CHK}). Since $[M]^{fl}=0$, it follows that $\bC(\bZ [G])(M)^G$ is stably rational over $\bC(\bZ [G])^G=\bC(G)$ by Theorem \ref{t7.4}.
\end{proof}

\begin{remark}
(1) The $G$-lattices $M$ considered in {\rm Section} \ref{sePT} 
and Theorems \ref{thmD4n},
\ref{thmQDQ}, \ref{thmCp2p} are faithful lattices.  
Thus we may apply  Theorem \ref{t7.5} to them. 
As to the question whether $\bC(G)$ is retract $\bm{C}$-rational, 
the readers may consult \cite{Ka1} and \cite{Ka}. 
In particular, if $G$ is a finite group containing an abelian 
normal subgroup $N$ such that $G/N$ is a cyclic group, 
then $\bC(G)$ is retract $\bm{C}$-rational by \cite[Theorem 5.10]{Ka}. 
This result provides an alternative proof that $B_0(G)=0$ 
if $G$ is a dihedral group, a quasi-dihedral group, 
or a generalized quaternion group. \\
(2) In general, if $M$ is a $G$-lattice, define
$H= \{\tau \in G: \tau \, {\rm acts} \, {\rm trivially} \, {\rm on} \, M \}$.
Then  $[M]^{fl}$ is invertible when $M$ is regarded as a $G$-lattice if and only if so is $[M]^{fl}$ when $M$ is regarded as a $G/H$-lattice by \cite[page 180]{CTS}. But it is unnecessary that the retract rationality of $\bC(G)$ always implies that of $\bC(G/H)$ (see \cite[Theorem 3.1]{Sa1}). \\
(3) Again consider the general case where $M$ is any $G$-lattice. 
If it is assumed that all the Sylow subgroups of $G$ are cyclic groups, 
then $\bC(M)^G$ is retract $\bC$-rational by \cite[Theorem 6.6]{Ka}. 
Note that we do not assume that $[M]^{fl}$ is invertible, 
because it is a consequence of \cite[Theorem 1.5]{EM2}.
\end{remark}

\begin{corollary}\label{t7.6}
Let $M$ be a faithful $G$-lattice such that $\bC(G)$ is retract $\bC$-rational and $H^2_u (G, M) \neq 0$. Then $[M]^{fl}$ is not invertible. In particular, if $M$ is one of the $G$-lattices in Theorems \ref{thmD4n},
\ref{thmQDQ}, \ref{thmCp2p}, then $[M]^{fl}$ is not invertible.
\end{corollary}

\begin{proof}
If $[M]^{fl}$ is invertible, then we may apply Theorem \ref{t7.5}. Since $\bC(G)$ is retract $\bm{C}$-rational, so is $\bC(M)^G$. It follows that $H^2_u(G,M)=0$ by Definition \ref{d1.12}. This leads to a contradiction.
\end{proof}

\section{GAP computation: an algorithm to compute $H_u^2(G,M)$}\label{seGAP}

The following function {\tt H2nrM(g)} of GAP \cite{GAP} returns
$H_u^2(G,M)$ for $G$-lattice $M$ with $G\leq GL_n(\bm{Z})$
for ${\tt g}=G$.
The functions below, e.g. {\tt H2nrM(g)}, are available from\\
{\tt https://www.math.kyoto-u.ac.jp/\~{}yamasaki/Algorithm/MultInvField/res.gap}.\\

\begin{verbatim}
BlockList:= function(L,n)
    local l,m;
    l:=Length(L);
    m:=l/n;
    return List([1..n],x->L{[(x-1)*m+1..x*m]});
end;

IsBicyclic:= function(g)
    local f3,p;
    if not IsAbelian(g) then
        return false;
    fi;
    f3:=Filtered(Collected(Factors(Order(g))),x->x[2]>2);
    for p in f3 do
        if Length(AbelianInvariants(SylowSubgroup(g,p[1])))>2 then
            return false;
        fi;
    od;
    return true;
end;

FindGenFiniteAbelian:= function(g)
    local e,a,ga,iso;
    e:=AbelianInvariants(g);
    if Length(e)>1 then
        e:=SmithNormalFormIntegerMat(DiagonalMat(e));
        e:=List([1..Length(e)],x->e[x][x]);
        e:=Filtered(e,x->x>1);
    fi;
    a:=AbelianGroup(e);
    ga:=GeneratorsOfGroup(a);
    iso:=IsomorphismGroups(a,g);
    return List(ga,x->Image(iso,x));
end;

EltFiniteAbelian:= function(g,c)
    local gg,F,gF,hom,cF,e;
    gg:=GeneratorsOfGroup(g);
    F:=FreeGroup(Length(gg));
    gF:=GeneratorsOfGroup(F);
    hom:=GroupHomomorphismByImages(F,g,gF,gg);
    cF:=PreImagesRepresentative(hom,c);
    e:=List(gF,x->ExponentSumWord(cF,x));
    return e;
end;

CheckSNF:= function(m)
    local s,r;
    s:=SmithNormalFormIntegerMat(m);
    r:=Rank(s);
    if r=0 then
        return 0;
    else
        return s[r][r];
    fi;
end;

Z2value:= function(z,eg,gg)
    local og,zmat,i1,i2,i3,i,j2,j3,s,s1,s2;
    og:=Length(eg);
    zmat:=NullMat(og,og);
    for j2 in [1..Length(gg)] do
        i2:=Position(eg,gg[j2]);
        zmat{[1..og]}{[i2]}:=z{[1..og]}{[j2]};
    od;
    s:=gg;
    s1:=gg;
    s2:=[];
    repeat
        for j2 in s1 do
            i2:=Position(eg,j2);
            for j3 in gg do
                if not j2*j3 in Concatenation(s,s2) then
                    Add(s2,j2*j3);
                    i3:=Position(eg,j3);
                    i:=Position(eg,j2*j3);
                    for i1 in [1..og] do
                        zmat[i1][i]:=zmat[i1][i2]*j3
                            +zmat[Position(eg,eg[i1]*j2)][i3]-zmat[i2][i3];
                    od;
                fi;
            od;
        od;
        s1:=s2;
        s:=Union(s,s1);
        s2:=[];
    until s1=[];
    return zmat;
end;

H2nrM:= function(g)
    local d,gg,og,eg,h,gh,oh,eh,j1,j2,j,i1,i2,l0,l1,l2,m0,m1,m2,m,
        zero,sg,sh,h2g,h2gg,h2h,z,zmat,zmats,s,res,ress,
        Hg,Hgg,K,ga,iso,Hh,Hhg,hom,Kg;
    d:=Length(Identity(g));
    gg:=GeneratorsOfGroup(g);
    og:=Order(g);
    if gg=[] or og=1 then
        return rec(H2G:=[],H2Ggen:=[],H2nrM:=[],H2nrMgen:=[]);
    fi;
    eg:=Elements(g);
    l0:=[]; l1:=[]; l2:=[];
    j:=0;
    for j1 in eg do
        for j2 in gg do
            j:=j+1;
            Add(l0,[Position(eg,j2),j,Identity(g)]);
            Add(l1,[Position(eg,j1*j2),j,Identity(g)]);
            Add(l2,[Position(eg,j1),j,j2]);
        od;
    od;
    m0:=BlockMatrix(l0,og,og*Length(gg));
    m1:=BlockMatrix(l1,og,og*Length(gg));
    m2:=BlockMatrix(l2,og,og*Length(gg));
    m:=MatrixByBlockMatrix(m0-m1+m2);
    zero:=Flat(List([0..Order(gg[1])-2],
        x->List([1..d],y->(Position(eg,gg[1]^x)-1)*Length(gg)*d+y)));
    m:=NullspaceIntMat(m{[1..Length(m)]}{zero})*m;
    m:=LatticeBasis(m);
    z:=1;
    while m<>[] and z<=og*Length(gg) do
        zero:=[(z-1)*d+1..z*d];
        l0:=m{[1..Length(m)]}{zero};
        if CheckSNF(l0)=1 then
            m:=NullspaceIntMat(l0)*m;
        fi;
        z:=z+1;
    od;
    if m=[] then
        return rec(H2G:=[],H2Ggen:=[],H2nrM:=[],H2nrMgen:=[]);
    fi;
    sg:=SmithNormalFormIntegerMatTransforms(m);
    h2g:=Filtered(List([1..sg.rank],x->sg.normal[x][x]),y->y>1);
    m:=Inverse(sg.coltrans);
    h2gg:=List([sg.rank-Length(h2g)+1..sg.rank],x->m[x]);
    h2gg:=List(h2gg,x->BlockList(x,og));
    h2gg:=List(h2gg,x->List(x,y->BlockList(y,Length(gg))));
    zmats:=List(h2gg,x->Z2value(x,eg,gg));
    Hg:=AbelianGroup(h2g);
    Hgg:=GeneratorsOfGroup(Hg);
    K:=Hg;
    if ValueOption("fromperm")=true or ValueOption("FromPerm")=true then
        iso:=IsomorphismPermGroup(g);
        ga:=List(ConjugacyClassesSubgroups(Range(iso)),Representative);
        ga:=Filtered(ga,IsBicyclic);
        ga:=Concatenation(List(ga,x->ConjugateSubgroups(Image(iso),x)));
        ga:=List(ga,x->PreImage(iso,x));
    else
        ga:=List(ConjugacyClassesSubgroups(g),Representative);
        ga:=Filtered(ga,IsBicyclic);
        ga:=Concatenation(List(ga,x->ConjugateSubgroups(g,x)));
    fi;
    ga:=Filtered(ga,x->Order(x)>1);
    for h in ga do
        gh:=FindGenFiniteAbelian(h);
        oh:=Order(h);
        eh:=Elements(h);
        l0:=[]; l1:=[]; l2:=[];
        j:=0;
        for j1 in eh do
            for j2 in gh do
                j:=j+1;
                Add(l0,[Position(eh,j2),j,Identity(h)]);
                Add(l1,[Position(eh,j1*j2),j,Identity(h)]);
                Add(l2,[Position(eh,j1),j,j2]);
            od;
        od;
        m0:=BlockMatrix(l0,oh,oh*Length(gh));
        m1:=BlockMatrix(l1,oh,oh*Length(gh));
        m2:=BlockMatrix(l2,oh,oh*Length(gh));
        m:=MatrixByBlockMatrix(m0-m1+m2);
        sh:=SmithNormalFormIntegerMatTransforms(m);
        h2h:=Filtered(List([1..sh.rank],x->sh.normal[x][x]),y->y>1);
        if h2h<>[] then
            Hh:=AbelianGroup(h2h);
            Hhg:=GeneratorsOfGroup(Hh);
            ress:=[];
            for zmat in zmats do
                res:=List(eh,x->List(gh,
                    y->zmat[Position(eg,x)][Position(eg,y)]));
                res:=Flat(res)*sh.coltrans;
                res:=Product([1..Length(h2h)],
                    x->Hhg[x]^res[sh.rank-Length(h2h)+x]);
                Add(ress,res);
            od;
            hom:=GroupHomomorphismByImages(Hg,Hh,Hgg,ress);
            K:=Intersection(K,Kernel(hom));
            if Order(K)=1 then
                return rec(H2G:=h2g,H2Ggen:=h2gg,H2nrM:=[],H2nrMgen:=[]);
            fi;
        fi;
    od;
    Kg:=FindGenFiniteAbelian(K);
    return rec(H2G:=h2g,H2Ggen:=h2gg,H2nrM:=List(Kg,Order),
        H2nrMgen:=List(Kg,x->EltFiniteAbelian(Hg,x)));
end;
\end{verbatim}

\bigskip

\section{Tables: multiplicative invariant fields with non-trivial unramified Brauer groups}\label{seTables}

Table $1$: $M$ is indecomposable of rank $4$ ($5$ cases with $H_u^2(G,\bm{Q}/\bm{Z})=0$)\vspace*{2mm}\\
\begin{tabular}{lllll}
\hline
$G(n,i)$ & $G$ & GAP ID & $H_u^2(G,M)$\\\hline
$(8,3)$ & $D_4$ & $(4,12,4,12)$ & $\bZ/2\bZ$\\
$(8,4)$ & $Q_8$ & $(4,32,1,2)$ & $(\bZ/2\bZ)^{\oplus 2}$\\
$(16,8)$ & $QD_8$ & $(4,32,3,2)$ & $\bZ/2\bZ$\\
$(24,3)$ & $SL_2(\bm{F}_3)$ & $(4,33,3,1)$ & $(\bZ/2\bZ)^{\oplus 2}$\\
$(48,29)$ & $GL_2(\bm{F}_3)$ & $(4,33,6,1)$ & $\bZ/2\bZ$\\
\end{tabular}\\

Table $2$-$1$: $M$ is indecomposable of rank $5$ ($27$ cases with $H_u^2(G,\bm{Q}/\bm{Z})=0$)\vspace*{2mm}\\
\begin{tabular}{lllll}
\hline
$G(n,i)$ & $G$ & CARAT ID & $H_u^2(G,M)$\\\hline
$(8,3)$ & $D_4$ & $(5,39,5)$ & $\bm{Z}/2\bm{Z}$\\
$(8,3)$ & $D_4$ & $(5,99,23)$ & $\bm{Z}/2\bm{Z}$\\
$(8,3)$ & $D_4$ & $(5,99,24)$ & $\bm{Z}/2\bm{Z}$\\
$(8,3)$ & $D_4$ & $(5,99,25)$ & $\bm{Z}/2\bm{Z}$\\
$(8,3)$ & $D_4$ & $(5,100,12)$ & $\bm{Z}/2\bm{Z}$\\
$(8,3)$ & $D_4$ & $(5,100,23)$ & $\bm{Z}/2\bm{Z}$\\
$(8,4)$ & $Q_8$ & $(5,773,4)$ & $(\bZ/2\bZ)^{\oplus 2}$\\
$(8,4)$ & $Q_8$ & $(5,774,4)$ & $\bm{Z}/2\bm{Z}$\\
$(16,3)$ & $(C_4 \times C_2) \rtimes C_2$ & $(5,109,5)$ & $\bm{Z}/2\bm{Z}$\\
$(16,3)$ & $(C_4 \times C_2) \rtimes C_2$ & $(5,116,20)$ & $\bm{Z}/2\bm{Z}$\\
$(16,3)$ & $(C_4 \times C_2) \rtimes C_2$ & $(5,118,18)$ & $\bm{Z}/2\bm{Z}$\\
$(16,4)$ & $C_4 \rtimes C_4$ & $(5,105,5)$ & $\bm{Z}/2\bm{Z}$\\
$(16,8)$ & $QD_8$ & $(5,672,2)$ & $\bm{Z}/2\bm{Z}$\\
$(16,8)$ & $QD_8$ & $(5,673,2)$ & $\bm{Z}/2\bm{Z}$\\
$(16,11)$ & $D_4\times C_2$ & $(5,71,19)$ & $\bm{Z}/2\bm{Z}$\\
$(16,11)$ & $D_4\times C_2$ & $(5,73,37)$ & $\bm{Z}/2\bm{Z}$\\
$(16,11)$ & $D_4\times C_2$ & $(5,76,49)$ & $\bm{Z}/2\bm{Z}$\\
$(16,11)$ & $D_4\times C_2$ & $(5,76,50)$ & $\bm{Z}/2\bm{Z}$\\
$(16,11)$ & $D_4\times C_2$ & $(5,76,51)$ & $\bm{Z}/2\bm{Z}$\\
$(16,11)$ & $D_4\times C_2$ & $(5,79,18)$ & $\bm{Z}/2\bm{Z}$\\
$(16,11)$ & $D_4\times C_2$ & $(5,119,4)$ & $\bm{Z}/2\bm{Z}$\\
$(16,12)$ & $Q_8\times C_2$ & $(5,664,2)$ & $(\bZ/2\bZ)^{\oplus 2}$\\
$(32,6)$ & $((C_4\times C_2) \rtimes C_2) \rtimes C_2$ & $(5,704,3)$ & $\bm{Z}/2\bm{Z}$\\
$(32,6)$ & $((C_4\times C_2) \rtimes C_2) \rtimes C_2$ & $(5,706,8)$ & $\bm{Z}/2\bm{Z}$\\
$(32,27)$ & $(C_2)^4 \rtimes C_2$ & $(5,142,14)$ & $\bm{Z}/2\bm{Z}$\\
$(32,28)$ & $(C_4 \times C_2\times C_2) \rtimes C_2$ & $(5,140,23)$ & $\bm{Z}/2\bm{Z}$\\
$(32,40)$ & $QD_8\times C_2$ & $(5,721,2)$ & $\bm{Z}/2\bm{Z}$\\\hline
\end{tabular}\\

Table $2$-$2$: $M=M_1\oplus M_2$ is decomposable of rank $5=4+1$
and $M_1$ is in Table $1$ ($18$ cases with $H_u^2(G,\bm{Q}/\bm{Z})=0$)\vspace*{2mm}\\
\begin{tabular}{lllllll}
\hline
$G(n,i)$ & $G$ & CARAT ID & $M_1\oplus M_2$\quad $(G_i=G|_{M_i})$& $[|G_1|,|G_2|]$ & $H_u^2(G,M)$\\\hline
$(8,3)$ & $D_4$ & $(5,63,12)$ & $(4,12,4,12)\oplus\bm{Z}^-$ & $[8,2]$ & $\bm{Z}/2\bm{Z}$\\
$(8,3)$ & $D_4$ & $(5,65,12)$ & $(4,12,4,12)\oplus\bm{Z}^-$ & $[8,2]$ & $\bm{Z}/2\bm{Z}$\\
$(8,3)$ & $D_4$ & $(5,99,5)$ & $(4,12,4,12)\oplus\bm{Z}$ & $[8,1]$ & $\bm{Z}/2\bm{Z}$\\
$(8,3)$ & $D_4$ & $(5,100,5)$ & $(4,12,4,12)\oplus\bm{Z}^-$ & $[8,2]$ & $\bm{Z}/2\bm{Z}$\\
$(8,4)$ & $Q_8$ & $(5,773,3)$ & $(4,32,1,2)\oplus\bm{Z}$ & $[8,1]$ & $(\bZ/2\bZ)^{\oplus 2}$\\
$(8,4)$ & $Q_8$ & $(5,774,3)$ & $(4,32,1,2)\oplus\bm{Z}^-$ & $[8,2]$ & $(\bZ/2\bZ)^{\oplus 2}$\\
$(16,8)$ & $QD_8$ & $(5,672,1)$ & $(4,32,3,2)\oplus\bm{Z}$ & $[8,1]$ & $\bm{Z}/2\bm{Z}$\\
$(16,8)$ & $QD_8$ & $(5,673,1)$ & $(4,32,3,2)\oplus\bm{Z}^-$ & $[16,2]$ & $\bm{Z}/2\bm{Z}$\\
$(16,8)$ & $QD_8$ & $(5,674,1)$ & $(4,32,3,2)\oplus\bm{Z}^-$ & $[16,2]$ & $\bm{Z}/2\bm{Z}$\\
$(16,8)$ & $QD_8$ & $(5,675,1)$ & $(4,32,3,2)\oplus\bm{Z}^-$ & $[16,2]$ & $\bm{Z}/2\bm{Z}$\\
$(16,11)$ & $D_4\times C_2$ & $(5,76,31)$ & $(4,12,4,12)\times\{\pm 1\}$
& $[8,2]$ & $\bm{Z}/2\bm{Z}$\\
$(16,12)$ & $Q_8\times C_2$ & $(5,664,1)$ & $(4,32,1,2)\times\{\pm 1\}$
& $[8,2]$ &$(\bZ/2\bZ)^{\oplus 2}$\\
$(24,3)$ & $SL_2(\bm{F}_3)$ & $(5,691,1)$ & $(4,33,3,1)\oplus\bm{Z}$ & $[24,1]$
& $(\bZ/2\bZ)^{\oplus 2}$\\
$(32,40)$ & $QD_8\times C_2$ & $(5,721,1)$ & $(4,32,3,2)\times\{\pm 1\}$
& $[16,2]$ & $\bm{Z}/2\bm{Z}$\\
$(48,29)$ & $GL_2(\bm{F}_3)$ & $(5,733,1)$ & $(4,33,6,1)\oplus\bm{Z}^-$ & $[48,2]$ & $\bm{Z}/2\bm{Z}$\\
$(48,29)$ & $GL_2(\bm{F}_3)$ & $(5,734,1)$ & $(4,33,6,1)\oplus\bm{Z}$ & $[48,1]$ & $\bm{Z}/2\bm{Z}$\\
$(48,32)$ & $SL_2(\bm{F}_3)\times C_2$ & $(5,730,1)$ & $(4,33,3,1)\times\{\pm 1\}$ & $[24,2]$ & $(\bZ/2\bZ)^{\oplus 2}$\\
$(96,189)$ & $GL_2(\bm{F}_3)\times C_2$ & $(5,776,1)$ & $(4,33,6,1)\times\{\pm 1\}$ & $[48,2]$ & $\bm{Z}/2\bm{Z}$\\\hline
\end{tabular}\\

Table $2$-$3$: $M=M_1\oplus M_2$ is decomposable of rank $5=3+2$ ($1$ case
with $H_u^2(G,\bm{Q}/\bm{Z})=0$)\vspace*{2mm}\\
\begin{tabular}{lllllll}
\hline
$G(n,i)$ & $G$ & CARAT ID & $M_1\oplus M_2$\quad $(G_i=G|_{M_i})$& $[|G_1|,|G_2|]$ & $H_u^2(G,M)$\\\hline
$(8,3)$ & $D_4$ & $(5,100,11)$ & $(3,3,1,3)\oplus (2,3,2,1)$ & $[4,8]$ & $\bm{Z}/2\bm{Z}$\\\hline
\end{tabular}

\newpage
Table $3$-$1$-$1$: $M$ is indecomposable of rank $6$ ($601$ cases with $H_u^2(G,\bm{Q}/\bm{Z})=0$)
\vspace*{-5mm}\\

Table $3$-$1$-$2$: $M$ is indecomposable of rank $6$ ($12$ cases with $H_u^2(G,\bm{Q}/\bm{Z})\neq 0$)
\vspace*{-5mm}\\
\begin{longtable}{llllllll}
\hline
$G(n,i)$ & $G$ & CARAT ID & $H_u^2(G,\bm{Q}/\bm{Z})$ & $H_u^2(G,M)$\\\hline
$(64,149)$ & $(QD_8 \times C_2) \rtimes C_2$ & $(6,6459,2)$ & $\bm{Z}/2\bm{Z}$ & $0$\\
$(64,149)$ & $(QD_8 \times C_2) \rtimes C_2$ & $(6,6459,4)$ & $\bm{Z}/2\bm{Z}$ & $0$\\
$(64,149)$ & $(QD_8 \times C_2) \rtimes C_2$ & $(6,6459,6)$ & $\bm{Z}/2\bm{Z}$ & $0$\\
$(64,149)$ & $(QD_8 \times C_2) \rtimes C_2$ & $(6,6459,8)$ & $\bm{Z}/2\bm{Z}$ & $\bm{Z}/2\bm{Z}$\\
$(64,150)$ & $(D_8 \times C_2) \rtimes C_2$ & $(6,6458,2)$ & $\bm{Z}/2\bm{Z}$ & $0$\\
$(64,150)$ & $(D_8 \times C_2) \rtimes C_2$ & $(6,6458,4)$ & $\bm{Z}/2\bm{Z}$ & $0$\\
$(64,150)$ & $(D_8 \times C_2) \rtimes C_2$ & $(6,6458,6)$ & $\bm{Z}/2\bm{Z}$ & $0$\\
$(64,150)$ & $(D_8 \times C_2) \rtimes C_2$ & $(6,6458,8)$ & $\bm{Z}/2\bm{Z}$ & $0$\\
$(64,177)$ & $(QD_8 \times C_2) \rtimes C_2$ & $(6,6464,2)$ & $\bm{Z}/2\bm{Z}$ & $0$\\
$(64,177)$ & $(QD_8 \times C_2) \rtimes C_2$ & $(6,6464,4)$ & $\bm{Z}/2\bm{Z}$ & $0$\\
$(64,177)$ & $(QD_8 \times C_2) \rtimes C_2$ & $(6,6464,6)$ & $\bm{Z}/2\bm{Z}$ & $\bm{Z}/2\bm{Z}$\\
$(64,177)$ & $(QD_8 \times C_2) \rtimes C_2$ & $(6,6464,8)$ & $\bm{Z}/2\bm{Z}$ & $0$\\
\end{longtable}

Table $3$-$2$: $M=M_1\oplus M_2$ is decomposable of rank $6=5+1$
and $M_1$ is in Table $2$-$1$ ($161$ cases with $H_u^2(G,\bm{Q}/\bm{Z})=0$)\vspace*{-2mm}\\
\end{tabular}

Table $3$-$3$-$2$: $M=M_1\oplus M_2$ is decomposable of rank $6=4+2$
and $M_1$ is in Table $1$ ($191$ cases with $H_u^2(G,\bm{Q}/\bm{Z})=0$)\
\vspace*{-2mm}\\

Table $3$-$3$-$3$: $M=M_1\oplus M_2$ is decomposable of rank $6=4+2$ ($12$ cases with $H_u^2(G,\bm{Q}/\bm{Z})\neq 0$)\vspace*{-2mm}\\
\begin{longtable}{lllllll}
\hline
$G(n,i)$ & $G$ & CARAT ID & $M_1\oplus M_2$\quad $(G_i=G|_{M_i})$& $[|G_1|,|G_2|]$ & $H_u^2(G,\bm{Q}/\bm{Z})$ & $H_u^2(G,M)$\\\hline
$(64,149)$ & $(QD_8 \times C_2) \rtimes C_2$ & $(6,6459,1)$ & $(4,32,6,1)\oplus(2,3,2,1)$ & $[32,8]$ & $\bm{Z}/2\bm{Z}$ & $0$\\
$(64,149)$ & $(QD_8 \times C_2) \rtimes C_2$ & $(6,6459,3)$ & $(4,32,6,1)\oplus(2,3,2,1)$ & $[32,8]$ & $\bm{Z}/2\bm{Z}$ & $0$\\
$(64,149)$ & $(QD_8 \times C_2) \rtimes C_2$ & $(6,6459,5)$ & $(4,32,6,2)\oplus(2,3,2,1)$ & $[32,8]$ & $\bm{Z}/2\bm{Z}$ & $0$\\
$(64,149)$ & $(QD_8 \times C_2) \rtimes C_2$ & $(6,6459,7)$ & $(4,32,6,2)\oplus(2,3,2,1)$ & $[32,8]$ & $\bm{Z}/2\bm{Z}$ & $0$\\
$(64,150)$ & $(D_8 \times C_2) \rtimes C_2$ & $(6,6458,1)$ & $(4,32,6,1)\oplus(2,3,2,1)$ & $[32,8]$ & $\bm{Z}/2\bm{Z}$ & $0$\\
$(64,150)$ & $(D_8 \times C_2) \rtimes C_2$ & $(6,6458,3)$ & $(4,32,6,1)\oplus(2,3,2,1)$ & $[32,8]$ & $\bm{Z}/2\bm{Z}$ & $0$\\
$(64,150)$ & $(D_8 \times C_2) \rtimes C_2$ & $(6,6458,5)$ & $(4,32,6,2)\oplus(2,3,2,1)$ & $[32,8]$ & $\bm{Z}/2\bm{Z}$ & $0$\\
$(64,150)$ & $(D_8 \times C_2) \rtimes C_2$ & $(6,6458,7)$ & $(4,32,6,2)\oplus(2,3,2,1)$ & $[32,8]$ & $\bm{Z}/2\bm{Z}$ & $0$\\
$(64,177)$ & $(QD_8 \times C_2) \rtimes C_2$ & $(6,6464,1)$ & $(4,32,6,1)\oplus(2,3,2,1)$ & $[32,8]$ & $\bm{Z}/2\bm{Z}$ & $0$\\
$(64,177)$ & $(QD_8 \times C_2) \rtimes C_2$ & $(6,6464,3)$ & $(4,32,6,1)\oplus(2,3,2,1)$ & $[32,8]$ & $\bm{Z}/2\bm{Z}$ & $0$\\
$(64,177)$ & $(QD_8 \times C_2) \rtimes C_2$ & $(6,6464,5)$ & $(4,32,6,2)\oplus(2,3,2,1)$ & $[32,8]$ & $\bm{Z}/2\bm{Z}$ & $0$\\
$(64,177)$ & $(QD_8 \times C_2) \rtimes C_2$ & $(6,6464,7)$ & $(4,32,6,2)\oplus(2,3,2,1)$ & $[32,8]$ & $\bm{Z}/2\bm{Z}$ & $0$\\
\end{longtable}

Table $3$-$4$: $M=M_1\oplus M_2$ is decomposable of rank $6=4+1+1$
and $M_1$ is in Table $1$ ($59$ cases with $H_u^2(G,\bm{Q}/\bm{Z})=0$)\
\vspace*{-2mm}\\

Table $3$-$5$: $M=M_1\oplus M_2$ is decomposable of rank $6=3+3$ 
($5$ cases with $H_u^2(G,\bm{Q}/\bm{Z})=0$)\
\vspace*{-2mm}\\

Table $3$-$6$: $M=M_1\oplus M_2$ is decomposable of rank $6=3+2+1$
and $M_1$ is in Table $2$-$3$ ($5$ cases with $H_u^2(G,\bm{Q}/\bm{Z})=0$)\
\vspace*{-2mm}\\
\begin{longtable}{lllllll}
\hline
$G(n,i)$ & $G$ & CARAT ID & $M_1\oplus M_2$\quad $(G_i=G|_{M_i})$& $[|G_1|,|G_2|]$ & $H_u^2(G,M)$\\\hline
$(8,3)$ & $D_4$ & $(6,4809,11)$ & $(5,100,11)\oplus\bm{Z}^-$ & $[8,2]$ & $\bm{Z}/2\bm{Z}$\\
$(8,3)$ & $D_4$ & $(6,4809,52)$ & $(5,100,11)\oplus\bm{Z}^-$ & $[8,2]$& $\bm{Z}/2\bm{Z}$\\
$(8,3)$ & $D_4$ & $(6,4813,11)$ & $(5,100,11)\oplus\bm{Z}^-$ & $[8,2]$& $\bm{Z}/2\bm{Z}$\\
$(8,3)$ & $D_4$ & $(6,4892,22)$ & $(5,100,11)\oplus\bm{Z}$ & $[8,1]$ & $\bm{Z}/2\bm{Z}$\\
$(16,11)$ & $D_4 \times C_2$ & $(6,4828,11)$ & $(5,100,11)\times\{\pm 1\}$ & $[8,2]$ & $\bm{Z}/2\bm{Z}$\\
\hline
\end{longtable}




\end{document}